\documentclass[12pt]{article} 

\usepackage[width=180mm,top=25mm,bottom=25mm]{geometry} 

\usepackage[utf8]{inputenc}  

\usepackage{amssymb} 
\usepackage{mathptmx} 
\usepackage{mathtools} 
\usepackage{hyperref}
\usepackage[ocgcolorlinks]{ocgx2}
\usepackage{amsthm} 
\usepackage[noabbrev,nosort]{cleveref}
\crefname{ineq}{inequality}{inequalities}
\creflabelformat{ineq}{(#2{\upshape#1}#3)}
\usepackage{stmaryrd} 
\usepackage{amsmath}
\usepackage{todonotes}
\usepackage{comment}
\usepackage{color}
\usepackage{enumitem}
\usepackage{pgfplots} 
\usepackage{stmaryrd}
\usepackage[normalem]{ulem}
\usepackage{soul}
\usepackage{mathrsfs}
\usepackage{bbold}
\usepackage{graphicx,stackengine,scalerel}

\newtheorem{prob}{Problem}[section] 
\newtheorem{lem}[prob]{Lemma} 
\newtheorem{thm}[prob]{Theorem} 
\newtheorem{defin}[prob]{Definition}

\newtheorem{prop}[prob]{Proposition}
\newtheorem{rem}[prob]{Remark}
\newtheorem{assump}[prob]{Assumption}
\newtheorem{coro}[prob]{Corollary}
\newtheorem{examp}[prob]{Example}
\newtheorem{notat}[prob]{Notation}

\crefname{prob}{problem}{problems}
\crefname{lem}{lemma}{lemmas} 
\crefname{thm}{theorem}{theorems} 
\crefname{defin}{definition}{definitions}
\crefname{memo}{memo}{memos}
\crefname{prop}{proposition}{propositions}
\crefname{rem}{remark}{remarks}
\crefname{coro}{corollary}{corollaries}
\crefname{examp}{example}{examples}
\crefname{notat}{notation}{notations}
\crefname{assump}{assumption}{assumptions}

\newcommand*\samethanks[1][\value{footnote}]{\footnotemark[#1]}
\newcommand{\R}{\mathbb{R}}

\newcommand{\T}{\mathcal{T}_1(\R^n)}
\newcommand{\Ha}{\mathcal{H}^1}
\newcommand{\Hak}{\mathcal{H}^k}
\newcommand{\N}{\mathcal{N}}

\newcommand{\ws}{\stackrel{*}{\rightharpoonup}}
\newcommand{\mres}{\mathbin{\vrule height 1.3ex depth 0pt width 0.13ex\vrule height 0.13ex depth 0pt width 1ex}} 

\newcommand{\cyan}[1]{\textcolor{black}{#1}}

\newcommand{\dom}{\textup{dom}}
\newcommand{\Jm}{\mathrm{D}}
\newcommand{\Jac}{\mathrm{J}}
\newcommand{\E}{\mathcal{E}}
\newcommand{\G}{\mathcal{G}}
\newcommand{\adm}{\mathcal{A}}
\newcommand{\V}{\mathcal{V}}
\newcommand{\M}{\mathcal{M}}
\newcommand{\Mb}{\mathbb{M}}
\newcommand{\B}{\mathcal{B}}
\newcommand{\D}{\mathcal{D}}
\newcommand{\rct}{\mathcal{R}}
\newcommand{\e}{\cyan{\mathrm{d}}}

\newcommand{\rectn}{\mathcal{R}_1(\mathbb{R}^n)}
\newcommand{\apD}{\mathrm{ap\,D}}
\newcommand{\apJ}{\mathrm{ap\,J}}
\newcommand{\spt}{\textup{supp}}

\newcommand{\Leb}{\mathcal{L}}
\newcommand{\Lebn}{\mathcal{L}^n}
\newcommand{\clos}{\textup{clos}}

\newcommand{\Dt}{\Delta t}

\makeatletter
\newsavebox{\@brx}
\newcommand{\llangle}[1][]{\savebox{\@brx}{\(\m@th{#1\langle}\)}%
  \mathopen{\copy\@brx\kern-0.5\wd\@brx\usebox{\@brx}}}
\newcommand{\rrangle}[1][]{\savebox{\@brx}{\(\m@th{#1\rangle}\)}%
  \mathclose{\copy\@brx\kern-0.5\wd\@brx\usebox{\@brx}}}
\makeatother

\usepackage[backend=biber,maxnames=5,style=numeric]{biblatex}
\title{Geometric flows of branched transportation networks}
\author{Julius Lohmann\thanks{Department of Mathematics, Institute of Science Tokyo}\thanks{International Research Fellow of Japan Society for the Promotion of Science, lohmann.j.aa@m.titech.ac.jp}\and Yoshihiro Tonegawa\samethanks[1]\thanks{tonegawa.y.ab@m.titech.ac.jp}}

\begin{document}
\setstcolor{cyan}
\maketitle
\noindent

\begin{abstract}
The branched transport problem is a nonconvex and nonsmooth variational optimization problem on normal $1$-currents in $\R^n$ with prescribed boundary.
The optimality is with respect to some non-decreasing, lower semicontinuous, and subadditive function $\tau:\R_+\to\R_+$ with $\tau(0)=0$ describing the cost $\tau(m)$ to move an amount of mass $m$ per unit distance.
The subadditivity leads to complicated, hierarchically ramified patterns in the support of (suboptimal) solutions.
These network-like sets appear to have regularity properties similar to those of the singular surfaces arising in the so-called Brakke flow, a weak generalization of the mean curvature flow.
We construct a geometric flow of transportation networks, which correspond to normal real $1$-rectifiable currents, by modifying Brakke's variational approximation scheme.
We prove the existence of a limit that is Hölder continuous with respect to the flat norm and whose branched transport cost decreases along the geometric evolution.
We further analyze a closely related geometric flow approximated by $1$-varifolds, whose weight measures model the branched transport cost, and establish its $1$-rectifiability together with a motion law analogous to Brakke’s inequality.
\end{abstract}
\tableofcontents

\section{Introduction}
\label{Intro}
The \textit{branched transport problem} was introduced by Maddalena, Solimini, and Morel \cite{MSM03} (Lagrangian formulation) as well as Xia \cite{X} (Eulerian formulation).
Let $\mu_-,\mu_+$ be two compactly supported Radon measures on $\R^n$ with $0\leq \mu_-(\R^n)=\mu_+(\R^n)<\infty$.
In the Eulerian formulation, the branched transport problem seeks an optimal \textit{mass flux} $T$ from the source $\mu_-$ to the sink $\mu_+$, which is a normal $1$-current with $\partial T=\mu_+-\mu_-$.
The optimality is with respect to a \textit{transportation cost} $\tau:\R_+\to\R_+$, where $\tau(m)$ represents the cost to move an amount of mass $m$ per unit distance.
Originally, the choice $\tau(m)=m^\alpha$ with $\alpha\in(0,1)$ was considered. 
It was generalized to arbitrary non-decreasing, lower semicontinuous, and subadditive transportation costs $\tau$ with $\tau(0)=0$ by Brancolini and Wirth \cite{BW}.
The subadditivity
\begin{equation*}
\tau(m_1+m_2)\leq\tau(m_1)+\tau(m_2)
\end{equation*}
leads to complex ramified transportation networks between $\mu_-$ and $\mu_+$.
Here we call a \textit{transportation network} any normal real rectifiable $1$-current $T$ with $\partial T=\mu_+-\mu_-$.
(Note that, in general, diffuse parts may also occur and transportation networks alone are not sufficient to describe the dynamics, see \cite[Prop.\ 2.32]{BW}.)
Any transporation network $T$ can be written $T=\theta\Ha\mres S$, where $\Ha\mres S$ denotes the restriction of the one-dimensional Hausdorff measure to some countably $\Ha$-rectifiable set $S$ and $\theta\in L^1_{loc}(\Ha\mres S;\R^n)$ is tangential to $S$ with $\theta\neq 0$ $\Ha$-a.e.\ on $ S$. Then $m(x)=|\theta(x)|$ represents the amount of mass transported through the point $x$ in direction $\theta(x)/|\theta(x)|$ for $\Ha$-a.e.\ $x\in S$, and the total branched transport cost is given by $|T|_\tau(\R^n)=\int_S\tau(|\theta|)\e\Ha$.

Transportation networks serve as models for pattern formation in mathematical biology, including anatomical networks such as blood vessels and the bronchial system, as well as botanical structures such as leaf venation and the crown and root networks of trees.
Furthermore, it has been shown \cite{BW16,LSW} that the branched transport problem with a concave transportation cost is equivalent to the generalized \textit{urban planning problem}.
The classical urban planning problem was proposed by Brancolini and Buttazzo \cite{BB} and by Buttazzo, Pratelli, Solimini, and Stepanov \cite{BPSS}, and later generalized by Schmitzer, Wirth, and the first author in \cite{LSW}.
This reformulation gives branched transport a distinctly microeconomic interpretation by separating a Wasserstein-type transportation cost induced by the generalized \textit{urban metric} from a maintenance cost term.
A good starting point into the literature on branched transport is provided in the textbooks by Bernot, Caselles, and Morel \cite{BCM} as well as Buttazzo, Pratelli, Solimini, and Stepanov \cite{BPSS}.

Over the past two decades, branched transport has been the subject of extensive research.
Many approaches have been proposed to approximate or reformulate the nonsmooth and nonconvex branched transport problem.
The most common strategy for smooth approximation draws inspiration from results on the \textit{elliptic approximation} of free discontinuity problems, such as those of Modica and Mortola \cite{MM77}, Aviles and Giga \cite{AG87}, and Ambrosio and Tortorelli \cite{AT90}.
It is based on $\Gamma$-convergence, introduced by De Giorgi, which provides a rigorous framework for approximating singular energies by smoother elliptic functionals in such a way that the associated minimizers converge as well.
To the best of our knowledge, the first attempt in the context of branched transport was made by Oudet and Santambrogio \cite{OS11}.
They defined a family of functionals on $W^{1,2}$ vector fields, indexed by some smoothing parameter $\varepsilon>0$, and replaced the double-well potential appearing in \cite{MM77} by a concave term.
Then the authors proved a $\Gamma$-convergence result for $n=2$ and $\tau(m)=m^\alpha$ with $\alpha\in(1/2,1)$.
A succesful approach related to convex reformulation was given by Brancolini, Rossmanith, and Wirth \cite{BRW18}.
They showed that, provided $n=2$, the branched transport problem with $\tau(m)=m^\alpha$ or $\tau(m)=\min\{ am,bm+c\}$, where $\alpha\in(0,1)$ and $a,b,c\in\R_+^*$ with $a>b$, can be reformulated as a Mumford\textendash Shah segmentation problem from image processing.
Subsequently, they used the idea of convex lifting \cite{PCCB09} to obtain a convex approximation from below.
Finally, we mention the idea of convexification via so-called \textit{multi-material transport}.
To the best of our knowledge, the first attempts in this context were made by Marchese and Massaccesi \cite{MM16,MM16b}.
The basic idea is to introduce $m$ artificial material types, rather than a single material type associated with the measures $\mu_-$ and $\mu_+$. 
This leads to a convex optimization problem on (a subspace of) the class of normal $1$-currents in $\R^n$ with coefficients in $\R^m$.
The actual multi-material transport problem was introduced by Marchese, Massaccesi, and Tione \cite{MMT17} and later generalized by Marchese, Massaccesi, Stuvard, and Tione \cite{MMSR16}.
Since the multi-material transport problem is a Plateau problem, one can then look for calibrations to certify the optimality of candidate minimizers.
What all the detailed approaches have in common is that they provide smooth or convex approximations, and in some cases tight ones, of specific classes of branched transport problems.

To the best of our knowledge, a rigorous mathematical model for the time-dependent (self-)optimization of (branched) transportation networks has not yet been investigated\footnote{Besides some ad hoc numerical approaches, Xia \cite{X10}, for example, proposed an algorithm that gradually optimizes the trivial normal $1$-current connecting a source with finitely many sinks, essentially by alternating local optimization and subdivision steps.}.
Such a model could support the dynamic redesign of urban transportation networks and help explain the adaptability of biological networks.
Moreover, studying geometric flows of transportation networks is mathematically interesting in its own right and may contribute to solving the nonconvex and nonsmooth branched transport problem, see below.
We propose a variant of the so-called \textit{Brakke flow} as a possible model for the (self-)optimization of transportation networks.
This geometric flow has been introduced by Brakke \cite{B78} and was further refined by Kim and the second author \cite{KT17}.
It is a generalization of the well-known \textit{mean curvature flow} and can be used as a model for the motion of grain boundaries in an annealing pure metal. 
In particular, such surfaces can have singularities and undergo topological changes over time.
In \cite{B78,KT17} the central objects to model such surfaces are \textit{varifolds}. 
A $k$-varifold is simply a nonnegative Radon measures on $\R^n\times G(n,k)$, where $G(n,k)$ denotes the Grassmannian.
In the original construction, Brakke derived a variational approximation scheme for the geometric flow. 
However, this scheme does not prevent the geometric flow---which is obtained, in essence, by letting the temporal resolution tend to zero---from vanishing instantaneously at $t=0$.
Therefore, the authors of \cite{KT17} introduced a modification by viewing the grain boundaries as topological boundaries---interpreting their union as a unit density $(n-1)$-varifold---of finitely many open subsets in $\R^n$ and by controlling the Lebesgue measure of the symmetric differences in the variational approximation when passing from one discrete time step to the next, see \cite[Def.\ 4.8]{KT17}. 
Instead, we will control the mass of a $2$-current in $\R^n$ which will be naturally induced through the geometric evolution of the transportation networks.
This approach will allow us to consider arbitrary codimension smaller than $n$, in contrast to \cite{KT17}.
Further, we will not need to introduce a weight function to model a growth condition near $|x|=\infty$ and the underlying structures of our transportation networks, which are the countably $\Ha$-rectifiable sets $S$ from above, do not necessarily need to be closed, as opposed to \cite{B78,KT17}.
In this article, we will restrict our attention to closed (in the sense of currents) transportation networks. 
We believe that this assumption provides the most natural basic setting for a transportation network geometric flow, by analogy with curve-shortening flow, which allows the evolving curve to collaps.
The evolution of the transportation networks will decrease their branched transport cost. 
In particular, the subadditivity of $\tau$ yields a branching rule forming angles different from $120^\circ$ at inner junctions of $S$, see \cref{fig0}.
\begin{figure}
	\centering
	\begin{tikzpicture}[scale=2, transform shape=false,rotate=-15] 
		\draw[line width=2,lightgray]  ({-1},{sqrt(3)}) -- (0,0);
		\draw[line width=3,lightgray]  (0,0) -- (2,0);
		\draw[line width=1,lightgray]  ({-1},{-sqrt(3)}) -- (0,0);
		\draw[->,line width=2,lightgray]  ({-1},{sqrt(3)}) -- ++({1/2},{-sqrt(3)/2});
		\draw[->,line width=1,lightgray]  ({-1},{-sqrt(3)}) -- ++({1/2},{sqrt(3)/2});
		\draw[->,line width=3,lightgray]  (0,0) -- ++(1,0);
		
		\draw[line width=2,gray]  ({-1},{sqrt(3)}) -- (0.5,0.25);
		\draw[line width=1,gray]  ({-1},{-sqrt(3)}) -- (0.5,0.25);
		\draw[line width=3,gray]  (0.5,0.25) -- (2,0);
		\draw[->,line width=2,gray]  ({-1},{sqrt(3)}) -- ++({(0.5+1)/sqrt(1.5*1.5+(0.25-sqrt(3))*(0.25-sqrt(3)))},{(0.25-sqrt(3))/sqrt(1.5*1.5+(0.25-sqrt(3))*(0.25-sqrt(3)))});
		\draw[->,line width=1,gray]  ({-1},{-sqrt(3)}) -- ++({(0.5+1)/sqrt(1.5*1.5+(0.25+sqrt(3))*(0.25+sqrt(3)))},{(0.25+sqrt(3))/sqrt(1.5*1.5+(0.25+sqrt(3))*(0.25+sqrt(3)))});
		\draw[->,line width=3,gray]  (0.5,0.25) -- ++({1.5/sqrt(1.5*1.5+0.25*0.25)},{-0.25/sqrt(1.5*1.5+0.25*0.25)});
		
		\draw[->,line width=2]  ({-1},{sqrt(3)}) -- ++({sqrt(3)/2},{-1/2});
		\draw[->,line width=1]  ({-1},{-sqrt(3)}) -- ++({sqrt(3)/2},{1/2});
		\draw[line width=2]  ({-1},{sqrt(3)}) -- (2,0);
		\draw[line width=1]  ({-1},{-sqrt(3)}) -- (2,0);
		
		\node[lightgray,fill,circle,minimum size=3mm,inner sep=0pt] at (0,0) {};
		\node[gray,fill,circle,minimum size=3mm,inner sep=0pt] at (0.5,0.25) {};
		\node[fill,circle,minimum size=3mm,inner sep=0pt] at (2,0) {};
		
		\node at (0.22,0.7) {$\textcolor{gray}{e_1}$};
		\node at (0,-0.6) {$\textcolor{gray}{e_2}$};
		\node at (1.2,0.24) {$\textcolor{gray}{e_3}$};
		
	\end{tikzpicture}
	\caption{The optimal angles at an inner junction are such that $\tau(m_1)\vec{e}_1+\tau(m_2)\vec{e}_2=\tau(m_1+m_2)\vec{e}_3$, where $m_i$ is transported along edge $e_i$ in unit direction $\vec{e}_i$ \cite{X}. If $\tau(m)=m^\alpha$ with $\alpha\in[0,1]$, then the limit cases $\alpha=0$ and $\alpha=1$ correspond to the Steiner problem (all angles are equal to $120^\circ$) and Wasserstein-$1$ transport (the inner junction degenerates for $\alpha\to 1$). Here, the thickness of the lines indicates the amount of mass transported along each edge; in particular, we have $m_1>m_2=m_3-m_1$ (Kirchhoff’s law).}
	\label{fig0}
\end{figure}
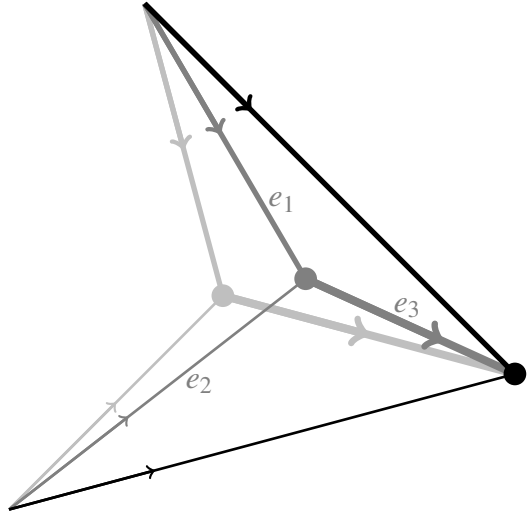
We will impose a suitable growth condition on $\tau$, in particular to ensure compactness properties.
Thanks to the closure theorem for integral currents, we will directly obtain the integrality of our geometric flow of transportation networks provided the initial datum $T_0$ is integral.
In order to derive a motion law, which will be a Brakke-type inequality, we will investigate a closely related geometric flow of varifolds whose weight measures model the branched transport cost up to possible cancellations due to opposite orientations.
Formally, our primary result is as follows (see \cref{MainResults} for details):
\\\\
\textit{If the initial datum $T_0$ is closed and integral, then, under some growth condition on $\tau$, there exists a Hölder continuous (with respect to the flat norm) geometric flow of integral transportation networks $(T_t)_{t\in\R_+}$ and a corresponding geometric flow of varifolds $(V_t)_{t\in\R_+}$ such that $V_t$ is $1$-rectifiable for a.e.\ $t\in\R_+$ and $(V_t)_{t\in\R_+}$ satisfies a Brakke-type inequality which represents the geometric evolution of $(T_t)_{t\in\R_+}$ up to cancellations.}
\\\\
Here ``up to cancellations'' means that $|T_t|_\tau(\R^n)\leq \| V_t\|(\R^n)$ and $T_t$ is approximated by a sequence of transportation networks $T_t^k=\theta_t^k\Ha\mres S_t^k$ whose induced ``$\tau$-weighted'' varifolds $v(\tau(|\theta_t^k|),S_t^k)$ approximate $V_t$ for a.e.\ $t\in\R_+$ (both in the weak-$*$ sense). The geometric flows $(T_t^k)_{t\in\R_+}$ are obtained by a suitable modification of Brakke's variational approximation scheme.

Our construction lays the foundation for many interesting new research directions. 
The most natural being the relaxation of the closedness condition to the general boundary condition $\partial T=\mu_+-\mu_-$.
In the classical setup, the corresponding case was investigated by Stuvard and the second author under some mild assumptions \cite{ST21}.
The essential idea in \cite{ST21} is a suitable damping of the motion near the fixed boundary.
A rigorous investigation of the case $\partial T=\mu_+-\mu_-$ would provide the foundation for a numerical analysis and implementation.
A classical contribution in this direction is Brakke’s \textit{Surface Evolver} \cite{B92}.
Another direction would be to study the regularity of the transportation network geometric flow at $t=\infty$.
For example, it is natural to expect a ``conservation of momentum'' condition, as in \cref{fig0}, to hold.
Similar properties were studied by Kim and the second author in \cite{KT20}.

To the best of found knowledge, the only attempt to construct a Brakke flow (its precise definition depends on the reference) by employing the theory of currents can be found in Ilmanen's work \cite{I94}.
He considered (assuming $M=\R^n$ and $k=1$ in \cite{I94}) the initial datum $T_0$ as an integral $1$-current in $\R^n\times\{ 0 \}$ and a family of energy functionals, indexed by $\varepsilon>0$, on integral $2$-currents in $\R^n\times\R$ whose boundary equals $T_0$.
Then he proved that there exist stationary points $P_\varepsilon$ approximating (after rescaling the time variable by $\varepsilon$) a Brakke flow: for $\varepsilon\to 0$ one obtains (in a suitable sense) an integral $2$-current $P$ in $\R^n\times\R_+$ with $\partial P=T_0$ and a family $(\mu_t)_{t\in\R_+}\subset\M_+(\R^n)$ representing the weight measures of $1$-varifolds moving in the sense of Brakke. Further, the time slices $t\mapsto T_t=\partial (P\mres (\R^n\times [t,\infty)))$ evolve continuously with respect to the weak-$*$ topology and $|T_t|(\R^n)\leq \mu_t(\R^n)$ for all $t\in\R_+$.
It may be possible to extend this construction by incorporating a transportation cost.
However, it should be noted that it heavily relies on the assumption that $T_0$ is given as a boundary.
We believe that our approach is more suitable for a later extension to the case $\partial T_0=\mu_+-\mu_-$.

The article is organized as follows. In \cref{MainResults} we summarize our main results and describe the variational approximation scheme for the transportation network geometric flow.
In \cref{Preliminaries}, we state the definitions and notation, along with several auxiliary results relevant to this work.
\Cref{GFInduced,ExTNFlow} are devoted to the the existence and Hölder regularity of a transportation network geometric flow $(T_t)$ and a corresponding geometric flow of measures $(\mu_t)$.
In \cref{RectifiablitySection} we prove, for a.e.\ time, the $1$-rectifiability of a family of varifolds $(V_t)$ whose weight measures correspond to $(\mu_t)$.
Finally, in \cref{MotionLawSection} we show that the family $(V_t)$ satisfies a Brakke-type inequality which is closely related to the geometric evolution of the $(T_t)$.
\section{Main results and variational approximation scheme}
\label{MainResults}
The main results of this article are
\begin{itemize}
	\item the existence and Hölder regularity of a geometric flow of transportation networks that decreases the branched transport cost (\cref{LimitCurrent}),
	\item the $1$-rectifiability of a closely related geometric flow of $1$-varifolds (\cref{rectifiability}),
	\item and a motion law similar to Brakke's inequality thereof (\cref{MotionLaw}).
\end{itemize}
Recall from \cref{Intro} the definition of a transportation cost.
\begin{defin}[{Transportation cost \cite[Def.\ 1.1]{BW}}]
	\label{TransportationCost}
	A \textbf{transportation cost} is a non-decreasing, lower semicontinuous, and subadditive function $\tau:\R_+\to\R_+$ with $\tau(0)=0$.
\end{defin}
If $\tau$ is a transportation cost with $\tau(m)=0$ for some $m\in\R_+^*$, then $\tau\equiv 0$ using that $\tau$ is non-decreasing and subadditive.
If $\tau:\R_+\to\R_+$ is concave with $\tau(0)=0$, then one can check that $\tau$ is a transportation cost.
We will need the following assumption.
\begin{assump}[Growth condition on $\tau$]
	\label{GrowthCondition}
	There exists $C\in\R_+^*$ such that $Cm\leq\tau(m)$ for all $m\in\R_+$.
\end{assump}
\begin{figure}
	\centering
	\begin{tikzpicture}
		\begin{axis}[axis on top=true,
			width=0.48\textwidth,
			scale only axis,
			axis equal image,
			samples = 500, 
			domain = 0:5, 
			axis line style={thick},
			xmax = 5,
			ymax = 3,
			axis x line=bottom, 
			axis y line=left, 
			xlabel = { $m$}, 
			ylabel={},
			ytick=\empty,
			x label style={at={(axis description cs:1,0)},anchor=south},
			y label style={at={(axis description cs:.07,0.9)},rotate=270,anchor=west},
			xtick = {1,2},
			xticklabels = {$1$,$2$},
			legend style = {cells={anchor=west},legend pos=outer north east},
			]
			\addplot[gray, line width = 1.8pt] expression {x};
			\addplot[lightgray, line width = 2.9pt] expression {min(2*x,x+1)};
			\addplot[black,domain=0:2,line width = 1.5pt] expression {pow(x,0.5)};
			\addplot[dashed] expression {ceil(x)};
			\addplot[black,domain=2:5,line width = 1.5pt] expression {sqrt(2)+(x-2)/(2*sqrt(2))};
			\addplot[dotted, mark=none] expression {pow(x,0.5)};
			\addplot[only marks,mark=*,mark options={fill=black,scale=.8},text mark as node=true] coordinates{(1,1)};
			\addplot[only marks,mark=*,mark options={fill=white,scale=.7},text mark as node=true] coordinates{(1,2)};
			\addplot[only marks,mark=*,mark options={fill=black,scale=.8},text mark as node=true] coordinates{(0,0)};
			\addplot[only marks,mark=*,mark options={fill=white,scale=.7},text mark as node=true] coordinates{(0,1)};
			\addplot[only marks,mark=*,mark options={fill=black,scale=.8},text mark as node=true] coordinates{(2,2)};
			\addplot[only marks,mark=*,mark options={fill=white,scale=.7},text mark as node=true] coordinates{(2,3)};
			\addplot[only marks,mark=*,mark options={fill=black,scale=.8},text mark as node=true] coordinates{(3,3)};
			\legend{Wasserstein cost $\mathcal W(m)=m$\\
				urban planning cost $\mathcal U(m)=\min\{ 2m,m+1\}$\\
				$\tau_2(m)$ approximating $\tau(m)=\sqrt{m}$\\
				capacity cost $\mathcal C(m)=\lceil m\rceil$\\};
		\end{axis}
	\end{tikzpicture}
	\caption{Transportation costs satisfying \cref{GrowthCondition}.}
	\label{fig1}
\end{figure}
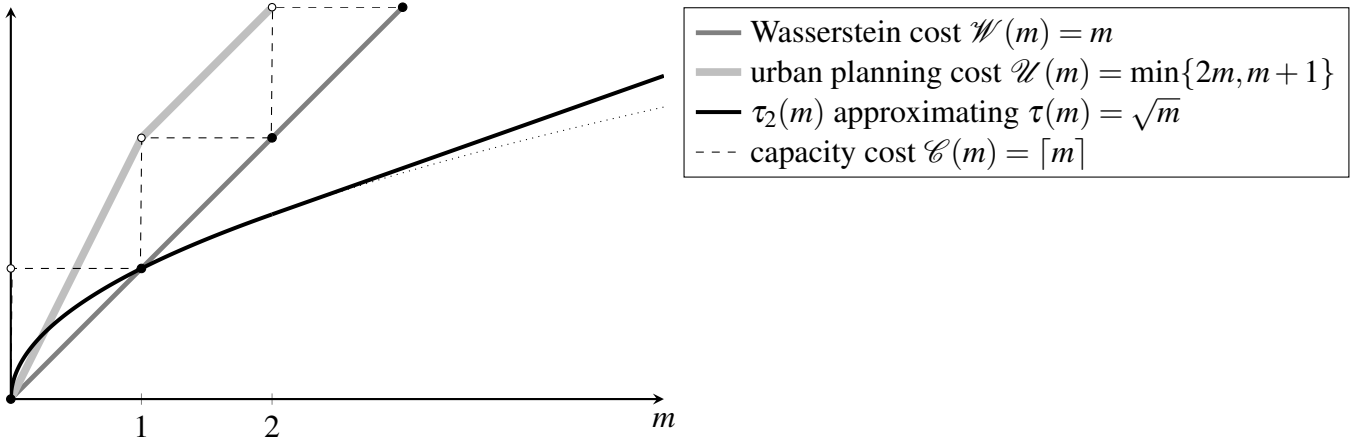
The simplest example of a transportation cost satisfying \cref{GrowthCondition} is the so-called \textit{Wasserstein cost} $\mathcal W(m)=m$ modeling Wasserstein-$1$ transport without any branching. 
Another important example is the so-called \textit{urban planning cost} $\mathcal U(m)=\min\{ am,bm+c\}$ for some $a,b,c\in\R_+^*$ with $a>b$ introduced in \cite{BW16}.
If a transportation cost $\tau$ is concave and increasing, for example $\tau(m)=m^\alpha$ for some $\alpha\in(0,1)$, then one may replace it by
\begin{equation*}
\tau_M(m)=\begin{cases*}
\tau(m)&if $m\leq M$,\\
\tau(M)+(m-M)\tau'(M)&else,
\end{cases*}
\end{equation*}
and take $M\in\R_+^*$ arbitrarily large such that $\tau'(M)$ exists (which is the case for a.e.\ $M\in\R_+$).
By definition $\tau_M$ satisfies \cref{GrowthCondition}.
Hence, for sufficiently large masses, there is no efficiency gain when the particles of a mass flux bundle.
A nonconcave example of a transportation cost satisfying \cref{GrowthCondition} is given by $\mathcal C(m)=\lceil m/K\rceil$ for some $K\in\R_+^*$, which may be used to model a jump in the cost if the capacity $K$ of a means of transport is reached. See \cref{fig1}.

First, let us sketch the variational approximation scheme for the geometric flow of transportation networks.
To this end, fix a transportation cost $\tau:\R_+\to\R_+$ and an initial transportation network, that is a normal real rectifiable $1$-current $T_0=\theta_0\Ha\mres S_0\in \T=\N_1(\R^n)\cap\rct_1(\R^n)$.
Here $\theta_0\in L^1_{loc}(\Ha\mres S_0;\R^n)$ is tangential to countably $\Ha$-rectifiable set $S_0$ and $\Ha\mres S_0$ denotes the restriction of the one-dimensional Hausdorff measure to $S_0$, see \cref{Preliminaries}.
Write $|T_0|_\tau(B)$ for the branched transport cost of $T_0$ inside a Borel set $B\subset\R^n$, cf.\ \cref{BTC}.
Further, set $\Omega=\textup{conv}(\spt(T_0))$, the convex hull of the support of $T_0$.
Assume that $|T_0|_\tau(\R^n)<\infty$.
For each $j\in\mathbb N$ we will choose $p_j\in\mathbb N$ and define a finite sequence of transportation networks
\begin{equation*}
T^{j,0},T^{j,1},\ldots,T^{j,j2^{p_j}}
\end{equation*}
as follows:
First, set $T^{j,0}=T_0$ for all $j\in\mathbb N$.
Now fix $j\in\mathbb N$.
We sketch how $T^{j,\ell+1}$ is obtained from $T^{j,\ell}$ for $\ell=0,\ldots,j2^{p_j}-1$. 
It constitutes a modification of the variational approximation schemes in \cite{B78,KT17}, see \cref{Prop6.1} and its proof.
\begin{itemize}
	\item First, we choose an appropriate \textit{admissible deformation} (see \cref{AdmDeform})
	\begin{equation*}
		f_1\in\adm_j(T^{j,\ell}) \qquad\textup{and set}\qquad\widetilde T^{j,\ell+1}=(f_1)_\# T^{j,\ell}.
	\end{equation*}
	\item Then we project $\widetilde T^{j,\ell+1}$ onto $\Omega$ via orthogonal projection $f_{1.5}:\R^n\to\Omega$,
	\begin{equation*}
		\widehat T^{j,\ell+1}=(f_{1.5})_\# \widetilde T^{j,\ell+1}.
	\end{equation*}
	\item Finally, we define a diffeomorphism (where $\Dt_j=2^{-p_j}$ and $\varepsilon_j<j^{-6}$)
	\begin{equation*}
		f_2=\textup{id}+\Dt_j h_{\tau,\varepsilon_j}(\widehat T^{j,\ell+1})\qquad\textup{and set}\qquad T^{j,\ell+1}=(f_2)_\#(\widehat T^{j,\ell+1}).
	\end{equation*}
\end{itemize}
Here $h_{\tau,\varepsilon_j}(\widehat T^{j,\ell+1})$ denotes an \textit{approximate $\tau$-curvature vector field}, cf.\ \cref{ApproxCurveVF}.
Afterwards, for every $j\in\mathbb N$, we set
\begin{equation*}
	T_t^j=T^{j,\lceil t/\Dt_j\rceil}
\end{equation*} 
for all $t\in[0,j]$. Note that $[0,j]$ increases linearly in $j$ while the decrease of $\Dt_j$ is exponentially ($j\mapsto p_j$ will be increasing).
Then we will show that there is a subsequence $(j_k)$ with (\cref{Prop6.4})
\begin{equation*}
	|T_t^{j_k}|_\tau\ws\mu_t\in\M_+(\R^n)
\end{equation*} 
for $k\to\infty$ for all $t\in\R_+$. This will require us to carefully adapt the proofs in \cite[Secs.\ 5 \&\ 6]{KT17} to our setting.
The first self-contained main result is an existence and regularity theorem related to the transportation networks from above.
It will use the following two assumptions.
\begin{assump}[$T_0$ closed]
	\label{ZeroBoundary}
	The transportation network $T_0$ is closed as a current: $\partial T_0=0$.
\end{assump}
\begin{assump}[Integrality of $T_0$]
	\label{IntegralityT0}
	The initial datum $T_0$ is integral: $T_0\in\mathcal I_1(\R^n)$.
\end{assump}
\begin{thm}[Existence and regularity of transportation network geometric flow $(T_t)$]
	\label{LimitCurrent}
	Consider the setup from above, see \cref{Prop6.1,Prop6.4}. Let \cref{GrowthCondition,ZeroBoundary} be satisfied.
	Then, possibly after refining the subsequence $(j_k)$, there exists a family of normal $1$-currents $(T_t)_{t\in\R_+}\subset\mathcal N_1(\R^n)$ with 
	\begin{equation*}
	\spt(T_t)\subset\Omega\qquad\textup{and}\qquad T_t^{j_k}\ws T_t
	\end{equation*}
	for $k\to\infty$ for all $t\in\R_+$.
	Moreover, the map $t\mapsto T_t$ satisfies the Hölder condition
	\begin{equation*}
	|T_s-T_t|^\flat\leq C_0|s-t|^{1/2}
	\end{equation*}
	 for all $s,t\in\R_+$ and some constant $C_0\in\R_+$, where $|\cdot|^\flat$ denotes the flat norm, see \cref{Currents}. If \cref{IntegralityT0} is satisfied, then $T_t\in\mathcal I_1(\R^n)$ for all $t\in\R_+$.
\end{thm}
The order in which the deformations $f_1,f_{1.5}$, and $f_2$ are applied will be of great importance.
In particular, we will need to carefully estimate the flat norm between transportation networks arising from different iterations, drawing in part on the homotopy formula for currents, upper bounds in terms of integrals involving the approximate $\tau$-curvature vector fields, and a triangle-type inequality for currents corresponding to different iterations, both obtained after applying a deformation of type $f_{1.5}$.
These estimates will also yield the Hölder continuity.

Now let the sequence $(j_k)$ be as in \cref{LimitCurrent}. 
Since $T_t^{j_k}$ is a transportation network, we can write $T_t^{j_k}=\theta_t^{j_k}\Ha\mres S_t^{j_k}$ for all $t\in [0,j_k]$ and $k\in\mathbb N$.
The next statment is related to the $1$-varifolds induced by $\tau(|\theta_t^{j_k}|)\Ha\mres S_t^{j_k}$, cf.\ \cref{Varifolds}.
\begin{thm}[$1$-rectifiability of limit varifolds]
	\label{rectifiability}
	Let the setup be as above, see \cref{Prop6.1,Prop6.4}. Then the following holds under \cref{GrowthCondition,IntegralityT0}: For a.e.\ $t\in\R_+$ the induced varifolds
	\begin{equation*}
		V_t^{j_{k}}=v(\tau(|\theta_t^{j_{k}}|),S_t^{j_{k}})\in\M_+(\R^n\times G(n,1))
	\end{equation*}
	satisfy $V_t^{j_{k}}\ws V_t\in \M_+(\R^n\times G(n,1))$ for $k\to\infty$ with $V_t$ fulfilling all the properties in \cref{AllardRect} $=$ Allard's rectifiability theorem \cite[Sec.\ 5.5, Thm.\ 1]{A72}, in particular $V_t$ is $1$-rectifiable. Note that automatically $\| V_t \|=\mu_t$, see above or \cref{Prop6.4}.
	Further, for a.e.\ $t\in\R_+$, there is a generalized curvature vector field $h_{V_t}\in L^2(\| V_t\|;\R^n)$ with $\delta V_t(g)=-\int g\cdot h_{V_t}\e\| V_t \|$ for all $g\in C_c^1(\R^n;\R^n)$.
	Moreover, we have
	\begin{equation}
		\label[ineq]{in1:timeIntmean}
		\int_{[0,R]}\int |h_{V_t}|^2\e\| V_t \|\e\Leb(t)<\infty
	\end{equation}
	for all $R\in\R_+$.
\end{thm}
Here $\| V_t\|$ denotes the weight measure and $\delta V_t$ is the first variation of $V_t$.
Note that a decrease of $\| V_t^{j_k}\|(\R^n)=|T_t^{j_k}|_\tau(\R^n)$ corresponds to a reduction of the branched transport cost.
Further, using $T_t^{j_k}\ws T_t$ and $V_t^{j_{k}}\ws V_t$ together with the lower semicontinuity of $|\cdot|_\tau(\R^n)$ \cite[Def.\ 2.2]{BW} (see also \cite[Thm.\ 5]{MW19} or the introduction of \cite{CDMS17}), we obtain
\begin{equation*}
|T_t|_\tau(\R^n)\leq\liminf_k|T_t^{j_k}|_\tau(\R^n)=\liminf_k\| V_t^{j_k}\|(\R^n)=\| V_t\|(\R^n)
\end{equation*}
for a.e.\ $t\in\R_+$.
The inequality is due to possible cancellations through opposite orientations in the limit.
The network-like sets $S_t^{j_{k}}$ are not necessarily closed, as opposed to the setup in \cite{B78,KT17}.
We will use a strategy different to the proof in \cite[Sec.\ 7]{KT17} by employing estimates similar to the inequalities in \cite[Lem.\ 5.2]{KT20} and \cite[Lem.\ 5.4]{KT20} together with Allard's rectifiability theorem.
The work here is to construct, for a.e.\ $t\in\R_+$, suitable admissible deformations and prove that, in the limit and up to a subsequence, the varifolds $V_t^{j_{k}}$ are concentrated in points where the one-dimensional density with respect to their weight measures $|T_t^{j_{k}}|_\tau$ is bounded away from zero.

Finally, we modify the proof of \cite[Thm.\ 9.3]{KT17} to prove a Brakke-type inequality for the geometric flow $(V_t)$, which may be seen as a $\tau$-weighted variant of the classical Brakke flow.
In contrast to \cite{KT17}, the varifolds $V_t^{j_k}$ are not unit density and their weights are given by $\tau(|\theta_t^{j_{k}}|)$ for all $t\in[0,j_k]$, thus Brakke’s perpendicularity theorem \cite[Sec.\ 5]{B78} is not applicable.
This explains the presence of the term $P^\perp\nabla\phi(x)$ below, the orthogonal projection of $\nabla\phi(x)$ onto the orthogonal complement of $P\in G(n,1)$.
\begin{thm}[{Motion law of $t\mapsto V_t$, cf.\ \cite[Secs.\ 3 \&\ 4]{B78} and \cite[Thm.\ 9.3]{KT17}}]
	\label{MotionLaw}
	Let the setup be as above, see \cref{Prop6.1,Prop6.4}, and $V_t, h_{V_t}$ as in \cref{rectifiability} for a.e.\ $t\in\R_+$. Under \cref{GrowthCondition,IntegralityT0}, the following holds for a.e.\ $(t_1,t_2)\in\R_+^2$ with $t_1<t_2$:
	\begin{multline*}
		\| V_{t_2} \|(\phi(.,t_2))-\| V_{t_1} \|(\phi(.,t_1)) \leq-\int_{[t_1,t_2]}\int |h_{V_t}|^2\phi\e\| V_t\|\e\Leb(t)\\
		+\int_{[t_1,t_2]}\int h_{V_t}(x)\cdot (P^\perp\nabla\phi(x))\e V_t(x,P)\e\Leb(t)
		+\int_{[t_1,t_2]}\| V_t\|(\partial_t\phi(.,t))\e\Leb(t)
	\end{multline*}
	for all $\phi\in C_c^\infty(\R^n\times\R_+;\R_+)$.
\end{thm}
\section{Preliminaries}
\label{Preliminaries}
\subsection{Basic definitions and notation}
\label{Basic}
Let $U\subset\R^n$ be open.
\begin{itemize}
	\item $\mathbb N,\mathbb N_0$: \textit{natural numbers} $\mathbb N=\{ 1,2,\ldots \}$ and $\mathbb N_0=\mathbb N\cup\{ 0 \}$
	\item $\R_-,\R_+,\R_+^*$: non-positive, non-negative, positive \textit{real numbers}
	\item $\lceil x\rceil$: smallest integer $z$ with $x\leq z$, where $x\in\R$
	\item $\mathcal S^{n-1}$: $(n-1)$-dimensional \textit{unit sphere} in $\R^n$
	\item $[x,y],(x,y)$: \textit{line segments} $x+[0,1](y-x)$ and $x+(0,1)(y-x)$ for $x,y\in\R^n$
	\item $B_r(x),\overline{B}_r(x)$: \textit{open and closed ball} with radius $r\in\R_+^*$ and centre $x$ in metric space 
	\item $v\cdot w,M : N$: \textit{Euclidean inner product} of vectors and \textit{Frobenius inner product} of matrices
	\item $v\otimes w$: \textit{dyadic product} between vectors $v,w\in\R^n$
	\item $c(x_1,\ldots,x_N)$: often used as a \textit{generic constant} depending only on data $x_1,\ldots,x_N$
	\item $\int$: often replaces $\int_A$ if domain of integration $A$ is clear from the context
	\item $1_A$: \textit{characteristic function} of set $A$, $1_A(x)=1$ if $x\in A$, $1_A(x)=0$ else
	\item $|\cdot|,|\cdot|_2$: Euclidean norm on $\R^n$, $2$-norm on $\R^{n\times n}$
	\item $|\cdot|_{2,2}$: norm on $\R^{n\times n\times n}$ defined by $|M|_{2,2}=\sup_{|x|\leq 1}|Mx|_2$, where $(Mx)_{i,j}=\sum_kM_{i,j,k}x_k$ for $x\in\R^n$
	\item $|f|_\infty$: \textit{uniform norm} $|f|_\infty=\sup_A\| f \|_V$ of map $f$ on $A$ with values in normed vectors space $V$
	\item $\textup{dom}(f)$: \textit{effective domain} $\{ x\in A\:|\: f(x)<\infty \}$ of function $f:A\to\R\cup\{ \infty \}$
	\item $\textup{Lip}(f)$: \textit{Lipschitz constant} of $f:U\to\R^m$ with respect to Euclidean norms
	\item $\partial_i f$: $i$-th \textit{partial derivative} of partially differentiable $f:U\to\R^m$
	\item $\nabla f,\nabla^2 f$: \textit{gradient} and \textit{Hessian matrix} of once respectively twice partially differentiable $f:U\to\R$
	\item $\Jm f,\Jac f$: \textit{Jacobian matrix} and \textit{Jacobian} $\Jac f=\textup{det}(\Jm f^\intercal \Jm f)^{1/2}$ of partially differentiable  $f:U\to\R^n$ 
	\item $\Jm^2f$: if $f:U\to\R^n$ twice partially differentiable, then $\Jm^2f\in \R^{n\times n\times n}$ given by $(\Jm^2f)_{i,j,k}=\partial_k\partial_j f_i$
	\item $\partial A$, $\clos(A)$: \textit{boundary, closure} of subset $A$ of topological space
	\item $\textup{conv}(A)$: \textit{convex hull} of subset $A$ of real vector space
	\item $\textup{dist}(A,B),\textup{dist}(x,B)$: distance between two subset $A, B$ of a metric space, $\textup{dist}(x,B)=\textup{dist}(\{ x\},B)$
	\item $\langle x,x'\rangle$: \textit{dual pairing} between $x\in X$ and $x\in X'$, where $X'$ topological dual space of TVS $X$
	\item $GL(k)$: linear group of invertible matrices in $\R^{k\times k}$
	\item $A^\intercal$: \textit{transpose} of $A\in\R^{m\times n}$
	\item $P^\perp$: matrix realizing the projection onto orthogonal complement of $k$-dimensional subspace $V$ of $\R^n$, where $P\in\R^{n\times n}$ projection onto $V$
	\item $\mathbb 1 $: identity matrix in $\R^{n\times n}$
	\item $G(n,k)$: \textit{Grassmannian} of $k$-dimensional subspaces of $\R^n$ \cite[Sec.\ 1.6.2]{F} with final topology induced by bijection $f:X/{\sim}\to G(n,k)$, where $X\subset \R^{n\times k}$ subset of matrices with linearly independent columns, $M_1\sim M_2$ if $M_1=M_2G$ for some $G\in GL(k)$, and $X/{\sim}$ equipped with quotient topology
	\item $\mathcal G_k(U)$: defined by $U\times G(n,k)$ (with product topology)
	\item $\mathcal L^k, \mathcal L, \mathcal H^k$: $k$-dimensional \textit{Lebesgue measure}, $\mathcal L=\mathcal L^1$; $k$-dimensional \textit{Hausdorff measure}
	\item $\B(X)$: \textit{Borel $\sigma$-algebra} over topological space $X$
	\item $\M(X),\M_+(X),\M(X;\R^n)$: $\R$-valued, $\R_+$-valued, and $\R^n$-valued Radon measures on Polish space $X$
	\item $\Lambda_k\R^n,\Lambda^k\R^n$: \textit{$k$-vectors} in $\R^n$, \textit{alternating multilinear $k$-forms} on $\R^n$
	\item $\D^k(U)$: LCTVS of smooth and compactly supported \textit{differential $k$-forms} on $U$ \cite[Sec.\ 4.1.7]{F} with topology induced by seminorms $p_{\ell,K}(\sum_{1\leq i_1<\ldots <i_k\leq n}f_{i_1\ldots i_k}\e x_{i_1}\wedge\ldots\wedge \e x_{i_k})=\sup_{x,\alpha,j_1,\ldots,j_k}|\partial^\alpha f_{j_1\ldots j_k}(x)|$ for $\ell\in\mathbb N_0$ and $K\subset U$ compact, where supremum over $x\in K$, $\alpha\in\mathbb N_0^n$ with $\alpha_1+\cdots+\alpha_n\leq \ell$, and $1\leq j_1<\ldots <j_k\leq n$ 
	\item $\e\omega$: \textit{exterior derivative} of $\omega\in\mathcal{D}^k(U)$
	\item $|\alpha|_{\textup{co}}$: \textit{comass} of $\alpha\in\Lambda^k\R^n$,
	\begin{equation*}
		|\alpha|_{\textup{co}}=\sup\{ \langle\xi,\alpha\rangle\:|\:\xi\in\Lambda_k\R^n\textup{ simple},|\xi|_{\textup{vol}}=1 \},
	\end{equation*}
	where $|\xi|_{\textup{vol}}$ denotes the volume of the parallelpiped spanned by the vectors in $\xi=v_1\wedge\ldots\wedge v_k$
	\item $|\omega|_K^\flat$: (dual) \textit{flat seminorm} of $\omega\in\mathcal{D}^k(U)$ with respect to compact $K\subset U$ \cite[Sec.\ 4.1.12]{F},
	\begin{equation*}
		|\omega|_K^\flat=\sup_{x\in K}\,\max\{ |\omega_x|_{\textup{co}},|\e\omega_x|_{\textup{co}} \}
	\end{equation*}
	\item $\mu\mres A$: \textit{restriction} of outer measure $\mu$ on set $X$ to $A\subset X$ \cite[Sec.\ 2.1.2]{F},
	\begin{equation*}
	(\mu\mres A)(B)=\mu(A\cap B)
	\end{equation*}
	for all $B\subset X$
	\item $L^1(\Hak\mres A;\Lambda_\ell\R^n),L_{loc}^1(\Hak\mres A;\Lambda_\ell\R^n)$: spaces of equivalence classes of ($\Hak\mres A$)-measurable maps $f:\R^n\to \Lambda_\ell\R^n $ (that is $f^{-1}(B)\cap A$ is $\Hak$-measurable for all $B\in\B(\Lambda_\ell\R^n)$) with
	\begin{equation*}
	\begin{cases*}
	\int_A\|f\|\e\Hak<\infty&if $f\in L^1(\Hak\mres A;\Lambda_\ell\R^n)$,\\
	\int_{A\cap K}\|f\|\e\Hak<\infty \textup{ for all compact $K\subset\R^n$}&if $f\in L_{loc}^1(\Hak\mres A;\Lambda_\ell\R^n)$,
	\end{cases*}
	\end{equation*}
	where $A\subset \R^n$ $\Hak$-measurable, $\|\cdot\|$ norm on $\Lambda_\ell\R^n$, and two maps equivalent if equal $\Hak$-a.e.\ in $A$
	\item $\textup{Tan}(A,x)$: \textit{tangent cone} of subset $A$ of normed vector space $X$ at $x\in X$ \cite[Sec.\ 3.1.21]{F},
	\begin{equation*}
	\textup{Tan}(A,x)=\{ y\in X\:|\: \forall \varepsilon\in\R_+^*\exists a\in A,r\in\R_+^{*}:|x-a|+|r(a-x)-y|<\varepsilon \};
	\end{equation*}
	we have $\textup{Tan}(A,x)\cap \mathcal S^{n-1}=\bigcup_{r\in\R_+^*}\textup{clos}(\{ (y-x)/|y-x|\:|\: y\in A\cap B_r(x)\setminus\{  x \}\})$, cf.\ \cite[Sec.\ 2.1]{A72}
	\item $\Theta_{*}^k(\mu,x),\Theta^{*k}(\mu,x)$: $k$-dimensional \textit{lower and upper densities} of outer measure $\mu$ on metric space $X$ in $x\in X$ \cite[Sec.\ 2.10.19]{F},
	\begin{equation*}
	\Theta_*^k(\mu,x)=\liminf_{r\searrow 0}   \frac{\mu(B_r(x))}{\omega_kr^k}\leq\limsup_{r\searrow 0}   \frac{\mu(B_r(x))}{\omega_kr^k}=\Theta^{*k}(\mu,x),
	\end{equation*}
	where $\omega_k=\mathcal L^k(\{ y\in\R^k\:|\:|y|<1 \})$; if $\Theta_*^k(\mu,x)=\Theta^{*k}(\mu,x)$, then we write $\Theta^k(\mu,x)=\Theta_*^k(\mu,x)$
	\item $\textup{Tan}^k(\mu,x)$: \textit{\textit{$(\mu,k)$}-approximate tangent cone} at $x\in X$ \cite[Sec.\ 3.2.16]{F},
	\begin{equation*}
	\textup{Tan}^k(\mu,x)=\bigcap\left\{ \textup{Tan}(B,x)\:\middle|\: B\subset X,\Theta^k(\mu\mres(X\setminus B),x)=0 \right\},
	\end{equation*}
	where $\mu$ outer measure on normed vector space $X$
\end{itemize}
In this article, we make use of several constants. Accordingly, we indicate below the corresponding cross-references within this document.
\begin{itemize}
	\item $C$: \cref{GrowthCondition} (growth condition on $\tau$)
	\item $c_\varepsilon$: \cref{ApproxTauVectorfields} (normalizing constant of $\Psi_{\varepsilon}$)
	\item $c$: \cref{EstimatesNabla} (estimates related to $\nabla^k\Psi_{\varepsilon}$)
	\item $c_{\varepsilon,n}$: defined in the proof of \cref{Prop5.3}
	\item $c_1$: $c_1=3n+20$ (defined in \cref{Prop5.7})
	\item $C_0$: \cref{LimitCurrent} (constant related to Hölder condition)
	\item $\hat c$: \cref{Lemma1}
	\item $c_2$: \cref{LowerBound} (constant depending implicitly on $t$)
\end{itemize}
\subsection{Currents}
\label{Currents}
In this section, we define the space of $k$-currents (in the sense of DeRham) in $\R^n$ and related notions. We follow the terminology by Federer and Fleming in \cite{F,FF60} up to some small modifications.
Let $U\subset\R^n$ be open.
\begin{defin}[$k$-current]
The space of \textbf{$k$-currents} in $U$, denoted by $\D_k(U)$, is defined as the topological dual space of $\D^k(U)$. 
\end{defin}
If $T\in \D_k(U)$, then we use the following standard notation:
\begin{itemize}
	\item $\partial T$: \textit{boundary} $\partial T\in\D_{k-1}(U)$ of $T$ defined by $\partial T(\omega)=T(\e\omega)$ for $\omega\in\D^k(U)$ if $k\in\mathbb N$ 
	\item $\spt(T)$: \textit{support} of $T$ \cite[Sec.\ 4.1.1]{F},
	\begin{equation*}
		\spt(T)=\bigcap\{ A\subset U\textup{ relatively closed}\:|\: T(\omega)=0\textup{ for all }\omega\in\D^k(U)\textup{ with }\textup{supp}(\omega)\subset U\setminus A \}
	\end{equation*}
	\item $\Mb(T)$: \textit{mass} of $T$, 
	\begin{equation*}
		\Mb(T)=\sup\{ T(\omega)\:|\:\omega\in\D^1(U),|\omega_x|_{\textup{co}} \leq 1\textup{ for all }x\in U \}
	\end{equation*}
	\item $|T|_K^\flat$: \textit{flat seminorm} of $T$ with respect to compact $K\subset U$ \cite[Sec.\ 4.1.12]{F},
	\begin{equation*}
		|T|_K^\flat=\sup\{ \langle\omega,T\rangle\:|\:\omega\in\D^k(U),|\omega|_K^\flat\leq 1 \}
	\end{equation*}
		\item $|T|^\flat$: defined by (called \textit{flat norm} if $\D_k(U)$ is restricted to the \textit{flat $k$-chains} in $U$)
	\begin{equation*}
		|T|^\flat=\inf\{ \mathbb M(\widetilde{T})+\mathbb M(\partial \widetilde{T} -T)\:|\:\widetilde{T}\in\D_{k+1}(U) \}
	\end{equation*}
\end{itemize}
\begin{defin}[{Cartesian product of currents \cite[Sec.\ 4.1.8]{F}}]
\label{ProdCurrents}
Let $V\subset\R^m$ and $T_U\in\D_k(U),T_V\in\D_\ell(V)$. Write $p_U:U\times V\to U$ and $p_V:U\times V\to V$ for the natural projections. Then there exists a unique $(k+\ell)$-current $T_U\times T_V\in\D_{k+\ell}(U\times V)$, called the \textbf{cartesian product of $T_U$ and $T_V$}, characterized by the property\footnote{If $W\subset\R^n$ open and $\omega\in\D^p(W),\widetilde{\omega}\in\D^q(W)$, then $\omega\wedge\widetilde{\omega}\in\D^{p+q}(W)$ is defined by $(\omega\wedge\widetilde{\omega})(x)=\omega(x)\wedge\widetilde{\omega}(x)$ for all $x\in W$ \cite[Sec.\ 4.1.6]{F}. Moreover, we have $p_U^\#\omega_1\in\D^i(U\times V)$ with $\langle\xi,(p_U^\#\omega_1)(u,v)\rangle=\langle (\Jm p_U(u,v))_\#\xi,\omega_1(p_U(u,v)) \rangle=\langle u_1\wedge\ldots\wedge u_i,\omega_1(u)\rangle$ for all $(u,v)\in U\times V$ and $\xi=(u_1,v_1)\wedge\ldots\wedge (u_i,v_i)\in\Lambda_i(\R^n\times\R^m)$ and similar for $p_V^\#\omega_2$ \cite[Secs.\ 1.3.1,1.4.1, 4.1.6]{F}.}
\begin{equation*}
(T_U\times T_V)(p_U^{\#}\omega_1\wedge p_V^{\#}\omega_2)=\begin{cases*}
\langle\omega_1,T_U\rangle\langle\omega_2,T_V\rangle&if $i=k$,\\
0&else,
\end{cases*}
\end{equation*}
for all $\omega_1\in\D^i(U),\omega_2\in\D^{k+\ell-i},i=0,\ldots,k+\ell$.
\end{defin}
In particular, we will use $\llbracket 0,1\rrbracket\times T\in\D_2(\R\times\R^n)$ for $T\in\D_1(\R^n)$, where $\llbracket 0,1\rrbracket\in\D_1(\R)$ is defined by \cite[p.\ 362]{F}
\begin{equation*}
	\llbracket 0,1\rrbracket (\phi)=\int_{[0,1]}\phi\e\Leb
\end{equation*}
for all $\phi\in\D^1(\R)$.
Next, we introduce countably $\Hak$-rectifiable sets and related notions. Afterwards, we will define rectifiable $k$-currents.
Every countably $\Hak$-rectifiable set $S\subset U$ is $(\mathcal H^k,k)$ rectifiable in the sense of \cite[Def.\ 3.2.14]{F}.
\begin{defin}[Countably $\Hak$-rectifiable]
\label{CountablyHkrectifiable}
A set $S\subset U$ is called \textbf{countably $\Hak$-rectifiable} if it is $\Hak$-measurable (in the sense of Carath\'eodory) and
\begin{equation*}
	\Hak\left( S\setminus\bigcup f_i(\R^k) \right)=0,
\end{equation*}
where the union is over countably many Lipschitz maps $f_i:\R^k\to U$.
\end{defin}
If $S\subset U$ is countably $\Hak$-rectifiable with $\Hak(S)<\infty$, then $\Theta^k(\Hak\mres S,x)=1$ and $\textup{Tan}^k(\Hak\mres S,x)\in G(n,k)$ for $\Hak$-a.e.\ $x\in S$ \cite[Thm.\ 3.2.19]{F}, cf.\ \cite[Ch.\ 3, § 1, Thm.\ 1.6]{LS}. We will often identify $\textup{Tan}^k(\Hak\mres S,x)$ with the projection $\textup{Tan}^k(\Hak\mres S,x):\R^n\to \textup{Tan}^k(\Hak\mres S,x)$. In the case $k=1$, we will also write $T_xS=\textup{Tan}^1(\Ha\mres S,x)$.
\begin{defin}[{Divergence relative to countably $\mathcal{H}^k$-rectifiable set, cf.\ \cite[Def.\ 4.2]{A72}}]
	Let $S\subset U$ be countably $\mathcal{H}^k$-rectifiable and locally $\Hak$-finite and $f\in C_c^1(U;\R^n)$. Then the \textbf{divergence of $f$ relative to $S$} is defined by
	\begin{equation*}
		\textup{div}_S(f(x))=P(x):\Jm f(x)
	\end{equation*}
	for $\mathcal{H}^k$-a.e.\ $x\in S$, where $P(x)=\sum_{i=1}^{k}e_i\otimes e_i$ for any orthonormal basis $e_1,\ldots,e_k\in \textup{Tan}^k(\Hak\mres S,x)$.
\end{defin}
\begin{defin}[Tangentiality to countably $\mathcal{H}^k$-rectifiable set]
	Let $S\subset U$ be countably $\Hak$-rectifiable and locally $\Hak$-finite. If $\theta\in L_{loc}^1(\Hak\mres S;\Lambda_k\R^n)$, then $\theta$ is \textbf{tangential to $S$} if 
	\begin{equation*}
		\theta(x)=v_1(x)\wedge\ldots\wedge v_k(x)
	\end{equation*}
	for certain $v_i(x)\in\textup{Tan}^k(\Hak\mres S,x)$ for $\Hak$-almost every $x\in S$.
\end{defin}
\begin{defin}[Real and integer-multiplicity rectifiable $k$-current]
	\label{RectCurrent}
	Let $T\in\D_k(U)$. Then $T$ is a \textbf{(real) rectifiable $k$-current} if there exists a countably $\Hak$-rectifiable and locally $\Hak$-finite set $S\subset U$ and $\theta\in L_{loc}^1(\Hak\mres S;\Lambda_k\R^n)$ tangential to $S$ such that $T=\theta\Hak\mres S$, that is
	\begin{equation*}
		\langle \omega,T\rangle=\int_S\langle\theta(x),\omega_x\rangle\e\Hak(x)
	\end{equation*}
	for all $\omega\in\D^k(U)$.
	We write $\rct_k(U)$ for the set of rectifiable $k$-currents in $U$.
	If $T=\theta\Hak\mres S\in \rct_k(U)$ with $|\theta(x)|_{\textup{vol}}\in\mathbb{N}$ for $\Hak$-a.e.\ $x\in S$, then we say that $T$ is an \textbf{integer-multiplicity rectifiable $k$-current}.
\end{defin}
Note that the definition of the space $\rct_k(U)$ (and $\mathcal I_k(U)$ below) is different from the definition in \cite[Sec.\ 4.1.24]{F}, see also \cite[Thm.\ 4.1.28]{F}.
\begin{defin}[Integral $k$-current]
	Let $T\in\D_k(U)$. If both $T$ and $\partial T$ are integer-multiplicity rectifiable currents, then $T$ is called an \textbf{integral $k$-current}. We write $\mathcal I_k(U)$ for the set of integral $k$-currents in $U$.
\end{defin}
Next we use the same definition as in \cite[Sec.\ 4.1.7]{F}.
\begin{defin}[Normal $k$-currents]
The subspace of \textbf{normal $k$-currents} $\mathcal{N}_k(U)\subset\D_k(U)$ is given by
\begin{equation*}
	\mathcal{N}_k(U)=\D_k(U)\cap\{\spt\textup{ compact}\}\cap\{\Mb+\Mb\circ\partial<\infty\}.
\end{equation*}
\end{defin}
The following definition can be found in \cite[Sec.\ 4.1.12]{F}.
\begin{defin}[Flat $k$-chains]
For compact $K\subset U$ let $\mathbb F_{k,K}(U)$ be the closure of $\N_k(U)\cap\{ \spt\subset K \}$ with respect to $|\cdot|_K^\flat$. Then the space of \textbf{flat $k$-chains} in $U$ is defined by
\begin{equation*}
\mathbb F_k(U)=\bigcup_{K}\mathbb F_{k,K}(U),
\end{equation*}
where the union is over all compact $K\subset U$.
\end{defin}
If $f:U\to V$ is a locally Lipschitz map, where $V\subset \R^m$ is open, then there is a natural notion of \textit{pushforward} $f_\#:\mathbb F_k(U)\to\mathbb F_k(V)$ \cite[Sec.\ 4.1.14]{F}. It is a linear map with $\spt(f_\# T)\subset f(\spt(T))$ and $\partial (f_\# T)=f_\#(\partial T)$ for all $T\in \mathbb F_k(U)$. Other properties relevant to our setting are summarized in \cref{FedererPushforward}.
\subsection{Transportation networks}
\label{TransportationNetworks}
In this section, we recall the definition of a transportation network and related notions from \cref{Intro,MainResults}. We will also consider the first variation with respect to the branched transport cost.
\begin{defin}[Transportation network]
	A \textbf{transportation network} is any $T=\theta\Ha\mres S\in\N_1(\R^n)\cap\rct_1(\R^n)$. We write $\T$ for the space of transportation networks.
\end{defin}
	Note that automatically $\Ha(S)<\infty$ because $T$ is compactly supported and $S$ locally $\Ha$-finite.
\begin{rem}[Interpretation of transportation network]
	Let $T\in\T$. Since $\Mb(\partial T)<\infty$, the $0$-current $\partial T$ can be represented by a Radon measure in $\M(\R^n)$ (Riesz representation theorem). The Hahn decomposition theorem implies $\partial T=\mu_+-\mu_-$ with $\mu_-,\mu_+\in\M_+(\R^n)$. By $\partial T(\R^n)=0$ we have $\mu_-(\R^n)=\mu_+(\R^n)$. Therefore, we can interpret $T$ as a \textit{mass flux} from the source $\mu_-$ to the sink $\mu_+$. Two important structure theorems can be found in \cite[Thm.\ 5.5]{Sil} and \cite[Thm.\ C]{S}.
\end{rem}
Now fix a transportation cost $\tau$ (recall \cref{TransportationCost}).
\begin{defin}[Branched transport cost]
	\label{BTC}
	Let $T=\theta\Ha\mres S\in\mathcal R_1(\R^n)$. The \textbf{branched transport cost} of $T$ inside $B\in\B(\R^n)$ is defined by
	\begin{equation*}
		|T|_\tau(B)=\int_{B\cap S}\tau(|\theta|)\e\Ha\in[0,\infty].
	\end{equation*}
\end{defin}
\begin{defin}[$\tau$-weight measure]
	\label{TauWeight}
	Let $T=\theta\Ha\mres S\in \rectn$. We write $|T|_\tau(\phi)=\int_S\phi\tau(|\theta|)\e\Ha$ for $\phi\in C(\R^n;\R_+)$. If $|T|_\tau(\R^n)<\infty$, then the \textbf{$\tau$-weight measure} $|T|_\tau\in\M_+(\R^n)$ is defined by
	\begin{equation*}
		|T|_\tau(\phi)=\int_S\phi\tau(|\theta|)\e\Ha
	\end{equation*}
	for all $\phi\in C_c(\R^n)$. If $\tau(m)=m$, then we write $|T|(\R^n)=|T|_\tau(\R^n)$ respectively $|T|(\phi)=|T|_\tau(\phi)$ if $|T|(\R^n)<\infty$ and $\phi\in C_c(\R^n)$, which just indicates the usual (total) variation.
\end{defin}
For the next statement, we will need the following definition.
\begin{defin}[{$(\mu,k)$-approximate differential \cite[Sec.\ 3.2.16]{F}}]
Let $X,Y$ be normed vector spaces, $\mu$ outer measure on $X$, and $A\subset X$. A map $f:A\to Y$ is called \textbf{$(\mu,k)$-approximately differentiable at $x\in A$} if there exist $r\in\R_+^*$ and map $g:B_r(x)\to Y$ such that $g$ is Fr\'echet differentiable in $x$ and $\Theta^k(\mu\mres (B_r(x)\cap\{ f\neq g \}),x)=0$. If $f$ is $(\mu,k)$-approximately differentiable at $x$, then we write
\begin{equation*}
	(\mu,k)\text{-}\apD f(x)=\Jm g(x)\bigg|_{\textup{Tan}^k(\mu,x)}.
\end{equation*}
\end{defin}
If $S\subset \R^n$ is countably $\Hak$-measurable with $\Hak(S)<\infty$, then any Lipschitz map $f:S\to\R^m$ is $(\Hak\mres S,k)$-approximately differentiable at $\Hak$-a.e.\ $x\in S$ \cite[Thm.\ 3.2.19]{F}. If $m\geq k$, then the \textit{$(\Hak\mres S,k)$-approximate Jacobian} is naturally defined by \cite[Cor.\ 3.2.20]{F}
\begin{equation*}
(\Hak\mres S,k)\text{-}\apJ f(x)=|((\Hak\mres S,k)\text{-}\apD f(x))_\#\xi|_{\textup{vol}}
\end{equation*}
for $\Hak$-a.e.\ $x\in S$, where $\#$ indicates the pushforward\footnote{If $L:\R^\ell\to \R^m$ is linear and $k\leq\min\{ \ell,m \}$, then $L_\#:\Lambda_k\R^\ell\to\Lambda_k\R^m$ is uniquely characterized by $L_\#(v_1\wedge\ldots\wedge v_k)=Lv_1\wedge\ldots\wedge v_k$ for all $v_1,\ldots,v_k\in\R^\ell$ \cite[Ch.\ 6, § 1]{LS}.} and $\xi\in\Lambda_k\textup{Tan}^k(\Hak\mres S,x)$ with $|\xi|_{\textup{vol}}=1$. If $k=1$, then we abbreviate $\apD f(x)=(\Ha\mres S,1)\text{-}\apD f(x)$ and $\apJ f(x)=(\Ha\mres S,1)\text{-}\apJ f(x)$ for $\Ha$-a.e.\ $x\in S$. 
We will frequently use the following statement, which is a straightforward implication of \cite[Prop.\ 4.1.30]{F}.
\begin{prop}[Transportation network properties preserved under locally Lipschitz pushforward]
	\label{FedererPushforward}
	Let $T=\theta\Ha\mres S\in \rectn$ be compactly supported. Further, suppose that $f:\R^n\to\R^n$ is locally Lipschitz. Then $f_\# T=\eta\Ha\mres f(S)\in \rectn$. In particular, if $T\in\T$, then $f_\# T\in\T$.
	More specifically, the following is satisfied.
	\begin{enumerate}
		\item For all $\omega\in\D^1(\R^n)$ we have
		\begin{equation*}
			\langle\omega, f_\# T\rangle=\int_S\langle \apD(f|_S)(x)\theta(x),\omega_{f(x)}\rangle\e\Ha(x).
		\end{equation*}
		\item There exist representatives $\bar\theta,\bar\eta$ such that\footnote{Summation takes into account that disconnected transportation paths in $T$ can be joined through application of $f_\#$. }
		\begin{equation*}
			\bar\eta(y)=\sum_{x\in(f|_S)^{-1}(y)}\apD(f|_S)(x)\bar\theta(x)/\apJ(f|_S)(x)
		\end{equation*}
		for $\Ha$-a.e.\ $y\in f(S)$. Further, we have\footnote{Divisions by $\apJ(f|_S)(x)$ ensure that $\bar\eta(y)$ has correct multiplicity $\sum_{x\in(f|_S)^{-1}(y)}(\pm)_x|\bar\theta(x)|$, where signs $(\pm)_x$ model possible cancellations due to opposite orientation.}
		\begin{equation*}
			\apD(f|_S)(x)\bar\theta(x)/\apJ(f|_S)(x)=\pm|\bar\theta(x)|\bar\eta(y)/|\bar\eta(y)|
		\end{equation*}
		if $x\in(f|_S)^{-1}(y)$ and $\bar\eta(y)\neq 0$ for $\Ha$-a.e.\ $y\in f(S)$.
	\end{enumerate}
\end{prop}
\begin{coro}[First variation of branched transport cost]
	\label{EnergyPushforward}
	Let $T\in\T$. Consider a vector field $g\in C_c^1(\R^n;\R^n)$ and set $f_t(x)=x+tg(x)$ for $t\in\R\setminus\{ 0\}$ and $x\in\R^n$.
	Then, for $|t|$ sufficiently small, we have
	\begin{equation*}
		|(f_t)_\# T|_\tau(\R^n)=\int_S\tau(|\theta|)\Jac f_t\e\Ha.
	\end{equation*}
	In particular, if $|T|_\tau(\R^n)<\infty$, then
	\begin{equation*}
		\partial_t|(f_t)_\# T|_\tau(\R^n)\bigg|_{t=0}=\int_S\tau(|\theta|)\textup{div}_S(g)\e\Ha.
	\end{equation*}
\end{coro}
\begin{proof}
	Note that $f_t$ is proper and $\textup{det}(\Jm f_t)\neq 0$ for $|t|$ sufficiently small, hence $f_t$ is a diffeomorphism for $|t|$ sufficiently small (Hadamard\textendash Cacciopoli theorem).
	By the first bullet point in \cref{FedererPushforward}, the area formula \cite[Cor.\ 3.2.20]{F}, and the second bullet point in \cref{FedererPushforward} we get
	\begin{equation*}
	|(f_t)_\# T|_\tau(\R^n)=\int_{f_t(S)}\tau(|\eta|)\e\Ha=\int_S\tau(|\eta(f_t(x))|)\Jac f_t\e\Ha=\int_S\tau(|\theta|)\Jac f_t\e\Ha,
	\end{equation*}
	which shows the first statement.
	We have $\Jac f_t=|e+t\mathrm{D}ge|$, where $e(x)\in T_xS$ for $\Ha$-a.e.\ $x\in S$. Thus
	\begin{equation*}
		\partial_t\Jac f_t\bigg|_{t=0}=\left(\frac{(e+t\mathrm{D}ge)^\intercal}{|e+t\mathrm{D}ge|}\mathrm{D}ge\right)\bigg|_{t=0}=\textup{div}_S(g)
	\end{equation*}
	$\Ha$-a.e.\ on $ S$.
	The second statement follows from the Leibniz integral rule:
	Since $|T|_\tau(\R^n)<\infty$ and $\Jac f_t$ is bounded, we have $\tau(|\theta|)|e+t\mathrm{D}ge|\in L^1(\Ha\mres S;\R)$ for all $t\in\R$.
	Further, for $|t|$ sufficiently small, we get $|\partial_t\tau(|\theta|)\Jac f_t|\leq \tau(|\theta|)(1+|\mathrm{D}g |_2)\in L^1(\Ha\mres S;\R)$. Therefore
	\begin{equation*}
		\partial_t|(f_t)_\# T|_\tau(\R^n)\bigg|_{t=0}=\int_S\partial_t(\tau(|\theta|)\Jac f_t)\e\Ha\bigg|_{t=0}=\int_S\tau(|\theta|)\textup{div}_S(g)\e\Ha.\qedhere
	\end{equation*}
\end{proof}
\Cref{EnergyPushforward} motivates the following definition.
\begin{defin}[$\tau$-variation]
	\label{TauVariation}
	Let $T\in\T$ with $|T|_\tau(\R^n)<\infty$. The \textbf{$\tau$-variation} $\delta_\tau T$ of $T$ is defined by
	\begin{equation*}
		\delta_\tau T(g)=\int_S\textup{div}_S(g)\tau(|\theta|)\e\Ha
	\end{equation*}
	for all $g\in C_c^1(\R^n;\R^n)$. 
\end{defin}
\begin{defin}[{$\phi$-weighted $\tau$-variation, cf.\ \cite[Sec.\ 4.9]{A72}}]
	\label{phiWeightedTauVariation}
	Let $T=\theta\Ha\mres S\in\rectn$ with $|T|_\tau(\R^n)<\infty$ and $\phi\in C^1(\R^n;\R_+)$. The \textbf{$\phi$-weighted $\tau$-variation} of $T$ is defined by
	\begin{equation*}
		\delta_\tau (T,\phi)(g)=\int_S(\phi\textup{div}_S(g)+\nabla\phi\cdot g)\tau(|\theta|)\e\Ha
	\end{equation*}
	for all $g\in C_c^1(\R^n;\R^n)$.
\end{defin}
Finally, we will need the following definitions, cf.\ \cite[Sec.\ 4.3]{B78} and \cite[Sec.\ 4.7]{KT17}.
\begin{defin}[Smoothed $\tau$-variation and $\tau$-weight measure]
	\label{Smoothed}
	Let $T=\theta\Ha\mres S\in\rectn$ with $|T|_\tau(\R^n)<\infty$ and $\Psi\in C_c^\infty(\R^n)$. Define
	\begin{equation*}
		(\Psi\ast\delta_\tau T)(y)=\int_ST_xS\nabla\Psi(x-y)\tau(|\theta(x)|)\e\Ha(x)\qquad\textup{and}\qquad (\Psi\ast |T|_\tau)(y)=\int\Psi(x-y)\e|T|_\tau(x)
	\end{equation*}
	for all $y\in\R^n$.
\end{defin}
It is easy to check that $\Psi\ast\delta_\tau T\in C^\infty(\R^n;\R^n)$ and $\Psi\ast |T|_\tau\in C^\infty(\R^n)$ and one can differentiate under the integrals (Leibniz integral rule). 
Further, the definitions are such that
\begin{equation*}
	\delta_\tau T(\Psi\ast g)=\int g\cdot (\Psi\ast\delta_\tau T)\e\Leb^n\qquad\textup{and}\qquad (\Psi\ast |T|_\tau)(\phi)=\int\phi(\Psi\ast |T|_\tau)\e\Leb^n
\end{equation*}
for all $g\in C_c^1(\R^n;\R^n)$ and $\phi\in C_c(\R^n)$, where the left-hand side in the second equation is defined by $(\Psi\ast |T|_\tau)(\phi)= |T|_\tau(\Psi\ast\phi)$.
\subsection{Varifolds}
\label{Varifolds}
Varifolds were introduced by Almgren \cite{A65}. A comprehensive treatment of varifolds can be found in Allard's work \cite{A72}.
Let $U\subset\R^n$ be open.
Recall from \cref{Basic} that $\mathcal G_k(U)=\R^n\times G(n,k)$.
\begin{defin}[{$k$-varifold and weight measure \cite[Def.\ 3.1]{A72}}]
A \textbf{$k$-varifold} in $U$ is any $[0,\infty]$-valued Radon measure on $\mathcal G_k(U)$. We write $\mathcal V_k(U)$ for the set of $k$-varifolds in $U$.
The \textbf{weight measure} $\| V\|:\B(U)\to[0,\infty]$ of $V\in\mathcal V_k(U)$ is defined by
\begin{equation*}
	\langle \phi, \| V\|\rangle =\int \phi(x)\e V(x,P)
\end{equation*}
for all $\phi\in C_c(U)$.
\end{defin}
Hence $\| V\|$ is just the pushforward of $V$ under the natural projection map $p_U:U\times G(n,k)\to U$. Note that, if $C_c(U)$ is equipped with the uniform norm, then $\| V\|$ may not be an element of the topological dual because we may have $\| V\|(U)=\infty$. However, if one takes the canonical LC topology, then $\| V\|$ is always an element of the topological dual which justifies the bracket notation.
\begin{defin}[$k$-varifold induced by rectifiable $k$-current]
\label{inducedV}
Let $T=\theta\Hak\mres S\in\mathcal R_k(U)$. Then the \textbf{induced varifold} $V=v(|\theta|, S)\in \mathcal V_k(U)$ by $T$ is given by
\begin{equation*}
\langle\phi,V\rangle=\int_S \phi(x,\textup{Tan}^k(\Hak\mres S,x))|\theta(x)|\Hak(x)
\end{equation*}
for all $\phi\in C_c(\mathcal G_k(U))$.
\end{defin}
\subsection{Admissible deformations and functions}
In this section, we introduce classes of admissible (test) functions and deformations. The former are precisely as in \cite[Sec.\ 4.4]{KT17}.
\begin{notat}[{Admissible functions, cf.\ \cite[Sec.\ 4.1]{B78} and \cite[Sec.\ 4.4]{KT17}}]
	We set
	\begin{equation*}
		\mathcal F_j^+=\{ \phi\in C^2(\R^n;[0,1])\:|\:|\nabla\phi|\leq j\phi\textup{ and }|\nabla^2\phi|_2\leq j\phi\textup{ in }\R^n \}
	\end{equation*}
	and
	\begin{equation*}
		\mathcal F_j^n=\{ g\in C^2(\R^n;\R^n)\:|\:|g|\leq j,|\Jm g|_2\leq j,|\Jm^2g|_{2,2}\leq j\textup{, and }|g|_{L^2}\leq j \},
	\end{equation*}
	where $|g|_{L^2}=\left(\int|g|^2\e\Lebn\right)^{1/2}$.
\end{notat}
If $\phi\in \mathcal F_j^+$, then $\phi>0$ in $\R^n$ or $\phi\equiv 0$, which can be read off from the following lemma.
\begin{lem}[{Properties of functions in $\mathcal F_j^+$, see \cite[Sec.\ 4.2]{B78} and \cite[Lem.\ 4.6]{KT17}}]
	\label{propertiesF}
	Let $\phi\in\mathcal F_j^+$. Then
	\begin{multline*}
		\phi(x)\leq\phi(y)e^{j|x-y|},\qquad|\phi(x)-\phi(y)|\leq j|x-y|\phi(x)e^{j|x-y|},\qquad\textup{and}\\|\phi(x)-\phi(y)-\nabla\phi(y)\cdot(x-y)|\leq j|x-y|^2\phi(y)e^{j|x-y|}
	\end{multline*}
	for all $x,y\in\R^n$.
\end{lem} 
\begin{lem}[Elements of $\Jm\mathcal F_j^n$ are $j$-Lipschitz]
	\label{propertiesFn}
	Let $g\in\mathcal F_j^n$. Then
	\begin{equation*}
		|\Jm g(x)-\Jm g(y)|_2\leq j|x-y|
	\end{equation*}
	for all $x,y\in\R^n$.
\end{lem} 
\begin{proof}
	We have (recall that $(Mx)_{i,j}=\sum_kM_{i,j,k}x_k$ for all $M\in\R^{n\times n\times n}$)
	\begin{equation*}
		\Jm g(x)-\Jm g(y)=\int_{[0,1]}\partial_t(\Jm g(y+t(x-y)))\e\Leb(t)=\int_{[0,1]}\Jm^2g(y+t(x-y))(x-y)\e\Leb(t). 
	\end{equation*}
	Using this one easily estimates
	\begin{equation*}
		|\Jm g(x)-\Jm g(y)|_2\leq\int_{[0,1]}|\Jm^2g(y+t(x-y))|_{2,2}\e\Leb(t)|x-y|\leq j|x-y|
	\end{equation*}
	because $|\Jm^2g|_{2,2}\leq j$.
\end{proof}
Next, fix $0<\delta\ll 1$ and consider the \textit{homotopy} $h_f:(-\delta,1+\delta)\times\R^n\to\R^n$ from $\textup{id}$ to $f:\R^n\to\R^n$ defined by (cf.\ \cite[Sec.\ 4.1.9]{F})
\begin{equation*}
	h_f(t,x)=(1-t)x+tf(x).
\end{equation*}
The following definition is a modification of the admissible classes in \cite[Sec.\ 4.9]{B78} and \cite[Def.\ 4.8]{KT17}.
\begin{defin}[Admissible deformations $\adm_j(T),\adm_j(C,T)$ and $\Delta_j|T|_\tau(A)$]
	\label{AdmDeform}
	Fix a transportation network $T\in\T$ and $j\in\mathbb{N}$. Assume that $|T|_\tau(\R^n)<\infty$. Then we write $\adm_j(T)$ for the set of \textbf{admissible deformations} $f:\R^n\to\R^n$ which are required to satisfy (recall \cref{ProdCurrents})
	\begin{itemize}
		\item $f$ locally Lipschitz with $|f-\textup{id}|_\infty\leq j^{-2}$,
		\item $\Mb\left((h_f)_\#(\llbracket 0,1\rrbracket\times T)\right)\leq j^{-1}(|T|_\tau(\R^n)-|f_\# T|_\tau(\R^n))$ and $|f_\# T|_\tau(\phi)\leq|T|_\tau(\phi)$ for all $\phi\in \mathcal F_j^+$.
	\end{itemize}
	If $B\subset \R^n$ is closed, then we also write
	\begin{equation*}
		\Delta_j|T|_\tau(B)=\inf_f|f_\#T|_\tau(B)-|T|_\tau(B)\in\R_-,
	\end{equation*}
	where the infimum is over $f$ in
	\begin{equation*}
		\adm_j(B,T)=\{ f\in\adm_j(T)\:|\:\{ x\:|\:f(x)\neq x \}\cup\{ f(x)\:|\:f(x)\neq x \}\subset B\}.
	\end{equation*}
\end{defin}
Note that $|T|_\tau(\R^n)<\infty$ implies $|f_\#T|_\tau(\R^n)<\infty$ for $f\in \adm_j(T)$ because $\phi\equiv 1\in\mathcal F_j^+$.
\begin{figure}
	\centering
	\begin{tikzpicture}[scale=2, transform shape=false]
		\draw[dashed] (0,0) circle (2);
		\draw[dashed] (-2,0) -- (0,0);
		
		\draw (0,0) (-0.07,-0.07) -- (0.07,0.07);
		\draw (0,0) (-0.07, 0.07) -- (0.07,-0.07);
		
		\filldraw[draw=lightgray, fill=lightgray] ({-1},{sqrt(3)}) -- (2,0) -- ({-1},{-sqrt(3)}) -- (1,0);
		\draw[gray,->,line width=1] (1,0.4) arc (0:300:0.1);
		\draw[gray,->,line width=1] (1,-0.4) arc (300:0:0.1);
		
		\draw  ({-1},{sqrt(3)}) -- (2,0);
		\draw  ({-1},{-sqrt(3)}) -- (2,0);
		\draw[line width=2]  ({-1},{sqrt(3)}) -- (1,0);
		\draw[line width=2]  ({-1},{-sqrt(3)}) -- (1,0);
		\draw[line width=2]  (1,0) -- (2,0);
		
		\draw[->]  ({-1},{sqrt(3)}) -- ++({sqrt(3)/4},{-1/4});
		\draw[->]  ({-1},{-sqrt(3)}) -- ++({sqrt(3)/4},{1/4});
		\draw[->,line width=2]  (1,0) -- ++(0.5,0);
		\draw[->,line width=2]  ({-1},{sqrt(3)}) -- ++({1/sqrt(7)},{-sqrt(3)/(2*sqrt(7))});
		\draw[->,line width=2]  ({-1},{-sqrt(3)}) -- ++({1/sqrt(7)},{sqrt(3)/(2*sqrt(7))});
		
		\node[fill,circle,minimum size=3mm,inner sep=0pt] at (2,0) {};
		\node[white,fill,circle,minimum size=2.5mm,inner sep=0pt] at (2,0) {};
		\node[fill,circle,minimum size=3mm,inner sep=0pt] at (1,0) {};
		\node[white,fill,circle,minimum size=2.5mm,inner sep=0pt] at (1,0) {};
		
		\node[fill,circle,minimum size=3mm,inner sep=0pt] at ({-1},{sqrt(3)}) {};
		\node[fill,circle,minimum size=3mm,inner sep=0pt] at ({-1},{-sqrt(3)}) {};
		
		\node at (-1,-0.2) {$r=2j^{-2}$};
		
		\node at (2.2,0) {$y$};
		\node at (0.8,0) {$\tilde y$};
		\node at ({-1.2},{sqrt(3)}) {$x_1$};
		\node at ({-1.2},{-sqrt(3)}) {$x_2$};
		
		\node at (0.6,1) {$T$};
		\node at (-0.5,1) {$\boldmath{f_\# T}$};
		\node at (0.7,-1.3) {$\textcolor{gray}{(h_f)_\#(\llbracket 0,1\rrbracket\times T)}$};
		
	\end{tikzpicture}
	\caption{Admissible deformation $f$ from \cref{RelocTriple}.}
	\label{fig2}
\end{figure}
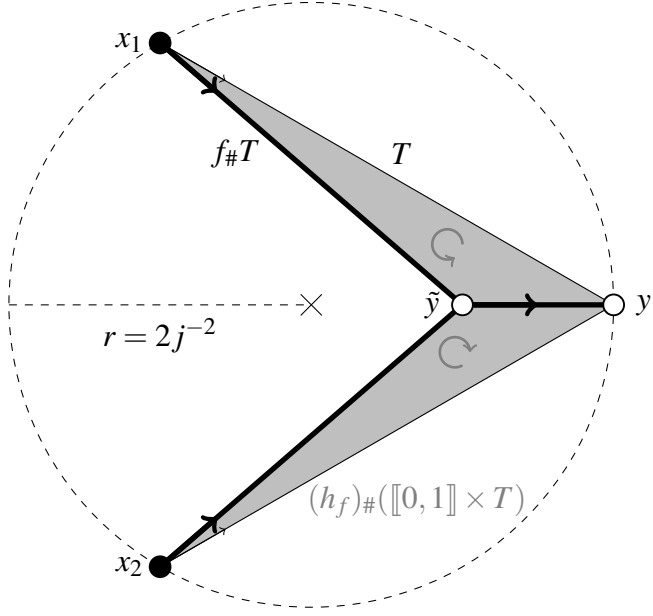
\begin{examp}[Relocation of triple junction]
\label{RelocTriple}
Let $j\in\mathbb N$, $r=2j^{-2}$, and $x_k=r(\cos(\alpha_k),\sin(\alpha_k))$ for $\alpha_k=2k\pi/3$ and $k=1,2$. 
Further, set $y=(r,0)$ and $\tilde y=(r/2,0)$.
Now define $K=\textup{conv}(\{ x_1,x_2,y \})$ and $\widetilde K=\textup{conv}(\{ x_1,x_2,\tilde y \})$.
Note that $\widetilde K\subset K$.
Assume that $T\in\T$ is given by $T=\vec{e}_1\Ha\mres e_1+\vec{e}_2\Ha\mres e_2$, where $e_i=[x_i,y]$ and $\vec{e}_i=(y-x_i)/|y-x_i|$ for $i=1,2$.
See \cref{fig2} for a sketch.
It may also be seen as the local picture of a transportation network inside a closed ball of radius $r$.
Let the transportation cost $\tau$ be given by $\tau(m)=m^\alpha$ for some $\alpha\in(0,\log_2(4\sqrt{3}-2\sqrt{7}))$.
We will define a locally Lipschitz map $f:\R^n\to\R^n$ such that $f\in\mathcal A_j(T)$ for $j$ sufficiently large.
It is enough to specify the definition of $f$ on $A=\textup{clos}(K\setminus \widetilde K)$ since $\spt(T)\subset A$.
Let $f|_A$ be given as the orthogonal projection onto $\widetilde K$.
Then we may extend $f$ to locally Lipschitz $f:\R^n\to\R^n$ such that $|f-\textup{id}|_\infty\leq j^{-2}$.
The homotopy formula for currents \cite[p.\ 363]{F} implies
\begin{equation*}
f_\# T-T=\partial (h_f)_\#(\llbracket 0,1\rrbracket\times T)+(h_f)_\#(\llbracket 0,1\rrbracket\times \partial T).
\end{equation*}
In \cref{fig2}, the current $(h_f)_\#(\llbracket 0,1\rrbracket\times T)$ is highlighted in gray.
Note that its boundary contains the term $-(h_f)_\#(\llbracket 0,1\rrbracket\times \partial T)$ by definition of $f$.
Next, let us show that $f$ is admissible for $j$ sufficiently large.
The mass of $(h_f)_\#(\llbracket 0,1\rrbracket\times T)$ can be calculated as the area of the parallelogram induced by $[\tilde y,y]$ and $[x_1,\tilde y]$,
\begin{equation*}
\Mb\left((h_f)_\#(\llbracket 0,1\rrbracket\times T)\right)=|\tilde y-y|r\sin\left( \frac{2\pi}{3} \right)=j^{-4}\sqrt{3}.
\end{equation*}
Moreover, we have 
\begin{multline*}
|T|_\tau(\R^n)=2\tau(1)\Ha([x_1,y])=j^{-2}4\sqrt{3}\qquad\textup{and}\qquad |f_\#T|_\tau(\R^n)=2\tau(1)\Ha([x_1,\tilde y])+\tau(2)\Ha([\tilde y,y])\\
=j^{-2}4\frac{\sqrt{7}}{2}+j^{-2}4\frac{2^\alpha}{4}.
\end{multline*}
Thus, the inequality
\begin{equation*}
\Mb\left((h_f)_\#(\llbracket 0,1\rrbracket\times T)\right)\leq j^{-1}(|T|_\tau(\R^n)-|f_\# T|_\tau(\R^n))\qquad\Leftrightarrow\qquad j^{-1}\frac{\sqrt{3}}{4}\leq \sqrt{3}-\frac{\sqrt{7}}{2}-\frac{2^\alpha}{4}
\end{equation*}
is true for $j$ sufficiently large---parameter $\alpha<\log_2(4\sqrt{3}-2\sqrt{7})$ ensures that the right-hand side is positive, otherwise $f_\#$ would increase the branched transport cost.
The triple junction $\tilde y$ is optimally placed if 
\begin{equation*}
\frac{\tilde y-x_1}{|\tilde y-x_1|}+\frac{\tilde y-x_2}{|\tilde y-x_2|}=2^\alpha\frac{y-\tilde y}{|y-\tilde y|}.
\end{equation*}
This follows from the well-known ``momentum conservation law'' \cite[p.\ 255]{X}.
This equality is satisfied for $\alpha=\log_2(4/\sqrt{7})$.
Otherwise, we have
\begin{equation*}
\begin{cases*}
\textup{the optimal triple junction lies on $(0,\tilde y)$}&if $\alpha<\log_2(4/\sqrt{7})$,\\
\textup{the optimal triple junction lies on $(\tilde y,y)$}&if $\alpha\in(\log_2(4/\sqrt{7}),\log_2(4\sqrt{3}-2\sqrt{7}))$,\\
\textup{the junction $y$ is already optimal (locally)}&if $\alpha\in[\log_2(4\sqrt{3}-2\sqrt{7}),1)$.
\end{cases*}
\end{equation*}
In the limit case $\alpha=0$ respectively $\tau(m)=1_{\R_+^*}(m)$, the optimal triple junction would be equal to the centre of the circle (Steiner tree).
Finally, we need to show $|f_\# T|_\tau(\phi)\leq|T|_\tau(\phi)$ for all $\phi\in \mathcal F_j^+$.
We use a trick similar to \cite[Sec.\ 4.10]{B78}, see also \cite[Lem.\ 4.12]{KT17},
\begin{equation*}
|f_\# T|_\tau(\phi)-|T|_\tau(\phi)\leq \max_A\phi|f_\# T|_\tau(1_A)-\min_A\phi |T|_\tau(1_A)\leq \min_A\phi e^{1/j}|f_\# T|_\tau(1_A)-\min_A\phi |T|_\tau(1_A)\leq 0
\end{equation*}
for $j$ sufficiently large (independent of $\phi$), where \cref{propertiesF} was used in the second inequality.
This shows $f\in\mathcal A_j(T)$ for $\alpha\in (0,\log_2(4\sqrt{3}-2\sqrt{7}))$ and $j$ sufficiently large.
\end{examp}
\subsection{Approximate $\tau$-curvature vector fields}
\label{ApproxTauVectorfields}
We define $\Psi_\varepsilon=c_\varepsilon\psi\widehat\Psi_\varepsilon\in C_c^\infty(\R^n)$ as in \cite[Sec.\ 4.6]{KT17}, which is different from the kernel in \cite[Sec.\ 4.3]{B78},
\begin{itemize}
	\item $\psi\in C_c^\infty(\R^n;[0,1])$ is a \textit{cut-off function} with $\psi\equiv 1$ in $\overline{B}_{1/2}(0)$, $\psi\equiv 0$ in $\R^n\setminus \overline{B}_1(0)$, and 
	\begin{equation*}
	\max(|\nabla\psi|,3^{-1}|\nabla^2\psi|_2)\leq 3
	\end{equation*}
	in $\R^n$,
	\item $\widehat\Psi_\varepsilon\in C^\infty(\R^n)$ is a \textit{Gaussian} given by
	\begin{equation*}
		\widehat\Psi_\varepsilon(x)=\frac{e^{-|x|^2/(2\varepsilon^2)}}{(2\pi\varepsilon^2)^{n/2}},
	\end{equation*}
	\item and $c_\varepsilon$ is a \textit{normalizing constant} such that $\int\Psi_{\varepsilon}\e\Lebn=1$.
\end{itemize}
It satisfies the following estimates \cite[Lem.\ 4.13]{KT17} (the proof does not depend on the codimension).
\begin{lem}[Estimates related to $\nabla \Psi_\varepsilon,\nabla^2\Psi_\varepsilon$]
	\label{EstimatesNabla}
	There exists $c\in\R_+^*$ such that
	\begin{equation*}
		|\nabla\Psi_\varepsilon(x)|\leq\frac{|x|}{\varepsilon^2}\Psi_\varepsilon(x)+c 1_{\overline B_1(0)\setminus\overline B_{1/2}(0)}(x)e^{-1/\varepsilon}
	\end{equation*}
	and
	\begin{equation*}
		|\nabla^2\Psi_\varepsilon(x)|_2\leq\frac{|x|^2}{\varepsilon^4}\Psi_\varepsilon(x)+\frac{c}{\varepsilon^2}\Psi_\varepsilon(x)+c 1_{\overline B_1(0)\setminus\overline B_{1/2}(0)}(x)e^{-1/\varepsilon}
	\end{equation*}
	for all $x\in\R^n$ for $\varepsilon\in\R_+^*$ sufficiently small.
\end{lem}
\begin{defin}[{Approximate $\tau$-curvature vector field, cf.\ \cite[Sec.\ 4.3]{B78} and \cite[Sec.\ 5]{KT17}}]
	\label{ApproxCurveVF}
	Let $\varepsilon\in\R_+^*$ and $T=\theta\Ha\mres S\in\rectn$ with $|T|_\tau(\R^n)<\infty$. We write
	\begin{equation*}
		h_{\tau,\varepsilon}(T)=-\Psi_\varepsilon\ast\left( \frac{\Psi_\varepsilon\ast\delta_\tau T}{\Psi_\varepsilon\ast|T|_\tau+\varepsilon} \right)\in C^\infty(\R^n;\R^n)
	\end{equation*}
	and
	\begin{equation*}
		\tilde h_{\tau,\varepsilon}(T)=- \frac{\Psi_\varepsilon\ast\delta_\tau T}{\Psi_\varepsilon\ast|T|_\tau+\varepsilon} \in C^\infty(\R^n;\R^n).
	\end{equation*}
\end{defin}
Note that $ h_{\tau,\varepsilon}(T)$ has compact support if $T$ has compact support.
By the proof of \cite[Lem.\ 5.1]{KT17} we obtain the following statement, cf.\ \cite[Sec.\ 4.4]{B78}. 
\begin{lem}[Estimates related to $\tilde h_{\tau,\varepsilon}(T),\Jm^kh_{\tau,\varepsilon}(T)$]
	\label{hEstimates}
	Let $M\in\R_+^*$. Then there exists $\varepsilon_1=\varepsilon_1(M,n)\in(0,1)$ such that
	\begin{equation*}
		|\tilde h_{\tau,\varepsilon}(T)|\leq 2\varepsilon^{-2},\quad |h_{\tau,\varepsilon}(T)|\leq 2\varepsilon^{-2},\quad |\Jm h_{\tau,\varepsilon}(T)|_2\leq 4\varepsilon^{-4},\quad\textup{and}\quad|\Jm^2 h_{\tau,\varepsilon}(T)|_{2,2}\leq 4\varepsilon^{-6}
	\end{equation*}
	for all $\varepsilon\in(0,\varepsilon_1)$ and $T\in\rectn$ with $|T|_\tau(\R^n)\leq M$.
\end{lem}
 The proof essentially relies on the estimates in \cite[Lem.\ 4.13]{KT17} respectively \cref{EstimatesNabla}. It can be adapted to our setting with codimension equal to $n-1$. The idea is to apply the proof of \cite[Lem.\ 5.1]{KT17} to $V=v(\tau(|\theta|),S)\in\V_1(\R^n)$, where $T=\theta\Ha\mres S$ is as in \cref{hEstimates}. 
 Note that
\begin{equation*}
	\| V\|=|T|_\tau,\qquad \Psi_\varepsilon\ast\delta V=\Psi_\varepsilon\ast\delta_\tau T,\qquad\textup{and}\qquad \Psi_\varepsilon\ast\| V\|=\Psi_\varepsilon\ast|T|_\tau.
\end{equation*}
Here $\delta V$ denotes the first variation of $V$ \cite[Def.\ 4.2]{A72}, $\Psi_\varepsilon\ast\delta V$ and $\Psi_\varepsilon\ast\| V\|$ are defined analogous to \cite[Sec.\ 4.3]{B78}.
In particular, these relations will be frequently used in \cref{Auxresults}.
We also have $\delta_\tau (T,\phi)=\delta (V,\phi)(g)$ for $\phi\in C^1(\R^n;\R_+)$ and $g\in C_c^1(\R^n;\R^n)$, where $\delta (V,\phi)(g)$ is as in \cite[Sec.\ 4.9]{A72}. 
\subsection{Estimates related to approximate $\tau$-curvature vector fields}
\label{Auxresults}
In this section, we list some auxiliary results whose proofs follow \cite[Sec.\ 5]{KT17}, cf.\ \cite[Sec.\ 4]{B78}.
Essentially, we employ the relations at the end of \cref{ApproxTauVectorfields} and note that the proofs do not depend on the codimension.
The following lemma will be used to show \cref{Prop5.3,LongProof}, see \cite[Lem.\ 4.17, (4.38)]{KT17} (note that we have $=$ instead of $\leq$).
\begin{lem}[Expression for $\int |T|_\tau(\overline{B}_r(\cdot))$]
	\label{intBall}
	Let $T\in\rectn$ with $|T|_\tau(\R^n)<\infty$. Then
	\begin{equation*}
		\int |T|_\tau(\overline{B}_r(x))\e\Leb^n(x)=\omega_nr^n|T|_\tau(\R^n)
	\end{equation*}
	for all $r\in\R_+^*$.
\end{lem}
\begin{proof}
Using Fubini's theorem and $1_{\overline{B}_r(x)}(y)=1_{\overline{B}_r(y)}(x)$, we get
\begin{equation*}
\int |T|_\tau(\overline{B}_r(x))\e\Leb^n(x)=\int\int 1_{\overline{B}_r(x)}(y)\e\Lebn(x)\e|T|_\tau(y)=\omega_nr^n|T|_\tau(\R^n).\qedhere
\end{equation*}
\end{proof}
The following statement corresponds to \cite[Prop.\ 5.3]{KT17}, compare with \cite[Sec.\ 4.6]{B78}.
\begin{prop}
	\label{Prop5.3}
	Let $M\in\R_+^*$. Then, for $\varepsilon\in\R_+^*$ sufficiently small (depending on $M,n$), we have
	\begin{equation*}
		\left|\int h_{\tau,\varepsilon}(T)\cdot g\e|T|_\tau+\int(\Psi_\varepsilon\ast\delta_\tau T)\cdot g\e\Lebn\right|\leq\varepsilon^{1/4}\left(\int\frac{|\Psi_\varepsilon\ast\delta_\tau T|^2}{\Psi_\varepsilon\ast|T|_\tau+\varepsilon}\e\Lebn\right)^{1/2}
	\end{equation*}
	for all $j\leq 2^{-1}\varepsilon^{-1/6},g\in\mathcal F_j^n$, and $T\in \rectn$ compactly supported with $|T|_\tau(\R^n)\leq M$.
\end{prop}
\begin{proof}
	Abbreviate $h=h_{\tau,\varepsilon}(T),\tilde h=\tilde h_{\tau,\varepsilon}(T)$, and 
	\begin{equation*}
		I=\int\frac{|\Psi_\varepsilon\ast\delta_\tau T|^2}{\Psi_\varepsilon\ast|T|_\tau+\varepsilon}\e\Lebn.
	\end{equation*}
	By a calculation similar to \cite[(5.10) \&\ (5.11)]{KT17} we get the analogue of \cite[(5.12)]{KT17},
	\begin{multline}
		\label[ineq]{in1:a}
		\left|\int h_{\tau,\varepsilon}\cdot g\e|T|_\tau+\int(\Psi_\varepsilon\ast\delta_\tau T)\cdot g\e\Lebn\right|\\
		\leq \underbrace{\int|g||\tilde{h}|\varepsilon\Lebn}_{I_1=}+\underbrace{\int|\tilde{h}(y)|\left|\int\Psi_\varepsilon(x-y)g(x)\e|T|_\tau(x)-(\Psi_\varepsilon\ast|T|_\tau)(y)g(y)\right|\e\Lebn(y)}_{I_2=}.
	\end{multline}
	We also get inequalities corresponding to \cite[(5.13) \&\ (5.14)]{KT17} by using the definition of $\mathcal F_j^n$,
	\begin{equation}
		\label[ineq]{in1:b}
		I_1\leq j\varepsilon|\tilde h|_{L^2}\leq j\varepsilon^{1/2}I^{1/2}.
	\end{equation}
	Instead of \cite[Lem.\ 4.7]{KT17}, we observe that $g$ is $j$-Lipschitz. Using this and the definition of $\Psi_\varepsilon$, we get the analogue of \cite[(5.15)]{KT17},
	\begin{equation}
		\label[ineq]{in1:B}
		\left|\int\Psi_\varepsilon(x-y)g(x)\e|T|_\tau(x)-(\Psi_\varepsilon\ast|T|_\tau)(y)g(y)\right|\leq j\int_{\overline{B}_1(y)}|x-y|\Psi_\varepsilon(x-y)\e|T|_\tau(x).
	\end{equation}
	In \cite[(5.16)]{KT17} we need to replace $n+1$ by $n$,
	\begin{equation}
		\label[ineq]{in1:B2}
		\sup_{\overline{B}_1(y)\setminus\overline{B}_{\sqrt \varepsilon}(y)}|x-y|\Psi_\varepsilon(x-y)\leq c(n)\varepsilon^{-n}e^{-1/(2\varepsilon)}=c_{\varepsilon,n}.
	\end{equation}
	Combining \cref{in1:B,in1:B2}, we obtain
	\begin{equation}
		\label[ineq]{in1:c}
		I_2\leq\underbrace{j\varepsilon^{1/2}\int|\tilde h|(\Psi_\varepsilon\ast|T|_\tau)\e\Lebn}_{I_{2,1}=}+\underbrace{jc_{\varepsilon,n}\int|\tilde{h}(y)||T|_\tau(\overline{B}_1(y))\e\Lebn(y)}_{I_{2,2}=}.
	\end{equation}
	Note that $(\Psi_\varepsilon\ast|T|_\tau)(\R^n)=|T|_\tau(\Psi_\varepsilon\ast 1)=|T|_\tau(\R^n)$. Using this and Hölder's inequality, one obtains (see \cite[(5.18) \&\ (5.19)]{KT17})
	\begin{equation}
		\label[ineq]{in1:d}
		I_{2,1}\leq j\varepsilon^{1/2}\left( \int|\tilde{h}|^2(\Psi_\varepsilon\ast |T|_\tau)\e\Lebn \right)^{1/2}(\Psi_\varepsilon\ast|T|_\tau)(\R^n)^{1/2}\leq j\varepsilon^{1/2}I^{1/2}M^{1/2}.
	\end{equation}
	Using Hölder's inequality again, we can estimate (see \cite[(5.20) \&\ (5.21)]{KT17})
	\begin{equation*}
		I_{2,2}\leq jc_{\varepsilon,n} I^{1/2}\varepsilon^{-1/2}M^{1/2}\left(\int|T|_\tau(\overline{B}_1(x))\e\Lebn(x)\right)^{1/2}.
	\end{equation*}
	Now recall that $\int|T|_\tau(\overline{B}_1(x))\e\Lebn(x)=\omega_n|T|_\tau(\R^n)\leq \omega_nM$ (\cref{intBall}), which yields
	\begin{equation}
		\label[ineq]{in1:e}
		I_{2,2}\leq jc_{\varepsilon,n} \varepsilon^{-1/2}M\omega_n^{1/2}I^{1/2}.
	\end{equation}
	Now combine \cref{in1:a,in1:b,in1:c,in1:d,in1:e} and use the assumption $j\leq 2^{-1}\varepsilon^{-1/6}$,
	\begin{multline*}
		\left|\int h_{\tau,\varepsilon}\cdot g\e|T|_\tau+\int(\Psi_\varepsilon\ast\delta_\tau T)\cdot g\e\Lebn\right|\leq I_1+I_{2,1}+I_{2,2}\\
		\leq j\varepsilon^{1/2}I^{1/2}+j\varepsilon^{1/2}I^{1/2}M^{1/2}+jc_{\varepsilon,n}\varepsilon^{-1/2}M\omega_n^{1/2}I^{1/2}
		\leq 2^{-1}\varepsilon^{1/3}(1+M^{1/2}+c_{\varepsilon,n}\varepsilon^{-1}M\omega_n^{1/2})I^{1/2}\\
		\leq \varepsilon^{1/4}I^{1/2}
	\end{multline*}
	for $\varepsilon$ sufficiently small because $c_{\varepsilon,n}\varepsilon^{-1}\to 0$ for $\varepsilon\to 0$.
\end{proof}
The next proposition is the analogue of \cite[Prop.\ 5.5]{KT17}, cf.\ \cite[Sec.\ 4.8]{B78}.
\begin{prop}
	\label{Prop5.5}
	Let $M\in\R_+^*$. Then, for $\varepsilon\in\R_+^*$ sufficiently small (depending on $M,n$), we have
	\begin{equation*}
		\left|\int h_{\tau,\varepsilon}(T)\cdot g\e|T|_\tau+\partial_\tau T(g)\right|\leq\varepsilon^{1/4}+\varepsilon^{1/4}\left(\int\frac{|\Psi_\varepsilon\ast\delta_\tau T|^2}{\Psi_\varepsilon\ast|T|_\tau+\varepsilon}\e\Lebn\right)^{1/2}
	\end{equation*}
	for all $j\leq 2^{-1}\varepsilon^{-1/6},g\in\mathcal F_j^n$, and $T\in \rectn$ compactly supported with $|T|_\tau(\R^n)\leq M$. (We write $\delta_\tau T(g)=\int_S\tau(|\theta|)\textup{div}_S(g)\e\Ha$, which is well-defined because $|\Jm g|_2\leq j$.)
\end{prop}
\begin{proof}
	Instead of \cite[(5.49)]{KT17}, we have
	\begin{equation*}
		\left|\int (\Psi_\varepsilon\ast\delta_\tau T)\cdot g\Lebn-\partial_\tau T(g)\right|\leq c(n) j\int\int\Psi_\varepsilon(x-y)|x-y|\e\Lebn(y)\e|T|_\tau(x)
	\end{equation*}
	using \cref{TauVariation,propertiesFn}. An estimate similar to \cref{in1:c} yields
	\begin{equation*}
		j\int\int\Psi_\varepsilon(x-y)|x-y|\e\Lebn(y)\e|T|_\tau(x)
		\leq jc_{\varepsilon,n}\int|T|_\tau(\overline{B}_1(y))\e\Lebn(y)+j\varepsilon^{1/2}\int(\Psi_\varepsilon\ast |T|_\tau)\e\Lebn.
	\end{equation*}
	Note that $c_{\varepsilon,n}\leq\varepsilon^{1/2}$ for $\varepsilon$ sufficiently small, thus
	\begin{equation*}
		jc_{\varepsilon,n}\int|T|_\tau(\overline{B}_1(y))\e\Lebn(y)+j\varepsilon^{1/2}\int(\Psi_\varepsilon\ast |T|_\tau)\e\Lebn\leq c(n) j\varepsilon^{1/2}|T|_\tau(\R^n),
	\end{equation*}
	where we also used \cref{intBall} and $\int(\Psi_\varepsilon\ast |T|_\tau)\e\Lebn=|T|_\tau(\R^n)$. We have shown
	\begin{equation*}
		\left|\int (\Psi_\varepsilon\ast\delta_\tau T)\cdot g\Lebn-\partial_\tau T(g)\right|\leq c(n) j\varepsilon^{1/2}|T|_\tau(\R^n).
	\end{equation*}
	Combining this with \cref{Prop5.3}, $j\leq 2^{-1}\varepsilon^{-1/6}$, and $|T|_\tau(\R^n)\leq M$ gives
	\begin{multline*}
		\left|\int h_{\tau,\varepsilon}(T)\cdot g\e|T|_\tau+\partial_\tau T(g)\right|\\
		\leq \left|\int h_{\tau,\varepsilon}(T)\cdot g\e|T|_\tau+\int (\Psi_\varepsilon\ast\delta_\tau T)\cdot g\Lebn\right|+\left|-\int (\Psi_\varepsilon\ast\delta_\tau T)\cdot g\Lebn+\partial_\tau T(g)\right|\\
		\leq c(n) \varepsilon^{1/4}\left(\int\frac{|\Psi_\varepsilon\ast\delta_\tau T|^2}{\Psi_\varepsilon\ast|T|_\tau+\varepsilon}\e\Lebn\right)^{1/2}+2^{-1}\varepsilon^{-1/6}\varepsilon^{1/2}M.
	\end{multline*}
	Now use that $2^{-1}\varepsilon^{-1/6}\varepsilon^{1/2}M=2^{-1}M\varepsilon^{1/3}\leq\varepsilon^{1/4}$ for $\varepsilon$ sufficiently small.
\end{proof}
The next statement is \cite[Lem.\ 4.14]{KT17} (a simple calculation shows that it is valid in our setting), compare with the calculations in \cite[Sec.\ 4.7]{B78}.
\begin{lem}[Formula for $\textup{id}\Psi_\varepsilon+\varepsilon^2\nabla\Psi_\varepsilon$]
	\label{FormulaPsi}
	We have 
	\begin{equation*}
		x\Psi_\varepsilon(x)+\varepsilon^2\nabla\Psi_\varepsilon(x)=\varepsilon^2c_\varepsilon\nabla\psi(x)\widehat{\Psi}_\varepsilon(x).
	\end{equation*}
\end{lem}
We also use the following estimates.
\begin{lem}[Auxiliary estimates]
	\label{AuxEst}
	We have (see \cite[(5.6)]{KT17})
	\begin{equation*}
		|(\Psi_\varepsilon\ast\delta_\tau T)(x)|\leq\varepsilon^{-2}(\Psi_\varepsilon\ast|T|_\tau)(x)+ce^{-1/\varepsilon}|T|_\tau(\overline{B}_1(x)),
	\end{equation*}
	where $c$ is the constant from \cref{EstimatesNabla}. Moreover,
	\begin{equation*}
		(\Psi_\varepsilon\ast|T|_\tau)(x)\leq c(\varepsilon,n)|T|_\tau(\overline{B}_1(x)).
	\end{equation*}
\end{lem}
The following statement is the analogue of \cite[Prop.\ 5.4]{KT17}, cf.\ \cite[Sec.\ 4.7]{B78}.
\begin{prop}
	\label{LongProof}
	Let $M\in\R_+$. Then, for $\varepsilon$ sufficiently small (depending on $M,n$), we have
	\begin{equation}
		\label[ineq]{in1:5.23}
		\left| \delta_\tau T(\phi h_{\tau,\varepsilon}(T))+\int\frac{\phi|\Psi_\varepsilon\ast\delta_\tau T|^2}{\Psi_\varepsilon\ast|T|_\tau+\varepsilon}\e\Lebn \right|\leq \varepsilon^{1/4}\left(\int\frac{\phi|\Psi_\varepsilon\ast\delta_\tau T|^2}{\Psi_\varepsilon\ast|T|_\tau+\varepsilon}\e\Lebn+1\right)
	\end{equation}
	and
	\begin{equation}
		\label[ineq]{in1:5.24}
		\int |h_{\tau,\varepsilon}(T)|^2\phi\e|T|_\tau\leq \int\frac{\phi|\Psi_\varepsilon\ast\delta_\tau T|^2}{\Psi_\varepsilon\ast|T|_\tau+\varepsilon}(1+\varepsilon^{1/4})\e\Lebn+\varepsilon^{1/4}
	\end{equation}
	for all $j\leq 2^{-1}\varepsilon^{-1/6}$, $\phi\in\mathcal F_j^+$, and $T\in\rectn$ compactly supported with $\E_\tau(T)\leq M$.
\end{prop}
\begin{proof}
	Fix $T=\theta\Ha\mres S\in\rectn$ compactly supported with $|T|_\tau(\R^n)\leq M$. For simplicity, we write $h_{\tau,\varepsilon}=h_{\tau,\varepsilon}(T)$ and $\tilde h_{\tau,\varepsilon}=\tilde h_{\tau,\varepsilon}(T)$. 
	Define $V=v(\tau(|\theta|),S)$ induced by $T$ as in \cref{inducedV}. Note that it also satisfies $\delta V(\phi h_{\varepsilon})=\delta_\tau T(\phi h_{\tau,\varepsilon})$ ($h_\varepsilon$ defined as in \cite[(5.1)]{KT17}). Hence, the two inequalities can be written as in \cite[Prop.\ 5.4]{KT17}. Therefore, it only needs to be checked whether the proof of \cite[Prop.\ 5.4]{KT17} is independent of the codimension.
	As in the beginning of the proof of \cite[Prop.\ 5.4]{KT17} (the calculations are algebraically the same), we can write 
	\begin{equation*}
		\delta_\tau T(\phi h_{\tau,\varepsilon})+\int\frac{\phi|\Psi_\varepsilon\ast\delta_\tau T|^2}{\Psi_\varepsilon\ast|T|_\tau+\varepsilon}\e\Lebn=I_1+I_2+I_3,
	\end{equation*}
	where 
	\begin{align*}
		I_1&=\int\int(\phi(x)-\phi(y)-\nabla\phi(y)\cdot(x-y))T_xS\nabla\Psi_\varepsilon(x-y)\e|T|_\tau(x)\cdot \tilde h_{\tau,\varepsilon}(y)\e\Leb^n(y),\\
		I_2&=\int\int\Psi_\varepsilon(x-y)T_xS(\nabla\phi(x)-\nabla\phi(y))\e|T|_\tau(x)\cdot \tilde h_{\tau,\varepsilon}(y)\e\Leb^n(y),\\
		I_3&=\int\int\nabla\phi(y)\cdot(x-y)T_xS\nabla\Psi_\varepsilon(x-y)+\Psi_\varepsilon(x-y)T_xS\nabla\phi(y)\e|T|_\tau(x)\cdot\tilde{h}_{\tau,\varepsilon}(y)\e\Leb^n(y).
	\end{align*}
	Further, we also define
	\begin{equation*}
		I_4=-\varepsilon^2\int\int T_xS\nabla_x(\nabla\phi(y)\cdot\nabla\Psi_\varepsilon(x-y))\e|T|_\tau(x)\cdot\tilde h_{\tau,\varepsilon}(y)\e\Leb^n(y).
	\end{equation*}
	\underline{Estimation of $I_1$}: We can estimate as in \cite[(5.30)]{KT17}. However, since the dimension equals $n$, we obtain a term $\varepsilon^{-n-2}$ instead of $\varepsilon^{-n-3}$. More precisely, by \cref{propertiesF} and \cref{EstimatesNabla} we have (recall that $\phi\in[0,1]$)
	\begin{multline*}
		|I_1|\leq \int\int|\phi(x)-\phi(y)-\nabla\phi(y)\cdot(x-y)||T_xS|_2|\nabla\Psi_\varepsilon(x-y)|\e|T|_\tau(x)| \tilde h_{\tau,\varepsilon}(y)|\e\Leb^n(y)\\
		\leq j\int| \tilde h_{\tau,\varepsilon}(y)|\int |x-y|^2e^{j|x-y|}\frac{|x-y|}{\varepsilon^2}\Psi_\varepsilon(x-y)\e|T|_\tau(x)\e\Leb^n(y)\\
		+j\int| \tilde h_{\tau,\varepsilon}(y)|\int |x-y|^2e^{j|x-y|}c 1_{\overline B_1\setminus\overline B_{1/2}}(x-y)e^{-1/\varepsilon}\e|T|_\tau(x)\e\Leb^n(y),
	\end{multline*}
	where we use the notation $\overline {B}_r=\overline {B}_r(0)$. We obtain 
	\begin{multline*}
		j\int| \tilde h_{\tau,\varepsilon}(y)|\int |x-y|^2e^{j|x-y|}\frac{|x-y|}{\varepsilon^2}\Psi_\varepsilon(x-y)\e|T|_\tau(x)\e\Leb^n(y)\\
		= j\int| \tilde h_{\tau,\varepsilon}(y)|\int_{S\cap\overline{B}_{\varepsilon^{5/6}}(y)} |x-y|^2e^{j|x-y|}\frac{|x-y|}{\varepsilon^2}\Psi_\varepsilon(x-y)\tau(|\theta(x)|)\e\Ha(x)\e\Leb^n(y)\\
		+ j\int| \tilde h_{\tau,\varepsilon}(y)|\int_{S\setminus\overline{B}_{\varepsilon^{5/6}}(y)} |x-y|^2e^{j|x-y|}\frac{|x-y|}{\varepsilon^2}\Psi_\varepsilon(x-y)\tau(|\theta(x)|)\e\Ha(x)\e\Leb^n(y).
	\end{multline*}
	We have (using $\frac{|x-y|^3}{\varepsilon^2}\leq \sqrt{\varepsilon}$ in $\overline{B}_{\varepsilon^{5/6}}(y)$)
	\begin{multline*}
		j\int| \tilde h_{\tau,\varepsilon}(y)|\int_{S\cap\overline{B}_{\varepsilon^{5/6}}(y)} |x-y|^2e^{j|x-y|}\frac{|x-y|}{\varepsilon^2}\Psi_\varepsilon(x-y)\tau(|\theta(x)|)\e\Ha(x)\e\Leb^n(y)\\
		\leq j\sqrt{\varepsilon}e^{j\varepsilon^{5/6}}\int| \tilde h_{\tau,\varepsilon}(y)|\int_{S\cap\overline{B}_{\varepsilon^{5/6}}(y)} \Psi_\varepsilon(x-y)\tau(|\theta(x)|)\e\Ha(x)\e\Leb^n(y)\\
		\leq  j\sqrt{\varepsilon}e^{j\varepsilon^{5/6}}\int| \tilde h_{\tau,\varepsilon}(y)|\int\Psi_\varepsilon(x-y)\e|T|_\tau(x)\e\Leb^n(y).
	\end{multline*}
	For the second summand, we get (using $j\leq 2^{-1}\varepsilon^{-1/6}$ and $\Psi_\varepsilon=c_\varepsilon\psi\widehat\Psi_\varepsilon$ with $c_\varepsilon\leq c(n)$, see \cite[p.\ 63]{KT17})
	\begin{multline*}
		j\int| \tilde h_{\tau,\varepsilon}(y)|\int_{S\setminus\overline{B}_{\varepsilon^{5/6}}(y)} |x-y|^2e^{j|x-y|}\frac{|x-y|}{\varepsilon^2}\Psi_\varepsilon(x-y)\tau(|\theta(x)|)\e\Ha(x)\e\Leb^n(y)\\
		=j\int| \tilde h_{\tau,\varepsilon}(y)|\int_{S\setminus\overline{B}_{\varepsilon^{5/6}}(y)} |x-y|^2e^{j|x-y|}\frac{|x-y|}{\varepsilon^2}c_\varepsilon\psi(x-y)\frac{1}{(2\pi)^{n/2}\varepsilon^n}e^{-\frac{|x-y|^2}{2\varepsilon^2}}\tau(|\theta(x)|)\e\Ha(x)\e\Leb^n(y)\\
		\leq j\varepsilon^{-n-2}c(n)\int| \tilde h_{\tau,\varepsilon}(y)|\int_{S\setminus\overline{B}_{\varepsilon^{5/6}}(y)} |x-y|^3e^{j|x-y|}1_{\overline{B}_1(y)}(x)e^{-\frac{|x-y|^2}{2\varepsilon^2}}\tau(|\theta(x)|)\e\Ha(x)\e\Leb^n(y)\\
		\leq j\varepsilon^{-n-2}c(n)\int| \tilde h_{\tau,\varepsilon}(y)|\int_{S\setminus\overline{B}_{\varepsilon^{5/6}}(y)} 1^3e^{j\cdot 1}1_{\overline{B}_1(y)}(x)e^{-\frac{(\varepsilon^{5/6})^2}{2\varepsilon^2}}\tau(|\theta(x)|)\e\Ha(x)\e\Leb^n(y)\\
		\leq j\varepsilon^{-n-2}c(n)e^je^{-\frac{\varepsilon^{1/3}}{2}}\int | \tilde h_{\tau,\varepsilon}(y)||T|_\tau(\overline{B}_1(y))\e\Leb^n(y).
	\end{multline*}
	(Note that this estimate also works when $|x-y|^3$ is replaced by $|x-y|$, which is used below in the argument related to $|I_2|$.)
	Further, we have
	\begin{multline*}
		j\int| \tilde h_{\tau,\varepsilon}(y)|\int |x-y|^2e^{j|x-y|}c 1_{\overline B_1\setminus\overline B_{1/2}}(x-y)e^{-1/\varepsilon}\e|T|_\tau(x)\e\Leb^n(y)\\
		\leq  je^{j}ce^{-1/\varepsilon}\int| \tilde h_{\tau,\varepsilon}(y)|\int_S 1_{\overline B_1\setminus\overline B_{1/2}}(x-y) \tau(|\theta(x)|)\e\Ha(x)\e\Leb^n(y)\\
		\leq je^{j}ce^{-1/\varepsilon}\int| \tilde h_{\tau,\varepsilon}(y)||T|_\tau(\overline{B}_1(y))\e\Leb^n(y).
	\end{multline*}
	In summary, we obtain
	\begin{multline*}
		|I_1|\leq j\sqrt{\varepsilon}e^{j\varepsilon^{5/6}}\int| \tilde h_{\tau,\varepsilon}(y)|\int \Psi_\varepsilon(x-y)\e|T|_\tau(x)\e\Leb^n(y)\\
		+j\varepsilon^{-n-2}c(n)e^je^{-\frac{\varepsilon^{1/3}}{2}}\int| \tilde h_{\tau,\varepsilon}(y)||T|_\tau(\overline{B}_1(y))\e\Leb^n(y)\\
		+je^{j}ce^{-1/\varepsilon}\int| \tilde h_{\tau,\varepsilon}(y)||T|_\tau(\overline{B}_1(y))\e\Leb^n(y).
	\end{multline*}
	Now we can estimate as in \cite[(5.31) \&\ (5.32)]{KT17} using also \cref{intBall} (note that it does not matter that above we have $\varepsilon^{-n-2}$ instead of $\varepsilon^{-n-3}$ as in \cite[(5.30)]{KT17}). This yields
	\begin{equation}
		\label[ineq]{in1:I1Estimate}
		|I_1|\leq j\sqrt\varepsilon\int\frac{(\Psi_\varepsilon\ast\partial_\tau T)^2}{\Psi_\varepsilon\ast|T|_\tau+\varepsilon}\e\Leb^n+j  c(M,n)\sqrt\varepsilon +j c(M,n)e^{-\frac{1}{2}\varepsilon^{-1/6}}. 
	\end{equation}
	for $\varepsilon$ sufficiently small.
	
	\underline{Estimation of $I_2$}: By \cite[p.\ 73]{KT17} we have $|\nabla\phi(x)-\nabla\phi(y)|\leq j|x-y|\phi(x)e^{j|x-y|}$. This can be used to obtain
	\begin{multline*}
		|I_2|\leq \int\int\Psi_\varepsilon(x-y)|\nabla\phi(x)-\nabla\phi(y)|\e|T|_\tau(x)|\tilde h_{\tau,\varepsilon}(y)|\e\Leb^n(y)\\
		\leq j\int\int\Psi_\varepsilon(x-y)|x-y|\phi(x)e^{j|x-y|}\e|T|_\tau(x)|\tilde h_{\tau,\varepsilon}(y)|\e\Leb^n(y).
	\end{multline*}
	This looks like the first summand in the first estimate for $|I_1|$ with $|x-y|^3\varepsilon^{-2}$ replaced by $|x-y|$. We may write $|x-y|\varepsilon^{-2}$ instead of $|x-y|$. Then it also holds $|x-y|\varepsilon^{-2}\leq\sqrt{\varepsilon}$ for $x\in \overline{B}_{\varepsilon^{5/6}}(y)$. Thus, splitting the integration as in the estimate for $|I_1|$ yields that \cref{in1:I1Estimate} also holds for $|I_1|$ replaced by $|I_2|$ (in principal, one could use a slightly better estimate because there is no integral involving $1_{\overline{B}_1\setminus\overline{B}_{1/2}}$).
	
	In the remainder, we only sketch the proof since the arguments are as in the proof of \cite[Prop.\ 5.4]{KT17}.
	\underline{Estimation of $I_3-I_4$}: Note that \cite[(5.33)]{KT17} also holds in our setting (product rule),
	\begin{equation*}
		\nabla\phi(y)\cdot(x-y)T_xS\nabla\Psi_\varepsilon(x-y)+\Psi_\varepsilon(x-y)T_xS\nabla\phi(y)=T_xS\nabla_x((x-y)\cdot\nabla\phi(y)\Psi_\varepsilon(x-y)).
	\end{equation*}
	Using this and \cref{FormulaPsi}, we obtain \cite[(5.34)]{KT17},
	\begin{equation*}
		I_3-I_4=\int\int T_xS\nabla_x\left(\nabla\phi(y)\cdot\varepsilon^2c_\varepsilon\nabla\psi(x-y)\widehat{\Psi}_\varepsilon(x-y)\right)\e|T|_\tau(x)\cdot\tilde{h}_{\tau,\varepsilon}(y)\e\Leb^n(y).
	\end{equation*}
	Now, using similar estimates as for $|I_1|$, we obtain
	\begin{equation}
		\label[ineq]{in1:I34Estimate}
		|I_3-I_4|\leq c(M,n) je^{-1/\varepsilon}.
	\end{equation}
	\underline{Estimation of $I_4$}:
	Taking $V=v(\tau(|\theta|),S)$ and applying \cite[(5.37)]{KT17} we get
	\begin{equation*}
		I_4=-\frac{\varepsilon^2}{2}\int\frac{\nabla\phi\cdot\nabla(|\Psi_\varepsilon\ast\delta_\tau T|^2)}{\Psi_\varepsilon\ast|T|_\tau+\varepsilon}\e\Leb^n.
	\end{equation*}
	\cite[(5.38)]{KT17} can also be transferred (using $\psi_r(x)=\psi(x/r)$),
	\begin{multline*}
		\int\frac{\nabla\phi\cdot\nabla(|\Psi_\varepsilon\ast\delta_\tau T|^2)}{\Psi_\varepsilon\ast|T|_\tau+\varepsilon}\e\Leb^n\\
		=-\int\nabla\cdot\left(\frac{\nabla\phi}{\Psi_\varepsilon\ast|T|_\tau+\varepsilon}\right)|\Psi_\varepsilon\ast\delta_\tau T|^2\e\Leb^n-\lim_{r\to\infty}\int\frac{(\nabla\psi_r\cdot\nabla\phi)|\Psi_\varepsilon\ast\delta_\tau T|^2}{\Psi_\varepsilon\ast|T|_\tau+\varepsilon}\e\Leb^n.
	\end{multline*}
	\cite[(5.39)]{KT17} also holds (using \cref{hEstimates}),
	\begin{equation*}
		\int\frac{(\nabla\psi_r\cdot\nabla\phi)|\Psi_\varepsilon\ast\delta_\tau T|^2}{\Psi_\varepsilon\ast|T|_\tau+\varepsilon}\e\Leb^n\leq 2j\varepsilon^{-2}\int|\nabla\psi_r||\Psi_\varepsilon\ast\delta_\tau T|\e\Leb^n.
	\end{equation*}
	Now we can use the first two estimates in \cref{AuxEst}, which yields (see \cite[(5.40)]{KT17})
	\begin{equation*}
		\int|\nabla\psi_r||\Psi_\varepsilon\ast\delta_\tau T|\e\Leb^n\leq c(\varepsilon,n)r^{-1}\int_{\overline{B}_r\setminus\overline{B}_{r/2}}|T|_\tau(\overline{B}_1(x))\e\Leb^n(x).
	\end{equation*}
	It can also be checked that \cite[(5.41)]{KT17} holds in our setting,
	\begin{equation*}
		|I_4|\leq\frac{\varepsilon^2}{2}\int\left(\frac{nj\phi}{\Psi_\varepsilon\ast|T|_\tau+\varepsilon}+\frac{j\phi|\nabla\Psi_\varepsilon\ast|T|_\tau|}{(\Psi_\varepsilon\ast|T|_\tau+\varepsilon)^2}\right)|\Psi_\varepsilon\ast\delta_\tau T|^2\e\Lebn.
	\end{equation*}
	The last estimate in \cite[(5.42)]{KT17} can also be copied because it does not depend on the dimension,
	\begin{equation*}
		|(\nabla\Psi_\varepsilon\ast|T|_\tau)(y)|\leq\varepsilon^{-3/2}(\Psi_\varepsilon\ast|T|_\tau)(y)+ce^{-\varepsilon^{-1/2}}|T|_\tau(\overline{B}_1(y)).
	\end{equation*}
	We also have (see \cite[(5.7)]{KT17})
	\begin{equation*}
		|\tilde{h}_{\tau,\varepsilon}|\leq\varepsilon^{-2}+c(n)M\varepsilon^{-1}e^{-1/\varepsilon}.
	\end{equation*}
	\cite[(5.43)]{KT17} can also be obtained,
	\begin{equation}
		\label[ineq]{in1:I4Estimate}
		|I_4|\leq c(n) j\varepsilon^{1/2}\int\frac{\phi|\Psi_\varepsilon\ast\delta_\tau T|^2}{\Psi_\varepsilon\ast|T|_\tau+\varepsilon}\e\Lebn+je^{-\varepsilon^{-1/6}}.
	\end{equation}
	Using \cref{in1:I1Estimate,in1:I34Estimate,in1:I4Estimate} (recall that \cref{in1:I1Estimate} also holds for $|I_2|$), we get \cref{in1:5.23}.
	Finally, let us show \cref{in1:5.24}.
	Calculations \cite[(5.44)\&\ (5.45)]{KT17} can be transferred, which gives
	\begin{equation*}
		\int |h_{\tau,\varepsilon}|^2\phi\e|T|_\tau\leq\int|\tilde{h}_{\tau,\varepsilon}(y)|^2\left(\phi(y)(\Psi_\varepsilon\ast|T|_\tau)(y)+j\phi(y)\int e^{j|x-y|}|x-y|\Psi_\varepsilon(x-y)\e|T|_\tau(x)   \right)\e\Lebn(y).
	\end{equation*}
	We also obtain an inequality similar to \cite[(5.46)]{KT17},
	\begin{equation*}
		\int e^{j|x-y|}|x-y|\Psi_\varepsilon(x-y)\e|T|_\tau(x)\leq\varepsilon^{1/2}e^{j\varepsilon^{1/2}}(\Psi_\varepsilon\ast|T|_\tau)(y)+c(n)e^{j-\varepsilon^{-1/2}}|T|_\tau(\overline{B}_1(y)).
	\end{equation*}
	The analogue of \cite[(5.47)]{KT17} is
	\begin{equation*}
		\int |h_{\tau,\varepsilon}|^2\phi\e|T|_\tau\leq\int|\tilde{h}_{\tau,\varepsilon}(y)|^2\left((\Psi_\varepsilon\ast|T|_\tau)(y)\phi(y)(1+j\varepsilon^{1/2}e^{j\varepsilon^{1/2}})+jc(n)e^{-\frac{1}{2}\varepsilon^{-1/2}}|T|_\tau(\overline{B}_1(y))  \right)\e\Lebn(y).
	\end{equation*}
	The final argument in the proof of \cite[Prop.\ 5.4]{KT17} then yields \cref{in1:5.24}. 
\end{proof}
The following statement is the analogue of \cite[Prop.\ 5.7]{KT17}, cf.\ \cite[Sec.\ 4.11]{B78}.
\begin{prop}
	\label{Prop5.7}
	Fix some $j\in\mathbb N$ and $M\in\R_+^*$.
	Define
	\begin{equation*}
		c_1=3n+20.
	\end{equation*}
	Set (implicitly depending on $\varepsilon\in\R_+^*,\Dt\in\R_+$, and $T=\theta\Ha\mres S\in \rectn$ with $|T|_\tau(\R^n)<\infty$)
	\begin{equation*}
		f=\mathrm{id}+\Delta t\, h_{\tau,\varepsilon}(T). 
	\end{equation*}
	Then, for $\varepsilon$ sufficiently small (depending on $M,n$), we have
	\begin{equation}
		\label[ineq]{in1:Q1}
		\left| \frac{|f_\#T|_\tau(\phi)-|T|_\tau(\phi)}{\Dt}-\delta_\tau(T,\phi)(h_{\tau,\varepsilon}(T)) \right|\leq\varepsilon^{c_1-10}  
	\end{equation}
	and
	\begin{equation}
		\label[ineq]{in1:Q2}
		\frac{|f_\#T|_\tau(\R^n)-|T|_\tau(\R^n)}{\Dt}+\frac{1}{4}\int\frac{|\Psi_\varepsilon\ast\delta_\tau T|^2}{\Psi_\varepsilon\ast|T|_\tau+\varepsilon}\e\Leb^n\leq 3\varepsilon^{1/4}  
	\end{equation}
	for all $\Dt\in(2^{-1}\varepsilon^{c_1},\varepsilon^{c_1}),\phi\in\mathcal F_j^+$, and $T=\theta\Ha\mres S\in \rectn$ compactly supported with $\Ha(S)<\infty$ and $|T|_\tau(\R^n)\leq M$.
	Further, for $\varepsilon$ sufficiently small, we get
	\begin{equation}
		\label[ineq]{in1:Q3}
		\left| \partial_\tau(T,\phi)(h_{\tau,\varepsilon}(T))-\partial_\tau(f_\# T,\phi)(h_{\tau,\varepsilon}(f_\# T)) \right|\leq \varepsilon^{c_1-2n-17}
	\end{equation}
	and
	\begin{equation}
		\label[ineq]{in1:Q4}
		\left| \int\frac{|\Psi_\varepsilon\ast\delta_\tau T|^2}{\Psi_\varepsilon\ast|T|_\tau+\varepsilon}\e\Lebn -\int\frac{|\Psi_\varepsilon\ast\delta_\tau (f_\# T)|^2}{\Psi_\varepsilon\ast| f_\# T|_\tau+\varepsilon}\e\Lebn\right|\leq\varepsilon^{c_1-3n-15}
	\end{equation}
	for all $\Dt\in(2^{-1}\varepsilon^{c_1},\varepsilon^{c_1}),\phi\in\mathcal F_j^+$, and $T=\theta\Ha\mres S\in \rectn$ compactly supported with $\Ha(S)<\infty$, $|T|_\tau(\R^n)\leq M$, and $|f_\# T|_\tau(\R^n)\leq M$.
\end{prop}
\begin{proof}
	\underline{\Cref{in1:Q1}:}
	Always assume that $\Dt\in(2^{-1}\varepsilon^{c_1},\varepsilon^{c_1})$.
	Write 
	\begin{equation*}
		F=f-\mathrm{id}.
	\end{equation*}
	We have (\cref{hEstimates})
	\begin{equation}
		\label[ineq]{in1:5.59}
		|F|=|\Delta t\, h_{\tau,\varepsilon}(T)|\leq \Dt2\varepsilon^{-2}\leq 2\varepsilon^{c_1-2}\qquad\textup{and}\qquad|\Jm F|_2=|\Delta t\, \Jm h_{\tau,\varepsilon}(T)|_2\leq \Dt4\varepsilon^{-4}\leq 4\varepsilon^{c_1-4}.
	\end{equation}
	By \cref{propertiesF} we have (recall that $\phi\in[0,1]$)
	\begin{equation}
		\label[ineq]{in1:I}
		|\phi(f)-\phi|\leq j|f-\textup{id}|\phi e^{j|f-\textup{id}|}\leq j2\varepsilon^{c_1-2}e^{j2\varepsilon^{c_1-2}}\leq\varepsilon^{c_1-3}
	\end{equation}
	for $\varepsilon$ sufficiently small. For $\Ha$-a.e.\ $x\in S$ we have $\Jac f(x)=|\Jm f(x)u_x |$, where $u_x\in T_xS$ with $|u_x|=1$. This yields
	\begin{equation}
		\label[ineq]{in1:II}
		|\Jac f(x)-1|=||\Jm f(x)u_x |-|u_x||\leq |\Jm f(x)u_x -u_x|=|\Jm F(x)u_x|\leq 4\varepsilon^{c_1-4}\leq 2^{-1}\varepsilon^{c_1-5}\leq\Delta t\varepsilon^{-5}
	\end{equation}
	for $\Ha$-a.e.\ $x\in S$ and $\varepsilon$ sufficiently small (independent of $x$).
	Using \cref{propertiesF} again, we obtain
	\begin{equation}
		\label[ineq]{in1:III}
		|\phi(f)-\phi-\nabla\phi\cdot F|\leq j|F|^2\phi e^{j|F|}\leq j4\varepsilon^{2c_1-4}e^{j2\varepsilon^{c_1-2}}\leq 2^{-1}\varepsilon^{2c_1-5}\leq\Delta t\varepsilon^{c_1-5}
	\end{equation}
	for $\varepsilon$ sufficiently small.
	Next, we estimate $|\Jac f-1-\Jm F:T_\cdot S|$. To this end, define
	\begin{equation*}
		g_x(\Delta t)=|(\mathbb 1 +\Delta t\Jm h_{\tau,\varepsilon}(T)(x))u_x|,
	\end{equation*}
	where $u_x$ is as above, for $\Ha$-a.e.\ $x\in S$. We get
	\begin{equation*}
		g_x^\prime(\Delta t)=\Jm h_{\tau,\varepsilon}(T)(x)u_x\cdot\frac{(\mathbb 1 +\Delta t\Jm h_{\tau,\varepsilon}(T)(x))u_x}{|(\mathbb 1 +\Delta t\Jm h_{\tau,\varepsilon}(T)(x))u_x|}.
	\end{equation*}
	Further, we have (with $v=(\mathbb 1 +\Delta t\Jm h_{\tau,\varepsilon}(T)(x))u_x$)
	\begin{multline*}
		\left|g_x^{\prime\prime}(\Delta t)\right|=\left||v|^{-1}(\Jm h_{\tau,\varepsilon}(T)(x)u_x)^\intercal\left(\mathbb 1-|v|^{-2}v\otimes v  \right)\Jm h_{\tau,\varepsilon}(T)(x)u_x\right|\\
		=\left||v|^{-1}|\Jm h_{\tau,\varepsilon}(T)(x)u_x|^2-|v|^{-3}|v^\intercal \Jm h_{\tau,\varepsilon}(T)(x)u_x|^2\right|\\
		=|v|^{-1}|\Jm h_{\tau,\varepsilon}(T)(x)u_x|^2-|v|^{-3}|v^\intercal \Jm h_{\tau,\varepsilon}(T)(x)u_x|^2\leq |v|^{-1}|\Jm h_{\tau,\varepsilon}(T)(x)u_x|^2\leq(1-4\varepsilon^{c_1-4})^{-1}(\Delta t)^{-2}|\Jm F|_2^2\\
		\leq (1-4\varepsilon^{c_1-4})^{-1}(\Delta t)^{-2}16\varepsilon^{2c_1-8}\leq (\Delta t)^{-2}\varepsilon^{2c_1-9}\leq (\Delta t)^{-1}2\varepsilon^{c_1-9}
	\end{multline*}
	for $\varepsilon$ sufficiently small. 
	Using this and Taylor's theorem, we obtain
	\begin{equation}
		\label[ineq]{in1:IV}
		|\Jac f(x)-1-\Jm F(x):T_x S|=|g_x(\Delta t)-g_x(0)-g_x^\prime(0)\Delta t|\leq\sup_{\xi\in(-2\Delta t,2\Delta t)}|2^{-1}g_x^{\prime\prime}(\xi)(\Delta t)^2|\leq \Delta t\varepsilon^{c_1-9}.
	\end{equation}
	Applying \cite[Cor.\ 3.2.20]{F} (note that $f$ is a diffeomorphism for $\varepsilon$ sufficiently small by the Hadamard\textendash Cacciopoli theorem) we get (writing $f_\# T=\eta\Ha\mres f(S)$ as in \cref{FedererPushforward})
	\begin{multline*}
		|f_\#T|_\tau(\phi)-|T|_\tau(\phi)-\Dt\delta_\tau(T,\phi)(h_{\tau,\varepsilon}(T))\\
		=\int_{f(S)}\phi\tau(|\eta|)\e\Ha-\int_S\phi\tau(|\theta|)\e\Ha-\Dt \int_S(\phi\textup{div}_S(h_{\tau,\varepsilon}(T))+\nabla\phi\cdot h_{\tau,\varepsilon}(T))\tau(|\theta|)\e\Ha\\
		=\int_S\phi(f)\tau(|\eta(f)|)\Jac f-\phi\tau(|\theta|)\e\Ha-\int_S(\phi\Jm F:T_\cdot S+\nabla\phi\cdot F)\tau(|\theta|)\e\Ha\\
		=\int_S\left((\phi(f)-\phi)(\Jac f-1)+\phi(f)-\phi-\nabla\phi\cdot F+(\Jac f-1-\Jm F:T_\cdot S)\phi\right)\tau(|\theta|)\e\Ha
	\end{multline*}
	for all $\varepsilon$ sufficiently small, $\Dt\in(2^{-1}\varepsilon^{c_1},\varepsilon^{c_1}),\phi\in\mathcal F_j^+$, and $T=\theta\Ha\mres S\in \rectn$ compactly supported with $\Ha(S)<\infty$ and $|T|_\tau(\R^n)\leq M$. Finally, application of \cref{in1:I,in1:II,in1:III,in1:IV} yields (note that $\varepsilon$ needs to satisfy $3M\leq\varepsilon^{-1}$)
	\begin{multline*}
		||f_\#T|_\tau(\phi)-|T|_\tau(\phi)-\Dt\delta_\tau(T,\phi)(h_{\tau,\varepsilon}(T))|\\
		\leq \int_S\left(|\phi(f)-\phi||\Jac f-1|+|\phi(f)-\phi-\nabla\phi\cdot F|+|\Jac f-1-\Jm F:T_\cdot S||\phi|\right)\tau(|\theta|)\e\Ha\\
		\leq  \int_S\left(\varepsilon^{c_1-3}\Dt\varepsilon^{-5}+\Dt\varepsilon^{c_1-5}+\Dt\varepsilon^{c_1-9}\right)\tau(|\theta|)\e\Ha\leq \Dt 3\varepsilon^{c_1-9}M\leq \Dt \varepsilon^{c_1-10}
	\end{multline*}
	for all $\varepsilon$ sufficiently small, $\Dt\in(2^{-1}\varepsilon^{c_1},\varepsilon^{c_1}),\phi\in\mathcal F_j^+$, and $T=\theta\Ha\mres S\in \rectn$ compactly supported with $\Ha(S)<\infty$ and $|T|_\tau(\R^n)\leq M$. This shows \cref{in1:Q1}.
	
	\underline{\Cref{in1:Q2}:} Note that $\phi\equiv 1\in\mathcal F_j^+$ for all $j\in\mathbb N$. Thus, taking $\varepsilon$ sufficiently small and applying \cref{in1:5.23}, we get (cf.\ \cite[(5.66)]{KT17})
	\begin{equation*}
		\delta_\tau(T,1)(h_{\tau,\varepsilon})\leq(\varepsilon^{1/4}-1)I+\varepsilon^{1/4}
	\end{equation*}
	using the abbreviation $I=\int\frac{|\Psi_\varepsilon\ast\delta_\tau T|^2}{\Psi_\varepsilon\ast|T|_\tau+\varepsilon}\e\Lebn$. Applying this and \cref{in1:Q1}, we obtain
	\begin{multline*}
		\frac{|f_\#T|_\tau(\R^n)-|T|_\tau(\R^n)}{\Dt}+\frac{1}{4}I=\frac{|f_\#T|_\tau(\R^n)-|T|_\tau(\R^n)}{\Dt}-\delta_\tau(T,1)(h_{\tau,\varepsilon})+\delta_\tau(T,1)(h_{\tau,\varepsilon})+\frac{1}{4}I\\
		\leq\varepsilon^{c_1-10}+(\varepsilon^{1/4}-3/4)I+\varepsilon^{1/4}\leq 3\varepsilon^{1/4},
	\end{multline*}
	which yields \cref{in1:Q2}.
	
	\underline{\Cref{in1:Q3,in1:Q4}:} Write $f_\# T=\eta\Ha\mres f(S)$ as in \cref{FedererPushforward} (recall that $f$ is a diffeomorphism for $\varepsilon$ sufficiently small).
	Considering the varifolds $v(\tau(|\theta|),S)$ and $v(\tau(|\eta|),f(S))$ and applying the area formula \cite[Cor.\ 3.2.20]{F} and \cref{FedererPushforward} we get the analogue of \cite[(5.67)]{KT17},
	\begin{multline}
		\label[ineq]{in1:V}
		|(\Psi_\varepsilon\ast|f_\# T|_\tau)(x)-(\Psi_\varepsilon\ast|T|_\tau)(x)|\\
		\leq\int|\Psi_\varepsilon(f(y)-x)-\Psi_\varepsilon(y-x)|\Jac f(y)\e|T|_\tau(y)+\int\Psi_\varepsilon(y-x)|\Jac f(y)-1|\e|T|_\tau(y).
	\end{multline}
	Using $|F|\leq 2\varepsilon^{c_1-2}$ (beginning of proof) and \cref{EstimatesNabla}, we get (see \cite[(5.68)]{KT17})
	\begin{equation}
		\label[ineq]{in1:VI}
		|\Psi_\varepsilon(f(y)-x)-\Psi_\varepsilon(y-x)|\leq c(n) \varepsilon^{c_1-n-4}1_{\overline{B}_1}(\hat y)\leq c(n) \varepsilon^{c_1-n-4}1_{B_2(x)}( y)
	\end{equation}
	for some $\hat y\in[y-x,f(y)-x]$, where the last inequality follows from
	\begin{equation*}
		1\geq|\hat y|=|y-x|+\lambda|f(y)-y|\geq|y-x|-2\varepsilon^{c_1-2}
	\end{equation*}
	for some $\lambda\in[-1,1]$. 
	By \cref{in1:II} and $\Dt<\varepsilon^{c_1}$ we get an inequality similar to \cite[(5.69)]{KT17},
	\begin{equation}
		\label[ineq]{in1:VII}
		\Psi_\varepsilon(x-y)|\Jac f(y)-1|\leq c(n)\varepsilon^{c_1-n-5}1_{\overline{B}_1(x)}(y).
	\end{equation}
	Using \cref{in1:V,in1:VI,in1:VII} and $\Jac f(y)\leq 1+\varepsilon^{c_1-5}$ by \cref{in1:II}, we obtain
	\begin{equation}
		\label[ineq]{in1:VIII}
		|(\Psi_\varepsilon\ast|f_\# T|_\tau)(x)-(\Psi_\varepsilon\ast|T|_\tau)(x)|\leq\varepsilon^{c_1-n-6}|T|_\tau(B_2(x))
	\end{equation}
	for $\varepsilon$ sufficiently small, see \cite[(5.70)]{KT17}.
	A calculation similar to \cite[(5.71)]{KT17} yields (using \cite[Cor.\ 3.2.20]{F}, \cref{FedererPushforward}, and \cref{FedererPushforward})
	\begin{equation*}
		|(\Psi_\varepsilon\ast\delta_\tau(f_\# T))(x)-(\Psi_\varepsilon\ast\delta_\tau T)(x)|=\left| \int\left(\nabla f(y)T_yS\nabla\Psi_\varepsilon(f(y)-x)\Jac f(y)-T_yS\nabla\Psi_\varepsilon(y-x)\right)\e |T|_\tau(y) \right|.
	\end{equation*}
	An argument as in the passage before \cite[(5.72)]{KT17} gives 
	\begin{equation}
		\label[ineq]{in1:VIIII}
		|(\Psi_\varepsilon\ast\delta_\tau(f_\# T))(x)-(\Psi_\varepsilon\ast\delta_\tau T)(x)|\leq \varepsilon^{c_1-n-8}|T|_\tau(B_2(x))
	\end{equation}
	for $\varepsilon$ sufficiently small. 
	The rough estimates in \cite[(5.73)]{KT17} can also be replaced,
	\begin{equation}
		\label[ineq]{in1:X}
		|(\Psi_\varepsilon\ast\delta_\tau T)(x)|, |(\Psi_\varepsilon\ast\delta_\tau (f_\# T))(x)|\leq\varepsilon^{-n-3}|T|_\tau(B_2(x)).
	\end{equation}
	Now use \cref{in1:VIII,in1:VIIII,in1:X},
	\begin{multline}
		\label[ineq]{in1:XI}
		\left| \frac{\Psi_\varepsilon\ast\delta_\tau(f_\# T)}{\Psi_\varepsilon\ast|f_\# T|_\tau+\varepsilon}-\frac{\Psi_\varepsilon\ast\delta_\tau T}{\Psi_\varepsilon\ast|T|_\tau+\varepsilon}  \right|\\
		\leq \left| \frac{\Psi_\varepsilon\ast\delta_\tau(f_\# T)-\Psi_\varepsilon\ast\delta_\tau T}{\Psi_\varepsilon\ast|f_\# T|_\tau+\varepsilon}\right|+\left|\frac{\Psi_\varepsilon\ast\delta_\tau T}{\Psi_\varepsilon\ast|f_\# T|_\tau+\varepsilon}-\frac{\Psi_\varepsilon\ast\delta_\tau T}{\Psi_\varepsilon\ast|T|_\tau+\varepsilon}  \right|\\
		\leq\varepsilon^{c_1-n-9}|T|_\tau(B_2(\cdot))+\varepsilon^{c_1-2n-11}|T|_\tau(B_2(\cdot))^2,
	\end{multline}
	see \cite[(5.74)]{KT17}.
	Applying \cref{in1:VIII,in1:VIIII,in1:X} again,
	\begin{equation}
		\label[ineq]{in1:XII}
		\left| \frac{(\Psi_\varepsilon\ast\delta_\tau(f_\# T))^2}{\Psi_\varepsilon\ast|f_\# T|_\tau+\varepsilon}-\frac{(\Psi_\varepsilon\ast\delta_\tau T)^2}{\Psi_\varepsilon\ast|T|_\tau+\varepsilon}  \right|
		\leq 2\varepsilon^{c_1-2n-12}|T|_\tau(B_2(\cdot))^2+\varepsilon^{c_1-3n-14}|T|_\tau(B_2(\cdot))^3,
	\end{equation}
	see \cite[(5.75)]{KT17}.
	Using \cref{in1:XII} and $\int|T|_\tau(\overline{B}_2(x))^k\e\Lebn(x)\leq M^{k-1}\omega_n2^nM=M^k\omega_n2^n$ (\cref{intBall}), we get \cref{in1:Q4}.
	Calculating $\nabla\Psi_\varepsilon$ and $\nabla^2\Psi_\varepsilon$ and applying \cref{in1:XI} gives the analogue of \cite[(5.76)]{KT17},
	\begin{align}
		|h_{\tau,\varepsilon}(T)-h_{\tau,\varepsilon}(f_\# T)|&\leq\varepsilon^{c_1-2n-12}(M+M^2),\label[ineq]{in1:5.76a}\\
		|\Jm h_{\tau,\varepsilon}(T)-\Jm h_{\tau,\varepsilon}(f_\# T)|&\leq\varepsilon^{c_1-2n-14}(M+M^2),\textup{ and}\label[ineq]{in1:5.76b}\\
		|\Jm^2h_{\tau,\varepsilon}(T)-\Jm^2 h_{\tau,\varepsilon}(f_\# T)|_2&\leq\varepsilon^{c_1-2n-16}(M+M^2).\label[ineq]{in1:5.76c}
	\end{align}
	Further, \cite[(5.77)]{KT17} becomes (using the definition of $\phi$-weighted $\tau$-variation, \cite[Cor.\ 3.2.20]{F}, \cref{FedererPushforward}, and \cref{FedererPushforward})
	\begin{multline*}
		\delta_\tau(T,\phi)(h_{\tau,\varepsilon}(T))-\delta_\tau(f_\# T,\phi)(h_{\tau,\varepsilon}(f_\# T)) \\
		=  \int_S[\Jm h_{\tau,\varepsilon}(T):T_\cdot S\phi+h_{\tau,\varepsilon}(T)\cdot\nabla\phi\\
		-\left( \Jm h_{\tau,\varepsilon}(f_\# T)(f):(\Jm fT_\cdot S)\phi(f)+h_{\tau,\varepsilon}(f_\# T)(f)\cdot\nabla\phi(f) \right)\Jac f]\tau(|\theta|)\e\Ha.
	\end{multline*}
	The last argument in the proof of \cite[Prop.\ 5.7]{KT17} uses \cite[(5.59)--(5.62)]{KT17} and \cite[(5.76)]{KT17} which correspond to \cref{in1:5.59,in1:I,in1:II,in1:5.76a,in1:5.76b,in1:5.76c}.
	In comparison, the exponent of $\varepsilon$ in \cref{in1:5.59,in1:I,in1:II,in1:5.76a,in1:5.76b,in1:5.76c} can be obtained by replacing $n$ in the proof of \cite[Prop.\ 5.7]{KT17} by $n-1$.
	Hence, applying the same argument as in the last two lines in the proof of \cite[Prop.\ 5.7]{KT17} gives \cref{in1:Q3}.
\end{proof}
\section{Geometric flow $(\mu_t)$ as a limit of induced $\tau$-weight measures $|T_t^{j_k}|_\tau$}
\label{GFInduced}
In this section, we will introduce our modification of the variational approximation schemes in \cite{B78,KT17} and state some related estimates.
Up to some small modifications, we follow the proofs in \cite[Sec.\ 6]{KT17}, see also \cite[Sec.\ 4.14]{B78}.
The following auxiliary statement will be used in the proof of \cref{Prop6.1}.
\begin{lem}[Orthogonal projection onto compact, convex set does not increase branched transport cost]
\label{ProjCost}
Let $\Omega\subset\R^n$ be compact and convex. Write $p:\R^n\to\Omega$ for the orthogonal projection onto $\Omega$. Then we have
\begin{equation*}
|p_\# T|_\tau(\R^n)\leq |T|_\tau(\R^n)
\end{equation*}
for all $T\in\T$.
\end{lem}
\begin{proof}
First, let us verify the inequality for the polyhedral $1$-chains $\mathbb P_1(\R^n)$. Each $P\in\mathbb P_1(\R^n)$ can be written
\begin{equation*}
	P=\sum_em_e\vec{e}\Ha\mres e,
\end{equation*}
where the sum if over finitely many non-overlapping edges $e=e_-+[0,1](e_+-e_-)\subset\R^n$ with orientation $(e_+-e_-)/|e_+-e_-|\in\mathcal S^{n-1}$ and mass $m_e\in\R_+$. 
Now label the edges $e_1,\ldots,e_N$ and write $p_\# P=\theta\Ha\mres S$. 
We can write (possibly neglecting some $\Ha$-zero set) $S$ as a disjoint union $S=\bigcup_I S_I$, where $I\subset\{ 1,\ldots,N \}$ and $S_I=\bigcap_{i\in I}p(e_i)\setminus\bigcup_{j\in\{ 1,\ldots, N \}\setminus I} p(e_j)$.
Using this decomposition (in the second equation), we obtain
\begin{multline*}
|P|_\tau(\R^n)=\sum_e\tau(m_e)\Ha(e)\geq\sum_e\tau(m_e)\Ha(p(e))=\sum_{I\subset\{ 1,\ldots,N \}}\sum_{i\in I}\tau(m_{e_i})\Ha(S_I)=\int_{S}\sum_{x\in p(e)}\tau(m_e)\e\Ha(x)\\
\geq \int_{S}\tau\bigg(\sum_{x\in p(e)} m_e\bigg)\e\Ha(x)\geq \int_{S}\tau(|\theta|)\e\Ha=|p_\# P|_\tau(\R^n),
\end{multline*}
where we also used the $1$-Lipschitz continuity of $p$ in the first inequality, the subadditivity of $\tau$ in the second inequality, and \cref{FedererPushforward} in the last inequality.
Now let $T\in\T$ and $\varepsilon\in\R_+^*$ be arbitrary. Take any compact and convex $K\subset\R^n$ with $\spt(T)\subset K$. By \cite[Prop.\ 2.32]{BW} and \cite[Thm.\ 5]{MW19} we have the following formula:
\begin{equation*}
|T|_\tau(\R^n)=\inf_{(P_i)}\,\liminf_i|P_i|_\tau(\R^n),
\end{equation*}
where the infimum is over sequences $(P_i)\subset\mathbb P_1(\R^n)$ with $\spt(P_i)\subset K$ and $|P_i-T|^\flat\to 0$. 
Fix any such sequence $(P_i)$ with $\lim_i|P_i|_\tau(\R^n)\leq |T|_\tau(\R^n)+\varepsilon/2$.
Write $T_i=p_\# P_i$. 
By \cite[p.\ 371]{F} and \cite[Memo 2.5]{LSW25} we obtain
\begin{equation*}
|T_i-p_\# T|_\Omega^\flat\leq |P_i-T|_K^\flat=|P_i-T|^\flat\to 0.
\end{equation*}
Using the formula from above, for each $i\in\mathbb N$, we can pick a sequence $(Q_i^j)\subset\mathbb P_1(\R^n)$ with $\lim_j|T_i-Q_i^j|^\flat= 0$ and $|Q_i^j|_\tau(\R^n)\leq |T_i|_\tau(\R^n)+\varepsilon/2$ for all $j$. Further, we can let $j=j(i)$ depend on $i$ such that $|Q_i^{j(i)}-p_\# T|^\flat\to 0$ for $i\to\infty$ (triangle inequality). This gives
\begin{equation*}
|p_\# T|_\tau(\R^n)\leq\liminf_i|Q_i^{j(i)}|_\tau(\R^n)\leq\liminf_i|T_i|_\tau(\R^n)+\frac{\varepsilon}{2}\leq \liminf_i|P_i|_\tau(\R^n)+\frac{\varepsilon}{2}\leq |T|_\tau(\R^n)+\varepsilon
\end{equation*}
where the formula from above was used in the first inequality and $T_i=p_\# P_i$ with $P_i\in\mathbb P_1(\R^n)$ in the third inequality (beginning of proof). Since $\varepsilon$ was arbitrary, this proves the claim. 
\end{proof}
The next statement corresponds to \cite[Prop.\ 6.1]{KT17}, cf.\ \cite[Secs.\ 4.13 \&\ 4.14]{B78}.
We will modify the scheme in \cite{KT17}. A brief summary of the essential steps is given in \cref{MainResults}.
\begin{prop}
	\label{Prop6.1}
	Let $T_0\in\T$ with $|T_0|_\tau(\R^n)<\infty$. 
	Write $\Omega=\textup{conv}(\spt(T_0))$.
	Define $T^{j,0}=T_0$ for all $j\in\mathbb N$. 
	For all $j\in\mathbb N$ there exist $\varepsilon_j\in(0,j^{-6}),p_j\in\mathbb N$, and $T^{j,1},\ldots,T^{j,j2^{p_j}}\in\T$ with $\spt(T^{j,\ell})\subset\{ \textup{dist}(\cdot,\Omega)\leq 2\varepsilon_j^{c_1-2} \}$ and
	\begin{align}
		|T^{j,\ell}|_\tau(\R^n)&\leq |T_0|_\tau(\R^n)+\varepsilon_j^{1/8}\ell\Dt_j,\label[ineq]{in1:6.3} \\
		\frac{|T^{j,\ell}|_\tau(\R^n)-|T^{j,\ell-1}|_\tau(\R^n)}{\Dt_j}&+\frac{1}{4}\int\frac{| \Psi_{\varepsilon_j}\ast\delta_\tau T^{j,\ell} |^2}{\Psi_{\varepsilon_j}\ast|T^{j,\ell}|_\tau+\varepsilon_j  }\e\Lebn-\frac{1-j^{-5}}{\Dt_j}\Delta_j|T^{j,\ell-1}|_\tau(\R^n) \leq\varepsilon_j^{1/8},\qquad\textup{and} \label[ineq]{in1:6.4}\\
		\frac{|T^{j,\ell}|_\tau(\phi)-|T^{j,\ell-1}|_\tau(\phi)}{\Dt_j} &\leq \delta_\tau(T^{j,\ell},\phi)(h_{\tau,\varepsilon_j}(T^{j,\ell}))+\varepsilon_j^{1/8}\label[ineq]{in1:6.5}
	\end{align}
	for all $\ell=1,\ldots,j2^{p_j}$ and $\phi\in\mathcal F_j^+$, where  $\Dt_j=2^{-p_j}$. Further, we can choose $p_j$ such that $j\mapsto p_j$ is increasing.
\end{prop}
\begin{proof}
	Set $M=|T_0|(\R^n)+1$, see \cite[(6.6)]{KT17}. For each $j\in\mathbb N$ take $\varepsilon_j\in(0,1)$ such that $j\leq 2^{-1}\varepsilon_j^{-1/6}$, the conclusions in \cref{hEstimates,Prop5.3,LongProof,Prop5.5,Prop5.7} hold, 
	\begin{equation}
		\label[ineq]{in1:6.7}
		\varepsilon_j^{1/8}j<1,\qquad\textup{and}\qquad 3\varepsilon_j^{1/4}+\varepsilon_j^{c_1-3n-15}<\varepsilon_j^{1/8},
	\end{equation}
	in particular $\varepsilon_j<j^{-6}$, cf.\ \cite[(6.7)]{KT17}.
	Choose $p_j\in \mathbb N$ such that $\Dt_j=2^{-p_j}\in(\varepsilon_j^{c_1}/2,\varepsilon_j^{c_1}]$ \cite[(6.8)]{KT17}.
	If necessary, decrease $\varepsilon_j$ to ensure that $j\mapsto p_j$ is increasing.
	No keep $j$ fixed.
	Assume that $T^{j,k}$ satisfies \cref{in1:6.3,in1:6.4,in1:6.5} for all $k\in\{ 0,1,\ldots,\ell \}$. 
	We show that \cref{in1:6.3,in1:6.4,in1:6.5} are true for all $k\in\{ 0,1,\ldots,\ell+1 \}$. 
	(If $k=0$, then \cref{in1:6.3} is trivially satisfied, and for $k=1$ it will be sufficient to use this inequality to obtain \cref{in1:6.3,in1:6.4,in1:6.5}.) 
	Take $f_1\in\adm_j(T^{j,\ell})$ such that (see \cite[(6.9)]{KT17} or \cite[Sec.\ 4.13]{B78})
	\begin{equation}
		\label[ineq]{in1:6.9}
		|(f_1)_\# T^{j,\ell}|_\tau(\R^n)-|T^{j,\ell}|_\tau(\R^n)\leq(1-j^{-5})\Delta_j|T^{j,\ell}|_\tau(\R^n).
	\end{equation}
	As in \cite[(6.10)]{KT17} define 
	\begin{equation}
		\label{6.10}
		\widetilde{T}^{j,\ell+1}=(f_1)_\# T^{j,\ell}.
	\end{equation}
	This is a transportation network by \cref{FedererPushforward}. The admissibility of $f_1$, induction hypothesis, and the first inequality in \labelcref{in1:6.7} yield (cf.\ \cite[(6.11)]{KT17})
	\begin{equation}
		\label[ineq]{in1:6.11}
		|\widetilde T^{j,\ell+1}|_\tau(\R^n)\leq|T^{j,\ell}|_\tau(\R^n)\leq|T_0|_\tau(\R^n)+\varepsilon_j^{1/8}\ell\Dt_j\leq |T_0|_\tau(\R^n)+\varepsilon_j^{1/8}j2^{p_j}2^{-p_j}<M.
	\end{equation}
	Let $f_{1.5}:\R^n\to\Omega $ be the orthogonal projection onto $\Omega$. Define
	\begin{equation*}
		\widehat{T}^{j,\ell+1}=(f_{1.5})_\#\widetilde{T}^{j,\ell+1},
	\end{equation*}
	which is still a transportation network (by Lipschitz continuity of $f_{1.5}$ and \cref{FedererPushforward}). Note that
	\begin{equation}
		\label[ineq]{in1:6.111}
		|\widehat T^{j,\ell+1}|_\tau(\R^n)\leq |\widetilde T^{j,\ell+1}|_\tau(\R^n)
	\end{equation}
	by \cref{ProjCost}.
	Next, set (see \cite[(6.12)]{KT17} or \cite[Sec.\ 4.13]{B78})
	\begin{equation*}
		f_2=\textup{id}+\Dt_j h_{\tau,\varepsilon_j}(\widehat{T}^{j,\ell+1})\in C_c^\infty(\R^n;\R^n).
	\end{equation*}
	\Cref{hEstimates} implies the analogue of \cite[(6.13)]{KT17},
	\begin{equation}
		\label[ineq]{in1:6.13}
		|\Dt_jh_{\tau,\varepsilon_j}(\widehat T^{j,\ell+1})|\leq 2\varepsilon_j^{c_1-2}\qquad\textup{and}\qquad |\Jm(\Dt_jh_{\tau,\varepsilon_j}(\widehat T^{j,\ell+1}))|\leq 4\varepsilon_j^{c_1-4}.
	\end{equation}
	A simple application of Hadamard's global inverse function theorem yields that $f_2$ is a diffeomorphism.
	As in \cite[(6.14)]{KT17}, we can define another transportation network by
	\begin{equation}
		\label{6.14}
		T^{j,\ell+1}=(f_2)_\# \widehat T^{j,\ell+1}.
	\end{equation}
	By construction and \cref{hEstimates} we obtain $\spt(T^{j,\ell+1})\subset\{ \textup{dist}(\cdot,\Omega)\leq 2\varepsilon_{j}^{c_1-2} \}$ as desired.
	\Cref{in1:Q2,in1:6.111,in1:6.7} (in the first inequality) and the second inequality in \labelcref{in1:6.11} (in the second inequality) yield the analogue of \cite[(6.16)]{KT17},
	\begin{multline*}
		|T^{j,\ell+1}|_\tau(\R^n)=|(f_2)_\#\widehat T^{j,\ell+1}|_\tau(\R^n)\leq |T^{j,\ell}|_\tau(\R^n)+\Dt_j\varepsilon_j^{1/8}\leq |T_0|_\tau(\R^n)+\varepsilon_j^{1/8}\ell\Dt_j+\Dt_j\varepsilon_j^{1/8}\\
		=|T_0|_\tau(\R^n)+\varepsilon_j^{1/8}(\ell+1)\Dt_j,
	\end{multline*}
	which gives \cref{in1:6.3}. In particular, we have $|T^{j,\ell+1}|_\tau(\R^n)\leq M$ using the first inequality in \labelcref{in1:6.7}. Thus, \cref{in1:Q1,in1:Q2,in1:Q3,in1:Q4} are applicable to $\widehat{T}^{j,\ell+1}$. Using \cref{in1:6.9} (in the first inequality), \cref{6.14,6.10,in1:6.111} (in the second inequality), \cref{in1:Q4,in1:Q2} (in the third inequality), and last inequality in \labelcref{in1:6.7} (in the last estimate) we obtain
	\begin{multline*}
		\frac{|T^{j,\ell+1}|_\tau(\R^n)-|T^{j,\ell}|_\tau(\R^n)}{\Dt_j}+\frac{1}{4}\int\frac{| \Psi_{\varepsilon_j}\ast\delta_\tau T^{j,\ell+1} |^2}{\Psi_{\varepsilon_j}\ast|T^{j,\ell+1}|_\tau+\varepsilon_j  }\e\Lebn-\frac{1-j^{-5}}{\Dt_j}\Delta_j|T^{j,\ell}|_\tau(\R^n)\\
		\leq \frac{|T^{j,\ell+1}|_\tau(\R^n)-|T^{j,\ell}|_\tau(\R^n)}{\Dt_j}+\frac{1}{4}\int\frac{| \Psi_{\varepsilon_j}\ast\delta_\tau T^{j,\ell+1} |^2}{\Psi_{\varepsilon_j}\ast|T^{j,\ell+1}|_\tau+\varepsilon_j  }\e\Lebn+\frac{1}{\Dt_j}\left( |T^{j,\ell}|_\tau(\R^n)-|(f_1)_\# T^{j,\ell}|_\tau(\R^n) \right)\\
		\leq\frac{|(f_2)_\# \widehat T^{j,\ell+1}|_\tau(\R^n)-|\widehat T^{j,\ell+1}|_\tau(\R^n)}{\Dt_j}+\frac{1}{4}\int\frac{| \Psi_{\varepsilon_j}\ast\delta_\tau ((f_2)_\# \widehat T^{j,\ell+1}) |^2}{\Psi_{\varepsilon_j}\ast|(f_2)_\# \widehat T^{j,\ell+1}|_\tau+\varepsilon_j  }\e\Lebn\leq 3\varepsilon_j^{1/4}+\frac{1}{4}\varepsilon_j^{c_1-3n-15}<\varepsilon_j^{1/8},
	\end{multline*}
	which yields \cref{in1:6.4}. Finally, application of \cref{6.14,in1:6.111,in1:6.11} (in the first inequality), \cref{in1:Q1} (in the second inequality), \cref{6.14,in1:Q3} (in the third inequality), and last inequality in \labelcref{in1:6.7} (in the last estimate) gives
	\begin{multline*}
		\frac{|T^{j,\ell+1}|_\tau(\phi)-|T^{j,\ell}|_\tau(\phi)}{\Dt_j}\leq \frac{|(f_2)_\#\widehat T^{j,\ell+1}|_\tau(\phi)-|\widehat T^{j,\ell+1}|_\tau(\phi)}{\Dt_j}\leq\varepsilon_j^{c_1-10}+\delta_\tau(\widehat T^{j,\ell+1},\phi)(h_{\tau,\varepsilon_j}(\widehat T^{j,\ell+1}))\\
		\leq \varepsilon_j^{c_1-10}+\varepsilon_j^{c_1-2n-17}+\delta_\tau( T^{j,\ell+1},\phi)(h_{\tau,\varepsilon_j}( T^{j,\ell+1}))\leq 3\varepsilon_j^{1/4}+\varepsilon_j^{c_1-3n-15}+\delta_\tau( T^{j,\ell+1},\phi)(h_{\tau,\varepsilon_j}( T^{j,\ell+1}))\\
		<\varepsilon_j^{1/8}+\delta_\tau( T^{j,\ell+1},\phi)(h_{\tau,\varepsilon_j}( T^{j,\ell+1})).
	\end{multline*}
\end{proof}
The following statement corresponds to \cite[Prop.\ 6.4]{KT17}, see \cite[Secs.\ 4.13 \&\ 4.14]{B78}.
\begin{prop}
	\label{Prop6.4}
	Let the assumptions be as in \cref{Prop6.1}. For $j\in\mathbb N$ and $t\in[0,j]$ define
	\begin{equation*}
		T_t^j=T^{j,\lceil t/\Dt_j\rceil}.
	\end{equation*}
	Then there exists a subsequence $(j_k)$ and a family $(\mu_t)_{t\in\R_+}\subset\M_+(\R^n)$ such that $|T_t^{j_k}|_\tau\ws\mu_t$ for $k\to\infty$ for all $t\in\R_+$,
	\begin{equation}
		\label[ineq]{in1:6.19}
		\limsup_k\int_{[0,R]}\left( \int\frac{|\Psi_{\varepsilon_{j_k}}\ast\delta_\tau T_t^{j_k}|^2}{\Psi_{\varepsilon_{j_k}}\ast|T_t^{j_k}|_\tau+\varepsilon_{j_k}}\e\Lebn-\frac{1}{\Dt_{j_k}}\Delta_{j_k}|T_t^{j_k}|_\tau(\R^n) \right)\e\Leb(t)<\infty
	\end{equation}
	for all $R\in\R_+$, and
	\begin{equation}
		\label[ineq]{in1:6.20}
		\lim_k j_k^{2n}\Delta_{j_k}|T_t^{j_k}|_\tau(\R^n)=0
	\end{equation}
	for a.e.\ $t\in\R_+$.
\end{prop}
\begin{proof}
	Abbreviate $D=\{ i2^{-j}\:|\: i,j\in\mathbb N_0 \}$. Note that $D$ is dense in $\R_+$. By \cref{in1:6.3} we have
	\begin{equation*}
		\limsup_{j}\sup_{t\in\R_+}|T_t^j|_\tau(\R^n)\leq|T_0|_\tau(\R^n).
	\end{equation*}
	As in \cite[(6.21)]{KT17}, application of the Banach\textendash Alaoglu theorem and a diagonal argument yields the existence of a subsequence $(j_k)$ and a family $(\mu_t)_{t\in D}\subset\M_+(\R^n)$ such that $|T_t^{j_k}|_\tau\ws\mu_t$ for $k\to\infty$ for all $t\in D$. Note that $\mu_t(\R^n)\leq |T_0|_\tau(\R^n)$, which follows from the lower semicontinuity of the total variation with respect to weak-$*$ convergence. As in the proof of \cite[Prop.\ 6.4]{KT17}, we choose a sequence $(\phi_i)\subset C_c^2(\R^n;\R_+)$ such that $\{ \phi_i\:|\: i\in\mathbb N \}$ is dense in $C_c(\R^n;\R_+)$. We can actually assume $(\phi_i)\subset C_c^2(\R^n;\R_+^*)$ by adding countably many small numbers $\delta_m$ ($m\in\mathbb N$) with $\delta_m\to 0$ to each member of the sequence.
	For $i,J\in\mathbb N$ define (see \cite[(6.23)]{KT17})
	\begin{equation*}
		g_{i,J}(t)=\mu_t(\phi_i)-t||\nabla^2\phi_i|_2|_\infty |T_0|_\tau(\R^n)
	\end{equation*}
	for all $t\in[0,J]\cap D$. We claim that $g_{i,J}$ is monotone decreasing. Since $\phi_i$ attains a maximum, we can assume $\phi_i\in(0,1)$ by rescaling $g_{i,J}$. Following the calculation in \cite[(5.66)]{KT17}, we have (see \cite[(6.24)]{KT17})
	\begin{multline}
		\label[ineq]{in1:6.24}
		\delta_\tau(T_t^j,\phi)(h_{\tau,\varepsilon_j}(T_t^j))=\delta_\tau T_t^j(\phi h_{\tau,\varepsilon_j}(T_t^j))+\int_S h_{\tau,\varepsilon_j}(T_t^j)\cdot(\nabla\phi-T\cdot S\nabla\phi )\tau(|\theta|)\e\Ha\\
		\leq \delta_\tau T_t^j(\phi h_{\tau,\varepsilon_j}(T_t^j))+\frac{1}{2}\int\frac{|\nabla\phi|^2}{\phi}\e|T_t^j|_\tau+\frac{1}{2}\int | h_{\tau,\varepsilon_j}(T_t^j) |^2\phi\e|T_t^j|_\tau\\
		\leq 2\varepsilon_j^{1/4}+\left( \frac{3}{2}\varepsilon_j^{1/4}-\frac{1}{2} \right) I_\phi+\frac{1}{2}\int\frac{|\nabla\phi|^2}{\phi}\e|T_t^j|_\tau\leq 2\varepsilon_j^{1/4}+\frac{1}{2}\int\frac{|\nabla\phi|^2}{\phi}\e|T_t^j|_\tau
	\end{multline}
	for all $j\in\mathbb N$ and $\phi\in \mathcal F_j^+\setminus\{ 0 \}$, where $T_t^j=\theta\Ha\mres S$ and $I_\phi=\int\phi|\Psi_{\varepsilon_j}\ast\delta_\tau T_t^j|^2/(\Psi_{\varepsilon_j}\ast|T_t^j|_\tau+\varepsilon_j)\e\Lebn$. Here we used $2ab\leq a^2+b^2$ and $|\nabla\phi-T_\cdot S\nabla\phi|\leq|\nabla\phi|$ as well as \cref{in1:5.23,in1:5.24}.
	Recall that we assumed $\phi_i<1$. Take $i_0,j_0\in\mathbb N$ such that $\phi_i+i_0^{-1}<1$ and $j_0\geq\max\{ J,|\nabla\phi_i|,|\nabla^2\phi_i|_2 \}$. Then
	\begin{equation*}
		\phi_i+i_0^{-1}\in\mathcal F_{j_0}^+\qquad\textup{and}\qquad j_0\geq J.
	\end{equation*}
	Now take any $t_1,t_2\in[0,J]\cap D$ with $t_1<t_2$. Write $t_1=i_12^{-j_1},t_2=i_22^{-j_2}$. We can choose larger $j_0$ such that $j_0\geq\max\{ j_1,j_2 \}$. Then
	\begin{equation*}
		t_1,t_2\in 2^{-j_0}\mathbb N_0.
	\end{equation*}
	Using \cref{in1:6.5,in1:6.24}, we get (writing $t_1=n_12^{-j_0},t_2=n_22^{-j_0}$)
	\begin{multline}
		\label[ineq]{in1:6.25}
		|T_{t_2}^{j_k}|_\tau(\phi_i+i_0^{-1})-|T_{t_1}^{j_k}|_\tau(\phi_i+i_0^{-1})=|T^{j_k,\lceil t_2/\Dt_{j_k}\rceil }|_\tau(\phi_i+i_0^{-1})-|T^{j_k,\lceil t_1/\Dt_{j_k}\rceil }|_\tau(\phi_i+i_0^{-1})\\
		=\sum_{\ell=\lceil t_1/\Dt_{j_k}\rceil}^{\lceil t_2/\Dt_{j_k}\rceil-1}\left( |T^{j_k,\ell+1}|_\tau(\phi_i+i_0^{-1})-|T^{j_k,\ell}|_\tau(\phi_i+i_0^{-1}) \right)\\
		\leq \Dt_{j_k}\sum_{\ell=\lceil t_1/\Dt_{j_k}\rceil}^{\lceil t_2/\Dt_{j_k}\rceil-1}\left( \delta_\tau(T^{j_k,\ell+1},\phi_i+i_0^{-1})(h_{\tau,\varepsilon_{j_k}}(T^{j_k,\ell+1}))+\varepsilon_{j_k}^{1/8} \right)\\
		=\Dt_{j_k}\sum_{\ell=\lceil t_1/\Dt_{j_k}\rceil}^{\lceil t_2/\Dt_{j_k}\rceil-1}\left( \delta_\tau(T_{(\ell+1)\Dt_{j_k}}^{j_k},\phi_i+i_0^{-1})(h_{\tau,\varepsilon_{j_k}}(T_{(\ell+1)\Dt_{j_k}}^{j_k}))+\varepsilon_{j_k}^{1/8} \right)\\
		\leq \Dt_{j_k}\sum_{\ell=\lceil t_1/\Dt_{j_k}\rceil}^{\lceil t_2/\Dt_{j_k}\rceil-1}\left( 2\varepsilon_{j_k}^{1/4}+\frac{1}{2}\int\frac{|\nabla\phi_i|^2}{\phi+i_0^{-1}}\e|T_{(\ell+1)\Dt_{j_k}}^{j_k}|_\tau+\varepsilon_{j_k}^{1/8}  \right)\\
		=(t_2-t_1)(2\varepsilon_{j_k}^{1/4}+\varepsilon_{j_k}^{1/8})+\sum_{\ell=\lceil t_1/\Dt_{j_k}\rceil}^{\lceil t_2/\Dt_{j_k}\rceil-1}\Dt_{j_k}\frac{1}{2}\int\frac{|\nabla\phi_i|^2}{\phi+i_0^{-1}}\e|T_{(\ell+1)\Dt_{j_k}}^{j_k}|_\tau\\
		=(t_2-t_1)(2\varepsilon_{j_k}^{1/4}+\varepsilon_{j_k}^{1/8})+\frac{1}{2}\int_{[t_1,t_2]}\int\frac{|\nabla\phi_i|^2}{\phi+i_0^{-1}}\e|T_{t}^{j_k}|_\tau\e\Leb(t)
	\end{multline}
	for all $j_k\geq j_0$\footnote{The last equality follows from $t_1,t_2\in\Dt_{j_k}\mathbb N_0$ for $j_k\geq j_0$, which is true by $p_{j_k}\geq j_k$ (\cref{Prop6.1}). For example, we get $t_1=\Dt_{j_k}n_12^{p_{j_k}-j_k}2^{j_k-j_0}\in\Dt_{j_k}\mathbb N_0$.}, see \cite[6.25]{KT17}.
	Using $|T_t^{j_k}|_\tau\ws\mu_t$ for $t\in D$, \cref{in1:6.3}, and $\varepsilon_{j_k}<j_k^{-6}$, we get the analogue of \cite[(6.26)]{KT17},
	\begin{equation}
		\label[ineq]{in1:6.26}
		\mu_{t_2}(\phi_i)-\mu_{t_1}(\phi_i)-i_0^{-1}|T_0|_\tau(\R^n)\leq \lim_k |T_{t_2}^{j_k}|_\tau(\phi_i+i_0^{-1})-|T_{t_1}^{j_k}|_\tau(\phi_i+i_0^{-1}).
	\end{equation}
	Now \cite[(6.27)]{KT17} simplifies to
	\begin{equation}
		\label[ineq]{in1:6.27}
		\frac{|\nabla\phi_i|^2}{\phi_i+i_0^{-1}}\leq \frac{|\nabla\phi_i|^2}{\phi_i}\leq 2||\nabla^2\phi_i|_2|_\infty,
	\end{equation}
	where we used the formula $|\nabla f|^2\leq 2f||\nabla^2f|_2|_\infty$ for all $f\in C_c^2(\R^n;\R_+^*)$ (it can be shown via Taylor's theorem).
	By \cref{in1:6.25,in1:6.26,in1:6.27,in1:6.3} we have
	\begin{equation*}
		\mu_{t_2}(\phi_i)-\mu_{t_1}(\phi_i)\leq (t_2-t_1)||\nabla^2\phi_i|_2|_\infty|T_0|_\tau(\R^n)+i_0^{-1}|T_0|(\R^n).
	\end{equation*}
	Letting $i_0\to\infty$ (cf.\ \cite[(6.28)]{KT17}),
	\begin{equation*}
		\mu_{t_2}(\phi_i)-\mu_{t_1}(\phi_i)\leq (t_2-t_1)||\nabla^2\phi_i|_2|_\infty|T_0|_\tau(\R^n).
	\end{equation*}
	It follow that $g_{i,J}$ is monotone decreasing on $[0,J]\cap D$. Now define
	\begin{equation*}
		P=\bigcup_{J\in\mathbb N}\{  t\in(0,J)\:|\: g_{i,J}\textup{ not continuous in }t \}\subset\R_+.
	\end{equation*}
	Set $P$ is countable because a monotone $\R$-valued function defined on an interval has countably many discontinuities.
	First, we claim that $\lim_k|T_t^{j_k}|_\tau(\phi_i)=\mu_t(\phi_i)$ for all $t\in\R_+\setminus(P\cup D)$.
	Fix $t\in\R_+\setminus(P\cup D)$. Let $(t_k)\subset D$ such that $t_k\searrow t$ and $T_{t_k}^{j_k}=T_{t}^{j_k}$.
	Now take $s\in D$ with $t\leq t_k<s$ for $k$ sufficiently large.
	\Cref{in1:6.25} implies (see \cite[(6.31)]{KT17})
	\begin{equation*}
		|T_s^{j_k}|_\tau(\phi_i+i_0^{-1})\leq|T_{t_k}^{j_k}|_\tau(\phi_i+i_0^{-1})+ o(1)
	\end{equation*}
	as $s\searrow t$. Applying $\lim_{i_0}\liminf_k$ yields an inequality corresponding to \cite[(6.32)]{KT17},
	\begin{equation*}
		\mu_s(\phi_i)\leq\liminf_k|T_{t_k}^{j_k}|_\tau(\phi_i)+o(1)
	\end{equation*}
	as $s\searrow t$. Letting $s\searrow t$ gives (recall that $T_{t_k}^{j_k}=T_{t}^{j_k}$ and $t\mapsto \mu_t(\phi_i)$ is continuous because $t\notin P$)
	\begin{equation}
		\label[ineq]{in1:6.33}
		\mu_t(\phi_i)\leq\liminf_k|T_{t}^{j_k}|_\tau(\phi_i),
	\end{equation}
	compare with \cite[(6.33)]{KT17}.
	Using similar arguments one also obtains (see \cite[(6.36)]{KT17})
	\begin{equation}
		\label[ineq]{in1:6.36}
		\limsup_k|T_{t}^{j_k}|_\tau(\phi_i)\leq\mu_t(\phi_i).
	\end{equation}
	\Cref{in1:6.33,in1:6.36} show $\lim_k|T_t^{j_k}|_\tau(\phi_i)=\mu_t(\phi_i)$ for all $t\in\R_+\setminus(P\cup D)$. Since $\{ \phi_i\:|\:i\in\mathbb N \}$ is dense in $C_c(\R^n;\R_+)$, this relation is also true for $\phi_i$ replaced by any $\phi\in C_c(\R^n;\R_+)$. Thus, it also holds for $\phi\in C_c(\R^n)$ by decomposing into positive and negative part. Using the beginning of the proof, we get the validity for $t\in D$. Since $P$ is countable, using a diagonal argument again, there exist measures $\{ \mu_t\:|\: t\in P \}$ such that, up to a subsequence, $\lim_k|T_t^{j_k}|_\tau(\phi)=\mu_t(\phi)$ for all $\phi\in C_c(\R^n),t\in P$. In summary, this shows the first claim.\\
	Next, we show \cref{in1:6.19}.
	Multiplication of \cref{in1:6.4} with $\Dt_{j_k}$ and summation over $\ell$ gives
	\begin{multline*}
		\sum_{\ell=1}^{j_k2^{p_{j_k}}}\left( |T^{j_k,\ell}|_\tau(\R^n)-|T^{j_k,\ell-1}|_\tau(\R^n)+\frac{\Dt_{j_k}}{4}\int\frac{| \Psi_{\varepsilon_{j_k}}\ast\delta_\tau T^{j_k,\ell} |^2}{\Psi_{\varepsilon_{j_k}}\ast|T^{j_k,\ell}|_\tau+\varepsilon_{j_k}  }\e\Lebn-\Dt_{j_k}\frac{1-{j_k}^{-5}}{\Dt_{j_k}}\Delta_{j_k}|T^{j_k,\ell-1}|_\tau(\R^n)\right)\\ \leq\varepsilon_{j_k}^{1/8}\Dt_{j_k}j_k2^{p_{j_k}}=\varepsilon_{j_k}^{1/8}j_k < 1,
	\end{multline*}
	where the inequality follows from the first inequality in \labelcref{in1:6.7}. This becomes
	\begin{multline*}
		|T^{j_k,j_k2^{p_{j_k}}}|_\tau(\R^n)-|T_0|_\tau(\R^n)+\frac{1}{4}\int_{[0,j_k]}\int\frac{| \Psi_{\varepsilon_{j_k}}\ast\delta_\tau T_t^{j_k} |^2}{\Psi_{\varepsilon_{j_k}}\ast|T_t^{j_k}|_\tau+\varepsilon_{j_k}  }\e\Lebn\e\Leb(t)\\
		-\int_{[0,j_k-\Dt_{j_k}]}\frac{1-{j_k}^{-5}}{\Dt_{j_k}}\Delta_{j_k}|T_t^{j_k}|_\tau(\R^n)\e\Leb(t) -(1-{j_k}^{-5})|T_0|_\tau(\R^n) < 1.
	\end{multline*}
	In particular, we have (using $1/4<1-j_k^{-5}$ for $j_k>1$ and $\Delta_{j_k}|T_t^{j_k}|_\tau(\R^n)\in\R_-$)
	\begin{multline*}
		\int_{[0,j_k-\Dt_{j_k}]}\left(\int\frac{| \Psi_{\varepsilon_{j_k}}\ast\delta_\tau T_t^{j_k} |^2}{\Psi_{\varepsilon_{j_k}}\ast|T_t^{j_k}|_\tau+\varepsilon_{j_k}  }\e\Lebn-\frac{1}{\Dt_{j_k}}\Delta_{j_k}|T_t^{j_k}|_\tau(\R^n)\right)\e\Leb(t)\\
		\leq 4\left( \int_{[0,j_k-\Dt_{j_k}]}\left(\frac{1}{4}\int\frac{| \Psi_{\varepsilon_{j_k}}\ast\delta_\tau T_t^{j_k} |^2}{\Psi_{\varepsilon_{j_k}}\ast|T_t^{j_k}|_\tau+\varepsilon_{j_k}  }\e\Lebn-\frac{1-{j_k}^{-5}}{\Dt_{j_k}}\Delta_{j_k}|T_t^{j_k}|_\tau(\R^n)\right)\e\Leb(t) \right)\\
		\leq 4\left( 1+|T_0|_\tau(\R^n)+(1-{j_k}^{-5})|T_0|_\tau(\R^n) \right).
	\end{multline*}
	Hence, for every $R\in\R_+$ we get (eventually $R<j_k-\Dt_{j_k}$)
	\begin{multline*}
		\limsup_k\int_{[0,R]}\left(\int\frac{| \Psi_{\varepsilon_{j_k}}\ast\delta_\tau T_t^{j_k} |^2}{\Psi_{\varepsilon_{j_k}}\ast|T_t^{j_k}|_\tau+\varepsilon_{j_k}  }\e\Lebn-\frac{1}{\Dt_{j_k}}\Delta_{j_k}|T_t^{j_k}|_\tau(\R^n)\right)\e\Leb(t)\\
		\leq\limsup_k 4\left( 1+|T_0|_\tau(\R^n)+(1-{j_k}^{-5})|T_0|_\tau(\R^n) \right)=4(1+2|T_0|_\tau(\R^n)).
	\end{multline*}
	Recall that $\Dt_{j_k}\leq\varepsilon_{j_k}^{c_1}<(j_k^{-6})^{c_1}\ll j_k^{-2n}$.
	Since, by the above, we have that $\lim_k \int_{[0,R]}\frac{1}{\Dt_{j_k}}\Delta_{j_k}|T_t^{j_k}|_\tau(\R^n)\e\Leb(t)$ exists, we obtain
	\begin{equation*}
		\lim_k \int_{[0,R]}\frac{-1}{j_k^{-2n}}\Delta_{j_k}|T_t^{j_k}|_\tau(\R^n)\e\Leb(t)=0
	\end{equation*}
	for all $R\in\R_+$. Therefore, the integrand (whose values lie in $\R_+$) converges to $0$ for a.e.\ $t\in\R_+$, which yields \cref{in1:6.20}.
\end{proof}
\section{Proof of \cref{LimitCurrent}: transportation network geometric flow $(T_t)$}
\label{ExTNFlow}
We are now ready to give the proof of \cref{LimitCurrent}.
\begin{proof}
	\underline{STEP 1 (Application of Banach\textendash Alaoglu theorem):} Note that \cref{GrowthCondition} implies $\tau'(0)\geq C$.
	Hence, we have $|T|_\tau(\R^n)\geq C|T|(\R^n)$ for all $T\in\mathcal R_1(\R^n)$.
	In particular, we have $|T_t^j|(\R^n)\leq C^{-1}|T_t^j|_\tau(\R^n)\leq C^{-1}(|T_0|_\tau(\R^n)+1)$ by \cref{in1:6.3,in1:6.7}. Thus, the Banach\textendash Alaoglu theorem and a diagonal argument imply the existence of a subsequence $(j_k)$ and vector-valued measures $(T_t)_{t\in D}\subset \M(\R^n;\R^n)$ such that $T_t^{j_k}\ws T_t$ for $k\to\infty$ for all $t\in D$ ($D$ defined as in the beginning of the proof of \cref{Prop6.4}). Note that this directly yields $\partial T_t=0$ for all $t\in D$ using also \cref{ZeroBoundary}.
	We can assume that the subsequence $(j_k)$ also satisfies \cref{in1:6.19}.
	To simplify the notation, we restrict $j$ to the indices $j_k$.
	Using $\spt(T_t^j)\subset\{ \textup{dist}(\cdot,\Omega)<2\varepsilon_j^{c_1-2} \}$ (\cref{Prop6.1}), it is not difficult to see that $\spt(T_t)\subset\Omega$ for all $t\in D$.
	
	\underline{STEP 2 (Continuity property with respect to flat norm):}
	Let $t_1,t_2\in D$ with $t_1<t_2$. As in the proof of \cref{Prop6.4}, we can always take $j$ sufficiently large such that $t_1,t_2\in\Dt_j\mathbb N_0$, in particular $t_2=t_1+N_j\Dt_j$ for $N_j\in\mathbb N$. 
	Now fix $j$ sufficiently large with this property.
	For $\ell=0,\ldots,N_j-1$ we have
	\begin{equation*}
		T_{t_1+(\ell+1)\Dt_j}^j=(f_2^\ell)_\#(f_{1.5}^\ell)_\#(f_1^\ell)_\# T_{t_1+\ell\Dt_j}^j,\quad\widetilde{T}_{t_1+\ell\Dt_j}^j=(f_1^\ell)_\# T_{t_1+\ell\Dt_j}^j,\quad\text{and}\quad\widehat{T}_{t_1+\ell\Dt_j}^j=(f_{1.5}^\ell)\widetilde{T}_{t_1+\ell\Dt_j}^j.
	\end{equation*}
	Application of the triangle inequality gives
	\begin{equation}
		\label[ineq]{in1:smartIneq}
		|T_{t_1}^j-T_{t_2}^j|_K^\flat\leq |T_{t_1+0\cdot\Dt_j}^j-\widehat{T}_{t_1-1\cdot\Dt_j}^j|_K^\flat+\sum_{\ell=0}^{N_j-1}|\widehat{T}_{t_1+(\ell-1)\Dt_j}^j-\widehat{T}_{t_1+\ell\Dt_j}^j|_K^\flat+|\widehat{T}_{t_1+(N_j-1)\Dt_j}^j-T_{t_1+N_j\Dt_j}^j|_K^\flat.
	\end{equation}
	(If $t_1=0$, then replace first two summands by $|T_0^j-\widehat{T}_0^j|_K^\flat$.)
	Here, we have chosen compact and convex $K\subset\R^n$ such that the supports of all the involved currents are contained in $K$. This is possible because (formally) deformation of type $f_{1.5}$ always follows deformations of type $f_2$ and $f_1$ (in the first step, only $f_1$) and $|\Dt_jh_{\tau,\varepsilon_j}(\widehat{T})|\leq 2\varepsilon_j^{c_1-2}$ respectively $|f_1-\textup{id}|_\infty\leq j^{-2}$ (see \cref{AdmDeform}, also recall the definition of the homotopy $h_{f_1}$). Further, we can assume that always $\spt((h_{f_1})_\#(\llbracket 0,1\rrbracket\times T))\subset K$ (which we will use below) because
	\begin{equation*}
		\spt((h_{f_1})_\#(\llbracket 0,1\rrbracket\times T))\subset h_{f_1}(\spt(\llbracket 0,1\rrbracket\times T))=h_{f_1}(\spt(\llbracket 0,1\rrbracket)\times \spt(T)),
	\end{equation*}
	see \cite[pp.\ 359--360]{F}, thus (by the above) $\spt((h_{f_1})_\#(\llbracket 0,1\rrbracket\times T))\subset\{\textup{dist}(\cdot,\Omega) \leq  2\varepsilon_j^{c_1-2}+j^{-2} \}$.
	First, we claim that 
	\begin{equation}
		\label[ineq]{in1:Claim}
		|\widehat{T}_{t_1+(\ell-1)\Dt_j}^j-\widehat{T}_{t_1+\ell\Dt_j}^j|_K^\flat\leq |\widehat{T}_{t_1+(\ell-1)\Dt_j}^j-T_{t_1+\ell\Dt_j}^j|_K^\flat+|T_{t_1+\ell\Dt_j}^j-(f_1^\ell)_\# T_{t_1+\ell\Dt_j}^j|_K^\flat
	\end{equation}
	for all $\ell=0,\ldots,N_j-1$. For simplicity, write \cref{in1:Claim} formally as ($p$ for ``previous iteration'')
	\begin{equation}
		\label[ineq]{in1:PrevCur}
		|\widehat T_p-\widehat T|_K^\flat\leq 	|\widehat T_p- T|_K^\flat+	| T-(f_1)_\# T|_K^\flat,
	\end{equation}
	see \cref{fig4}.
	\begin{figure}
		\centering
		\begin{tikzpicture}[scale=2, transform shape=false]
			\filldraw[draw=lightgray, fill=lightgray] (-2,0) -- (2,0) -- (2,-1) -- (-2,-1);
			\node at (0,-0.5) {$\textcolor{gray}{\Omega}$};
			\draw (-2,0) -- (2,0);
			
			\node[fill,circle,minimum size=2mm,inner sep=0pt] at (-1,0) {};
			\node[fill,circle,minimum size=2mm,inner sep=0pt] at (-0.5,1) {};
			\node[fill,circle,minimum size=2mm,inner sep=0pt] at (1,2) {};
			\node[fill,circle,minimum size=2mm,inner sep=0pt] at (1,0) {};
			
			\draw (-1,0) -- (-0.5,1);
			\draw (-0.5,1) -- (1,2) ;
			 
			\node at (-1.2,0.2) {$\widehat T_p$};
			\node at (-1.2,1) {$T=(f_2^p)_\#\widehat T_p$};
			\node at (1.6,2) {$\widetilde T=(f_1)_\# T$};
			\node at (1.6,0.2) {$\widehat T=(f_{1.5})_\#\widetilde T$};
			
			\node at (-0.5,0.5) {$Q_1$};
			\node at (0.63,1.5) {$Q_2$};
		\end{tikzpicture}
		\caption{Formal sketch of the currents from \cref{in1:PrevCur}. Each $k$-current is indicated as a $(k-1)$-dimensional set.}
		\label{fig4}
	\end{figure}
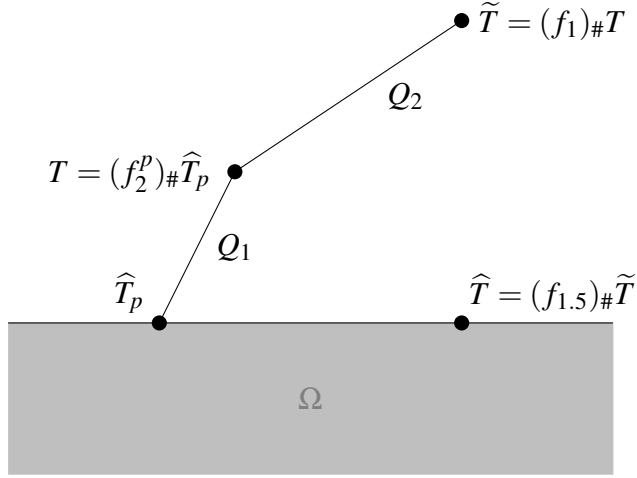
	Now recall that \cite[pp.\ 367]{F} (a minimizer always exists by \cite[pp.\ 368]{F})
	\begin{equation*}
		|\overline{T}|_K^\flat=\min\{ \mathbb M(\overline{T}-\partial Q)+\mathbb M(Q)\:|\: Q\in\D_2(\R^n),\textup{supp}(Q)\subset K \}\in[0,\infty]
	\end{equation*} 
	for all $\overline{T}\in\D_1(\R^n)$ with $\textup{supp}(\overline{T})\subset K$.
	Let $Q_1\in\mathcal N_2(\R^n)$ be a minimizer of $|\widehat T_p- T|_K^\flat$ and $Q_2\in\mathcal N_2(\R^n)$ be a minimizer of $| T-(f_1)_\# T|_K^\flat$. 
	We obtain (using also that the support of $\widehat T_p$ is contained in $\Omega$)
	\begin{multline*}
		|\widehat T_p-\widehat T|_K^\flat\leq \mathbb M((f_{1.5})_\# (Q_1+Q_2))+\mathbb M(\partial ((f_{1.5})_\# (Q_1+Q_2))-(\widehat T_p-\widehat T))\\
		=\mathbb M((f_{1.5})_\# (Q_1+Q_2))+\mathbb M((f_{1.5})_\#(\partial (Q_1+Q_2)-(\widehat T_p-\widetilde T)))\\
		\leq \mathbb M(Q_1+Q_2)+\mathbb M(\partial (Q_1+Q_2)-(\widehat T_p-\widetilde T))=\mathbb M(Q_1+Q_2)+\mathbb M(\partial (Q_1+Q_2)-(\widehat T_p-T+T-\widetilde T))\\
		\leq \mathbb M(Q_1)+\mathbb M(\partial Q_1-(\widehat T_p-T))+\mathbb M(Q_2)+\mathbb M(\partial Q_2-(T-\widetilde T))=|\widehat T_p- T|_K^\flat+	| T-(f_1)_\# T|_K^\flat.
	\end{multline*}
	This shows \cref{in1:Claim}. 
	(If $t_1=0$, then $|T_0^j-\widehat T_0^j|_K^\flat\leq |T_0^j-(f_1^0)_\# T_0^j|_K^\flat$ using a similar argument.)
	Thus, \cref{in1:smartIneq} can be written (if $t_1=0$, then replace first three summands by $|T_0^j-(f_1^0)_\# T_0^j|_K^\flat$)
	\begin{multline}
		\label[ineq]{in1:smartIneq2}
		|T_{t_1}^j-T_{t_2}^j|_K^\flat\leq |T_{t_1+0\cdot\Dt_j}^j-\widehat{T}_{t_1-1\cdot\Dt_j}^j|_K^\flat+\sum_{\ell=0}^{N_j-1}\left(|\widehat{T}_{t_1+(\ell-1)\Dt_j}^j-T_{t_1+\ell\Dt_j}^j|_K^\flat+|T_{t_1+\ell\Dt_j}^j-(f_1^\ell)_\# T_{t_1+\ell\Dt_j}^j|_K^\flat\right)\\
		+|\widehat{T}_{t_1+(N_j-1)\Dt_j}^j-T_{t_1+N_j\Dt_j}^j|_K^\flat.
	\end{multline}
	Next, we estimate ($f_2^{-1}$ corresponds to the deformation prior to $f_1^0$)
	\begin{itemize}
		\item $|T_{t_1+\ell\Dt_j}^j-(f_1^\ell)_\# T_{t_1+\ell\Dt_j}^j|_K^\flat$ for $\ell=0,1,\ldots,N_j-1$ and
		\item $|\widehat{T}_{t_1+(\ell-1)\Dt_j}^j-T_{t_1+\ell\Dt_j}^j|_K^\flat=|\widehat{T}_{t_1+(\ell-1)\Dt_j}^j-(f_2^{\ell-1})_\#\widehat{T}_{t_1+(\ell-1)\Dt_j}^j|_K^\flat$ for $\ell=0,1,\ldots,N_j$.
	\end{itemize}
	(If $t_1=0$, then $\ell$ starts from $1$ in second bullet point.) Let us start by considering the terms $|T_{t_1+\ell\Dt_j}^j-(f_1^\ell)_\# T_{t_1+\ell\Dt_j}^j|_K^\flat$ for $\ell=0,1,\ldots,N_j-1$. To simplify the notation, we write them (formally) as $|T-(f_1)_\# T|_K^\flat$. The homotopy formula for currents is applicable to $T-(f_1)_\# T$ by \cite[p.\ 371]{F} (essentially because $T\in \mathbb F_1(\R^n)$). It is stated on \cite[p.\ 363]{F},
	\begin{equation}
		\label{homotopy2}
		(f_1)_\# T-T=\partial h_\#(\llbracket 0,1\rrbracket\times T)+ h_\#(\llbracket 0,1\rrbracket\times \partial T),
	\end{equation}
	where $h=h_{f_1}$. Now we use that $\partial T=0$ by \cref{ZeroBoundary}. The current $\llbracket 0,1\rrbracket\times 0\in\D_1(\R\times\R^n)$ is uniquely characterized by the conditions on \cite[p.\ 360]{F}. Using these, it is not difficult to show that $\llbracket 0,1\rrbracket\times \partial T=0$, thus the last term in \cref{homotopy2} vanishes. Now recall that always $|\cdot|_K^\flat\circ\partial\leq|\cdot|_K^\flat$ \cite[p.\ 367]{F}. 
	We obtain (using $Q=0$ in the formula from above)
	\begin{multline*}
		|(f_1)_\# T-T|_{K}^\flat=|\partial h_\#(\llbracket 0,1\rrbracket\times T)|_{K}^\flat \leq | h_\#(\llbracket 0,1\rrbracket\times T)|_{K}^\flat \leq\mathbb M(h_\#(\llbracket 0,1\rrbracket\times T))\\
		\leq j^{-1}(|T|_\tau(\R^n)-|(f_1)_\# T|_\tau(\R^n)),
	\end{multline*}
	where the last inequality follows from \cref{AdmDeform}.
	Using this and \cref{in1:6.4,in1:6.3}, we get
	\begin{multline}
		\sum_{\ell=0}^{N_j-1}|T_{t_1+\ell\Dt_j}^j-(f_1^\ell)_\# T_{t_1+\ell\Dt_j}^j|_K^\flat\leq j^{-1}\sum_{\ell=0}^{N_j-1}(|T_{t_1+\ell\Dt_j}^j|_\tau(\R^n)-|(f_1^\ell)_\# T_{t_1+\ell\Dt_j}^j|_\tau(\R^n))\\
		\leq j^{-1}\sum_{\ell=0}^{N_j-1}|\Delta_j |T^{j,t_1/\Dt_j+\ell}|_\tau(\R^n)|\\
		\leq (1-j^{-5})^{-1} j^{-1}\sum_{\ell=0}^{N_j-1}(|T^{j,t_1/\Dt_j+\ell}|_\tau(\R^n)-|T^{j,t_1/\Dt_j+\ell+1}|_\tau(\R^n))+(1-j^{-5})^{-1}j^{-1}\sum_{\ell=0}^{N_j-1}\Dt_j\varepsilon_j^{1/8}\\
		\leq 2j^{-1}|T^{j,t_1/\Dt_j}|_\tau(\R^n)+2j^{-1}\varepsilon_j^{1/8}(t_2-t_1)\leq 2j^{-1}(|T_0|_\tau(\R^n)+\varepsilon_j^{1/8}t_1)+2j^{-1}\varepsilon_j^{1/8}(t_2-t_1).\label[ineq]{in1:FirstSum2}
	\end{multline}
	Next, we estimate the terms $|\widehat{T}_{t_1+(\ell-1)\Dt_j}^j-(f_2^{\ell-1})_\#\widehat{T}_{t_1+(\ell-1)\Dt_j}^j|_K^\flat$ for $\ell=0,1,\ldots,N_j$.
	As before, we (formally) write $|\widehat{T}_{t_1+(\ell-1)\Dt_j}^j-(f_2^{\ell-1})_\#\widehat{T}_{t_1+(\ell-1)\Dt_j}^j|_K^\flat$ as $|\widehat{T}-(f_2)_\#\widehat{T}|_K^\flat$.
	We apply \cite[Rmk.\ 4.1.13]{F} with $g=f_2,f=\textup{id},$ and
	\begin{equation*}
		A=\{ (1-t)x+tf_2(x)\:|\:x\in\spt(\widehat{T}),t\in[0,1] \}\subset K.
	\end{equation*}
	Using \cite[p.\ 359]{F}, we obtain
	\begin{equation*}
		\spt((f_2)_\#\widehat{T}-\widehat{T})\subset\spt((f_2)_\#\widehat{T})\cup\spt(\widehat{T})\subset f_2(\spt(\widehat{T}))\cup\spt(\widehat{T})\subset A.
	\end{equation*}
	Note that (using \cite[p.\ 367]{F}) $|(f_2)_\#\widehat{T}-\widehat{T}|_{K}^\flat\leq |(f_2)_\#\widehat{T}-\widehat{T}|_A^\flat$. Application of \cite[Rmk.\ 4.1.13]{F} gives
	\begin{multline*}
		|(f_2)_\#\widehat{T}-\widehat{T}|_{K}^\flat\leq |(f_2)_\#\widehat{T}-\widehat{T}|_A^\flat\leq \Dt_j\int |h_{\tau,\varepsilon_j}(\widehat{T})|\max\{ 1,|\mathbb 1+\Dt_j\Jm h_{\tau,\varepsilon_j}(\widehat{T})|_2 \}\e|\widehat{T} |\\
		\leq \Dt_j(1+\Dt_j4\varepsilon_j^{-4})\int |h_{\tau,\varepsilon_j}(\widehat{T})|\e|\widehat{T} |,
	\end{multline*}
	where we also used \cref{hEstimates} (which is applicable by the beginning of the proof of \cref{Prop6.1}). 
	This gives (using $\Dt_j\varepsilon_j^{-4}\leq\varepsilon_j^{c_1-4}\leq 1$)
	\begin{multline*}
		|T_{t_1+0\cdot\Dt_j}^j-\widehat{T}_{t_1-1\cdot\Dt_j}^j|_K^\flat+\sum_{\ell=0}^{N_j}|\widehat{T}_{t_1+(\ell-1)\Dt_j}^j-T_{t_1+\ell\Dt_j}^j|_K^\flat\\
		=|(f_2^{-1})_\#\widehat{T}_{t_1-1\cdot\Dt_j}^j-\widehat{T}_{t_1-1\cdot\Dt_j}^j|_K^\flat+\sum_{\ell=0}^{N_j}|\widehat{T}_{t_1+(\ell-1)\Dt_j}^j-(f_2^{\ell-1})_\#\widehat{T}_{t_1+(\ell-1)\Dt_j}^j|_K^\flat\\
		\leq 10\sum_{\ell=0}^{N_j}\Dt_j\int |h_{\tau,\varepsilon_j}(\widehat{T}_{t_1+(\ell-1)\Dt_j}^j)|\e|\widehat{T}_{t_1+(\ell-1)\Dt_j}^j|=10\int_{[t_1-2\Dt_j,t_2-\Dt_j]}\int |h_{\tau,\varepsilon_j}(\widehat{T}_t^j)|\e|\widehat{T}_t^j |\e\Leb(t).
	\end{multline*}
	By \cref{GrowthCondition}, $f_1^\ell\in\adm_j(T_{t_1+\ell\Dt_j}^j)$, and \cref{in1:6.3}, we have
	\begin{multline*}
		|\widehat{T}_t^j|(\R^n)=|(f_{1.5})_\# \widetilde{T}_t^j|(\R^n)\leq | \widetilde{T}_t^j|(\R^n)=| (f_1^\ell)_\# T_{t_1+\ell\Dt_j}^j|(\R^n)\leq C^{-1} |(f_1^\ell)_\# T_{t_1+\ell\Dt_j}^j|_\tau(\R^n)\leq\\
		C^{-1}|T_{t_1+\ell\Dt_j}^j|_\tau(\R^n)\leq C^{-1}(|T_0|_\tau(\R^n)+\varepsilon_j^{1/8}(t_1+\ell\Dt_j))
	\end{multline*}
	if $t\in( t_1+(\ell-1)\Dt_j,t_1+\ell\Dt_j]$. 
	In particular, we obtain (using the first inequality in \labelcref{in1:6.7})
	\begin{equation*}
		|\widehat{T}_t^j|(\R^n)\leq  C^{-1}(|T_0|_\tau(\R^n)+\varepsilon_j^{1/8}j)<  C^{-1}M 
	\end{equation*}
	for all $t\in [t_1-2\Dt_j,t_2-\Dt_j]$ (see integral above), where $M=|T_0|_\tau(\R^n)+1$.
	Using the above estimate and the Cauchy\textendash Schwarz inequality, we have 
	\begin{multline*}
		10\int_{[t_1-2\Dt_j,t_2-\Dt_j]}\int |h_{\tau,\varepsilon_j}(\widehat{T}_t^j)|\e|\widehat{T}_t^j |\e\Leb(t)\\
		\leq 10C^{-1/2}M^{1/2} \int_{[t_1-2\Dt_j,t_2-\Dt_j]}  \left(\int |h_{\tau,\varepsilon_j}(\widehat{T}_t^j)|^2\e|\widehat{T}_t^j |\right)^{1/2}\e\Leb(t).
	\end{multline*}
	Note that also $|\widehat T_t^j |(\phi)\leq C^{-1}|\widehat T_t^j |_\tau(\phi)$ for all $\phi\in C_c(\R^n)$ by \cref{GrowthCondition} (used in first inequality below).
	Further, we can use \cref{in1:5.24} (in second inequality below) since 
	\begin{equation*}
		|\widehat{T}_t^j|_\tau(\R^n)\leq|(f_1^\ell)_\# T_{t_1+\ell\Dt_j}^j|_\tau(\R^n)\leq |T_{t_1+\ell\Dt_j}^j|_\tau(\R^n)\leq M
	\end{equation*}
	if $t\in( t_1+(\ell-1)\Dt_j,t_1+\ell\Dt_j]$ (see \cref{in1:6.11,in1:6.111}). This yields 
	\begin{multline*}
		10C^{-1/2}M^{1/2} \int_{[t_1-2\Dt_j,t_2-\Dt_j]}  \left(\int |h_{\tau,\varepsilon_j}(\widehat{T}_t^j)|^2\e|\widehat{T}_t^j |\right)^{1/2}\e\Leb(t)\\
		\leq 10C^{-1}M^{1/2} \int_{[t_1-2\Dt_j,t_2-\Dt_j]}  \left(\int |h_{\tau,\varepsilon_j}(\widehat{T}_t^j)|^2\e|\widehat{T}_t^j |_\tau\right)^{1/2}\e\Leb(t)\\
		\leq 10C^{-1}M^{1/2}\int_{[t_1-2\Dt_j,t_2-\Dt_j]} \left(    (1+\varepsilon_j^{1/4})\int\frac{|\Psi_{\varepsilon_j}\ast\delta_\tau \widehat{T}_t^j|^2}{\Psi_{\varepsilon_j}\ast|\widehat{T}_t^j|_\tau+\varepsilon_j}\e\Lebn+\varepsilon_j^{1/4}    \right)^{1/2}\e\Leb(t)\\
		\leq 10C^{-1}M^{1/2}
		\int_{[t_1-2\Dt_j,t_2-\Dt_j]}  \left(    (1+\varepsilon_j^{1/4})\left( \int\frac{|\Psi_{\varepsilon_j}\ast\delta_\tau (f_2^{\ell(t)})_\#\widehat{T}_t^j|^2}{\Psi_{\varepsilon_j}\ast|(f_2^{\ell(t)})_\#\widehat{T}_t^j|_\tau+\varepsilon_j}\e\Lebn +\varepsilon_j^{c_1-3n-15} \right)+\varepsilon_j^{1/4}    \right)^{1/2}\e\Leb(t)\\
		\leq 10C^{-1}M^{1/2}
		\int_{[t_1-2\Dt_j,t_2-\Dt_j]}  (1+\varepsilon_j^{1/8})\left(\left( \int\frac{|\Psi_{\varepsilon_j}\ast\delta_\tau (f_2^{\ell})_\#\widehat{T}_t^j|^2}{\Psi_{\varepsilon_j}\ast|(f_2^{\ell})_\#\widehat{T}_t^j|_\tau+\varepsilon_j}\e\Lebn\right)^{1/2} +\varepsilon_j^{(c_1-3n-15)/2} \right)+\varepsilon_j^{1/8}\e\Leb,
	\end{multline*}
	where we also applied \cref{in1:Q4} (in the third inequality) and the subadditivity of the square root (last inequality). 
	Here $\ell=\ell(t)$ is such that $t\in( t_1+(\ell(t)-1)\Dt_j,t_1+\ell(t)\Dt_j]$.
	Application of the Cauchy\textendash Schwarz inequality gives
	\begin{multline*}
		\int_{[t_1-2\Dt_j,t_2-\Dt_j]}   \left( \int\frac{|\Psi_{\varepsilon_j}\ast\delta_\tau (f_2^{\ell(t)})_\#\widehat{T}_t^j|^2}{\Psi_{\varepsilon_j}\ast|(f_2^{\ell(t)})_\#\widehat{T}_t^j|_\tau+\varepsilon_j}\e\Lebn\right)^{1/2} \e\Leb(t)\\
		\leq (t_2-t_1+\Dt_j)^{1/2}\left(\int_{[t_1-2\Dt_j,t_2-\Dt_j]}  \int\frac{|\Psi_{\varepsilon_j}\ast\delta_\tau (f_2^{\ell(t)})_\#\widehat{T}_t^j|^2}{\Psi_{\varepsilon_j}\ast|(f_2^{\ell(t)})_\#\widehat{T}_t^j|_\tau+\varepsilon_j}\e\Lebn \e\Leb(t)\right)^{1/2}.
	\end{multline*}
	The latter integral is bounded by some $C_*\in\R_+$ (\cref{in1:6.19}). In summary, using \cref{in1:smartIneq2,in1:FirstSum2} plus the above, we have
	\begin{multline}
		|T_{t_1}^j-T_{t_2}^j|_K^\flat\leq 
		\sum_{\ell=0}^{N_j-1}|T_{t_1+\ell\Dt_j}^j-(f_1^\ell)_\# T_{t_1+\ell\Dt_j}^j|_K^\flat
		+|T_{t_1+0\cdot\Dt_j}^j-\widehat{T}_{t_1-1\cdot\Dt_j}^j|_K^\flat+\sum_{\ell=0}^{N_j}|\widehat{T}_{t_1+(\ell-1)\Dt_j}^j-T_{t_1+\ell\Dt_j}^j|_K^\flat\\
		\leq 2j^{-1}(|T_0|_\tau(\R^n)+\varepsilon_j^{1/8}t_1)+2j^{-1}\varepsilon_j^{1/8}(t_2-t_1)+|T_{t_1+0\cdot\Dt_j}^j-\widehat{T}_{t_1-1\cdot\Dt_j}^j|_K^\flat+\sum_{\ell=0}^{N_j}|\widehat{T}_{t_1+(\ell-1)\Dt_j}^j-T_{t_1+\ell\Dt_j}^j|_K^\flat\\
		\leq 2j^{-1}(|T_0|_\tau(\R^n)+\varepsilon_j^{1/8}t_1)+2j^{-1}\varepsilon_j^{1/8}(t_2-t_1)+10C^{-1}M^{1/2}\cdot\\
		\left( (1+\varepsilon_j^{1/8})(t_2-t_1+\Dt_j)^{1/2}C_*+(1+\varepsilon_j^{1/8})\varepsilon_j^{(c_1-3n-15)/2}(t_2-t_1+\Dt_j)+\varepsilon_j^{1/8}(t_2-t_1+\Dt_j)  \right)\\
		\leq 2j^{-1}M+2\varepsilon_j^{1/8}+10C^{-1}M^{1/2}
		\left( 2(t_2-t_1)^{1/2}C_*+2\varepsilon_j^{5/2}(t_2-t_1)+\varepsilon_j^{1/8}(t_2-t_1)  \right).\label[ineq]{in1:lastIn}
	\end{multline}
	This gives the following continuity property: If $\varepsilon>0$ and $t_1\in D$, then there exists $\delta>0$ such that for each $t_2\in D$ with $|t_1-t_2|<\delta$ there is some $j_0\in\mathbb N$ such that $|T_{t_1}^j-T_{t_2}^j|_K^\flat<\varepsilon$ for all $j\geq j_0$.
	
	\underline{STEP 3 (Continuation to $t\in\R_+$):} Let $s\in\R_+\setminus D$. Pick arbitrary sequences $(s_i),(\bar s_j)\subset D$ with $s_i\to s$ and $\bar s_j\to s$. Recall that (see STEP 1 and beginning of STEP 2) $\mathbb M(T)+\mathbb M(\partial T)\leq C^{-1}(|T_0|_\tau(\R^n)+1)$ and $\spt(T)\subset K$ for $T=T_{s_i}$ or $T=T_{\bar s_j}$. Using Federer's compactness theorem \cite[Thm.\ 4.2.17]{F}, we obtain
	\begin{equation*}
		|T_{s_{i_\ell}}-T|_K^\flat+|T_{\bar s_{j_\ell}}-\overline T|_K^\flat\to 0
	\end{equation*}
	for $\ell\to\infty$ for certain $T,\overline{T}\in\mathcal N_1(\R^n)$ with $\partial T=\partial \overline{T}=0$ and subsequences $(s_{i_\ell})\subset (s_{i})$ and $(\tilde s_{j_\ell})\subset (\tilde s_{j})$.
	We have
	\begin{equation*}
		|T-\overline{T}|_K^\flat\leq |T-T_{s_{i_\ell}}|_K^\flat+|T_{s_{i_\ell}}-T_{s_{i_\ell}}^j|_K^\flat+|T_{s_{i_\ell}}^j-T_{\bar s_{j_\ell}}^j|_K^\flat+|T_{\bar s_{j_\ell}}^j-T_{\bar s_{j_\ell}}|_K^\flat+|T_{\bar s_{j_\ell}}-\overline{T}|_K^\flat.
	\end{equation*}
	Now the right-hand side can be made arbitrarily small by choosing $\ell, j$ sufficiently large (invoking also STEP 2). Recall that the weak-$*$ convergence of the sequences $(T_t^j)$ for $j\to\infty$ ($t\in D$) implies the convergence with respect to $|\cdot|_K^\flat$ because (as above) $\mathbb M(T_t^j)+\mathbb M(\partial T_t^j)\leq C^{-1}(|T_0|_\tau(\R^n)+1), \spt(T_t^j)\subset K$, and $T_t^j\in\mathcal N_1(\R^n)$ (see, for example, the summary in \cite[Rmk.\ 3]{MW19}). We set $T_s=T=\overline{T}$. Finally, it is not difficult to check that $\spt(T_s)\subset\Omega$ (because this is true for the $T_{s_i}$ by STEP 1).
	
	\underline{STEP 4 (Hölder continuity of $t\mapsto T_t$):} Let $s,t\in\R_+$ and $(s_i),(t_i)\subset D$ with $s_i\to s$ and $t_i\to t$. Fix $\varepsilon>0$. Using the triangle inequality and \cref{in1:lastIn}, we get
	\begin{equation*}
		|T_s-T_t|_K^\flat\leq |T_{s_i}-T_{t_i}|_K^\flat+\varepsilon\leq C_{i,j}+C_0|s_i-t_i|^{1/2}+\varepsilon
	\end{equation*}
	for $i,j$ sufficiently large, where $C_0\in\R_+$ is a constant and $(C_{i,j})\subset\R_+^*$ with $C_{i,j}\to 0$ for $j\to \infty$. Letting $j\to \infty$ and then $i\to \infty$ yields (since $\varepsilon>0$ was arbitrary)
	\begin{equation*}
		|T_s-T_t|_K^\flat\leq C_0|s-t|^{1/2}.
	\end{equation*}
	\underline{STEP 5 ($T_s^j\ws T_s$ for all $s\in\R_+\setminus D$):} Fix $s\in\R_+\setminus D$. Note that $s\in\R_+^*$. For every $j\in\mathbb N$ with $s\in[0,j]$ we define $s_j=(\ell_j+1)\Dt_j$, where $\ell_j\in\mathbb N_0$ is chosen such that $s\in(\ell_j\Dt_j,(\ell_j+1)\Dt_j]$. By construction we get $s_j\to s$ and $T_{s_j}^j=T_s^j$. 
	We have (triangle inequality) 
	\begin{equation*}
		|T_s-T_s^j|_K^\flat\leq |T_s-T_{s_j}|_K^\flat+|T_{s_j}-T_{\tilde s}|_K^\flat+|T_{\tilde s}-T_{\tilde s}^j|_K^\flat+|T_{\tilde s}^j-T_{s_j}^j|_K^\flat+|T_{s_j}^j-T_s^j|_K^\flat
	\end{equation*}
	for all $\tilde s\in\R_+$.
	The first summand on the right-hand side goes to zero by $s_j\to s$. The second term can be made arbitrarily small for $\tilde s\in D$ sufficiently close to $s_j$ and $j$ sufficiently large (STEP 4). The third term converges to zero by $\tilde s\in D$. For the fourth term we can use \cref{in1:lastIn} again\footnote{If $\tilde s=i2^{-k}$, then $\tilde s\in2^{-j}\mathbb N_0\subset 2^{-p_j}\mathbb N_0=\Dt_j\mathbb N_0$ for $j\geq k$ due to $p_j\geq j$ (\cref{Prop6.1}).}. The last term is equal to zero by $T_{s_j}^j=T_s^j$. This shows $|T_s-T_s^j|_K^\flat\to 0$ implying the weak-$*$ convergence (since all the involved currents are normal $1$-currents with uniformly bounded mass and zero boundary and supports lying in $K$).
	\underline{$T_0\in\mathcal I_1(\R^n)\implies T_t\in\mathcal I_1(\R^n)$ for all $t\in\R_+$:} The property of being a transportation network is stable with respect to deformations of type $f_1,f_{1.5}$, and $f_2$ (recall the proof of \cref{Prop6.1}). Hence, all the currents $T_t^j=T^{j,\lceil t/\Dt_j\rceil}$ ($t\in[0,j]$) from \cref{Prop6.4} are transportation networks. By the representation of the multiplicity in \cref{FedererPushforward}, we also have $T_t^j\in \mathcal{I}_1(\R^n)$ for all $t\in[0,j]$ if $T_0\in\mathcal I_1(\R^n)$. Finally, STEPS 1 and 3 together with the closure theorem for integral currents \cite[Thm.\ 8.12]{FF60} imply the claim.
\end{proof}
\section{Rectifiability of $(V_t)$ with $\| V_t\|=\mu_t$}
\label{RectifiablitySection}
In this section, we will give the proof of \cref{rectifiability}. The idea is to combine estimates similar to \cite[Lem.\ 5.2]{KT20} and \cite[Lem.\ 5.4]{KT20} with Allard's rectifiability theorem \cite[Sec.\ 5.5, Thm.\ 1]{A72}. However, since our transportation networks are not necessarily concentrated on closed sets, which is used in the proof of \cite[Lem.\ 5.2]{KT20}, we will need a different construction. To this end, we will use the following statement.
\begin{lem}
	\label{GrowthSBall}
	Let $S\subset\R^n$ be countably $1$-rectifiable\footnote{There exist countably many Lipschitz functions $f_i:\R\to\R^n$ such that $\mathcal H^1(S\setminus\bigcup_if_i(\R))=0$.} and $x\in\R^n$. Define
	\begin{equation*}
		f(r)=\mathcal H^1(\overline{B}_r(x)\cap S)\in[0,\infty]
	\end{equation*}
	for $r\in\R_+$. Then we have $f'(r)=0$ or $f'(r)\geq 1$ for a.e.\ $r\in\dom(f)$. In particular, if $f(R)<\lambda R$ for some $\lambda\in(0,1]$ and $R\in\R_+$, then $\Leb([0,R]\cap \{ f'=0 \})> (1-\lambda)R$.
	Further, if $T=\theta\Ha\mres S\in\mathcal I_1(\R^n)$ and $\kappa:\R_+\to\R_+$ a transportation cost, then
	\begin{equation*}
		g(r)=\int_{\overline{B}_r(x)\cap S}\kappa(|\theta|)\e\Ha
	\end{equation*}
	satisfies $g'(r)=0$ or $g'(r)\geq \kappa(1)$ for a.e.\ $r\in\dom(g)$. In particular, if $g(R)<\lambda R$ for some $\lambda\in(0,\tau(1)]$ and $R\in\R_+$, then $\Leb([0,R]\cap \{ g'=0 \})> (1-\lambda/\kappa(1))R$.
\end{lem}
\begin{proof}
	Let $\lambda\in(0,1]$ and $R\in\R_+$ such that $f(R)<\lambda R$. 
	We show $\Leb([0,R]\cap \{ f'=0 \})> (1-\lambda)R$ assuming that we already know that $f'(r)=0$ or $f'(r)\geq 1$ for a.e.\ $r\in\dom(f)$. 
	Write $B=[0,R]\cap \{ f'=0 \}$. Note that $f$ is non-decreasing. Thus 
	\begin{equation*}
		\lambda R>f(R)\geq\int_{[0,R]}f'(r)\e\Leb(r)=\int_{B}f'(r)\e\Leb(r)+\int_{[0,R]\setminus B}f'(r)\e\Leb(r)\geq \int_{[0,R]\setminus B}1\e\Leb(r)=R-\Leb(B) ,
	\end{equation*}
	This gives $\Leb([0,R]\cap \{ f'=0 \})> (1-\lambda)R$.\\ 
	To show $f'(r)=0$ or $f'(r)\geq 1$ for a.e.\ $r\in\dom(f)$, we apply \cite[Lem.\ 3.2.18]{F} (note that $S$ is countably $(\mathcal H^1,1)$ rectifiable in the sense of \cite[Def.\ 3.2.14]{F}). Up to some $\mathcal H^1$-zero set (which we can neglect by definition of $f$), we obtain $S=\bigcup_j\psi_j(K_j)$, where $K_j\subset\R$ compact, $\psi_j:\R\to\R^n$ Lipschitz, and $\psi_j(K_j)$ mutually disjoint ($j\in\mathbb N$). Now define $S_j=\psi_j(K_j)$ and $f_j(r)=\mathcal H^1(\overline{B}_r(x)\cap S_j)$. This gives (because $f=\sum_{j }f_j$ with non-decreasing $f_j$, cf. \cite[Ex.\ 222Y (a)]{Frem03})
	\begin{equation*}
		f'(r)=\sum_{j=1}^{\infty}f_j'(r)
	\end{equation*}
	for a.e.\ $r\in\dom(f)$. Hence, for each $j\in\mathbb N$, we only need to show that $f_j'(r)=0$ or $f_j'(r)\geq 1$ for a.e.\ $r\in\dom(f)$. By \cite[Thm.\ 1.8, Ch.\ 2, \S 1]{LS} (with $A=(\overline{B}_{r+h}(x)\setminus\overline{B}_r(x))\cap S_j$ and $|\cdot|:A\to(r,r+h]$) we get (for $h\in\R_+^*$)
	\begin{equation*}
		\frac{f_j(r+h)-f_j(r)}{h}=\frac{1}{h}\mathcal H^1((\overline{B}_{r+h}(x)\setminus\overline{B}_r(x))\cap S_j)\geq\frac{1}{h}\int_{(r,r+h]}\mathcal H^0(\partial B_s(x)\cap S_j)\e\Ha(s).
	\end{equation*}
	By Lebesgue's differentiation theorem the right-hand side converges to $\mathcal H^0(\partial B_r(x)\cap S_j)$ for a.e.\ $r\in\dom(f)$. Hence, we only need to show that $f_j'(r)>0$ implies $\mathcal H^0(\partial B_r(x)\cap S_j)>0$ for a.e.\ $r\in\dom(f_j)$ (namely, those $r$ for which the conclusion of Lebesgue's differentiation theorem holds and $f_j'(r)$ exists). If we had $f_j'(r)>0$ and $\mathcal H^0(\partial B_r(x)\cap S_j)=0$ for some $r\in \dom(f)$, then $\textup{dist}(\partial B_r(x),S_j)>0$ (using the closedness of $S_j$). This would yield a contradiction because then $f_j'(r)=0$.\\
	To show the last part, we define $g_j(r)=\int_{\overline{B}_r(x)\cap S_j}\kappa(|\theta|)\e\Ha$. Now we can use exactly the same arguments as above together with the estimate (recall that $|\theta|\in\mathbb N$ $\Ha\mres S$-a.e.)
	\begin{equation*}
		\frac{g_j(r+h)-g_j(r)}{h}=\frac{1}{h}\int_{(\overline{B}_{r+h}(x)\setminus\overline{B}_r(x))\cap S_j}\kappa(|\theta|)\e\Ha \geq\frac{\kappa(1)}{h}\int_{(r,r+h]}\mathcal H^0(\partial B_s(x)\cap S_j)\e\Ha(s).\qedhere
	\end{equation*}
\end{proof}
\begin{examp}[Case $k>1$]
	\label{Casek2}
	In general, the statement is not true for $k>1$, countably $k$-rectifiable $S\subset\R^n$, and $\Ha$ replaced by $\mathcal H^k$, i.e.\ $f(r)=\mathcal H^k(\overline{B}_r(x)\cap S)$ for $r\in\R_+$.
	For example, take $k=2,\ x=0$, and $S=\R\times[-\ell,\ell]\times\{ 0 \}$ for some $\ell\in(0,1/5)$, see \cref{fig5}.
	Then $4\ell<f'(r)<5\ell \in(0,1)$ for $r\in(\ell,\infty)$ sufficiently large because the mean curvature of $\partial B_r(0)$ is decreasing.
	The assumption $k=1$ was used in the step before Lebesgue's differentiation theorem was applied (for $\mathcal H^k$ in place of $\Ha=\Leb$ on $(r,r+h]$, the right-hand side of the inequality would be equal to zero).
\end{examp}
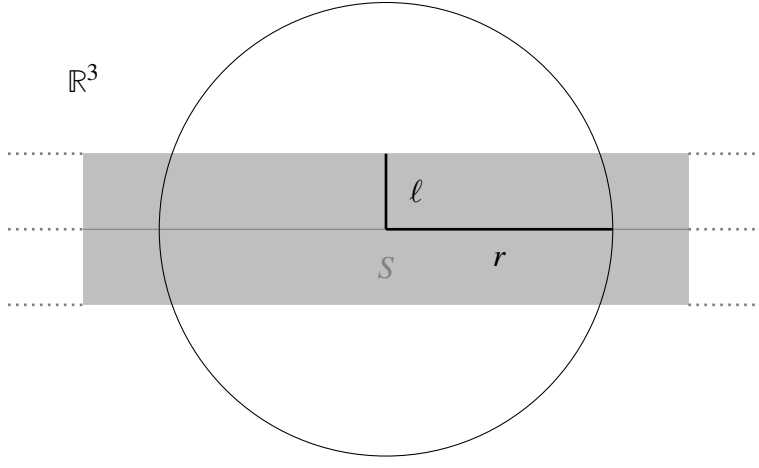
\begin{figure}
	\centering
	\begin{tikzpicture}[scale=2, transform shape=false]
		
		\filldraw[draw=lightgray, fill=lightgray] (-2,-0.5) -- (-2,0.5) -- (2,0.5) -- (2,-0.5);
		\draw[gray] (-2,0) -- (2,0);
		\draw[line width=1] (0,0) -- (1.5,0);
		\draw[line width=1] (0,0) -- (0,0.5);
		\draw (0,0) circle (1.5);
		
		\draw[gray,dotted,line width=1] (-2.5,0) -- (-2,0);
		\draw[gray,dotted,line width=1] (2,0) -- (2.5,0);
		\draw[gray,dotted,line width=1] (-2.5,0.5) -- (-2,0.5);
		\draw[gray,dotted,line width=1] (2,0.5) -- (2.5,0.5);
		\draw[gray,dotted,line width=1] (-2.5,-0.5) -- (-2,-0.5);
		\draw[gray,dotted,line width=1] (2,-0.5) -- (2.5,-0.5);
		
		\node at (0,-0.25) {$\textcolor{gray}{S}$};
		\node at (0.75,-0.2) {$r$};
		\node at (0.2,0.25) {$\ell$};
		\node at (-2,1) {$\R^3$};
	\end{tikzpicture}
	\caption{Bird's-eye view on $S$ from \cref{Casek2}.}
	\label{fig5}
\end{figure}
We recall the following straightforward implication of the Besicovitch covering theorem.
\begin{coro}
	\label{Besicovitch}
	Let $A\in\mathcal B(\R^n)$ and $\mu\in\M_+(\R^n)$. Assume that $\mathcal F\subset \mathcal B(\R^n)$ is a Besicovitch cover of $A$. Then there exists a sequence $(B_i)\subset\mathcal F$ of mutually disjoint balls $B_i$ such that 
	\begin{equation*}
		\mu(A)\leq\mathbb B_n\sum_{i=1}^{\infty}\mu(B_i),
	\end{equation*}
	where $\mathbb B_n$ denotes the Besicovitch constant.
\end{coro}
In the remainder of this section, we fix an arbitrary $\lambda\in (0,\tau(1)]$ (the strongest statements are always for $\lambda=\tau(1)$).
The following statement is a variant of \cite[Lem.\ 5.2]{KT20}.
\begin{prop}
	\label{Lemma1}
	Let the setup be as in \cref{Prop6.1,Prop6.4}. Fix $t\in\R_+$ and take $k$ sufficiently large such that $t\in[0,j_k]$. Define $r_k=j_k^{-5/2}$ and (write $T_t^{j_k}=\theta_t^{j_k}\Ha\mres S_t^{j_k}$)
	\begin{equation*}
		X_t^{k}=\left\{ x\in S_t^{k}\:\middle|\: |T_t^{j_k}|_\tau(\overline{B}_r(x))\geq \lambda r\textup{ for all }r\in(0,r_k) \right\}.
	\end{equation*} 
	Then $X_t^{k}$ is closed. Now let \cref{IntegralityT0,GrowthCondition} be satisfied. Increase $k$ such that $j_k^{-3/2}\leq C/2$. Then we have
	\begin{equation*}
		|T_t^{j_k}|_\tau(S_t^k\setminus X_t^k)\leq-\frac{\hat c}{r_k^n}\Delta_{j_k}|T_t^{j_k}|_\tau(\R^n),
	\end{equation*}
	where $\hat c>0$ only depends on $\Omega$ (implicitly on the dimension $n$). 
\end{prop}
\begin{proof}
	Assume that $X_t^{k}$ is not closed. 
	Then there exists $x\in\R^n\setminus X_t^k$ with $|T_t^{j_k}|_\tau(\overline{B}_r(x))< \lambda r$ for some $r\in(0,r_k)$ and a sequence $(x_i)\subset X_t^{k}$ with $x_i\to x$.
	Define $r_i =\textup{dist}(x_i,\partial B_r(x))<r$ for $i$ sufficiently large. 
	Since $x_i\in X_t^k$, we obtain $|T_t^{j_k}|_\tau(\overline{B}_{r_i}(x_i))\geq \lambda r_i$.
	Note that $r_i\to r$ by $x_i\to x$.
	Hence, we get (using also $\overline{B}_{r_i}(x_i)\subset \overline{B}_{r}(x)$)
	\begin{equation*}
		\lambda r>|T_t^{j_k}|_\tau(\overline{B}_r(x))\geq\liminf_i|T_t^{j_k}|_\tau(\overline{B}_{r_i}(x_i))\geq \liminf_i\lambda r_i=\lambda r,
	\end{equation*}
	which yields a contradiction. Thus, set $X_t^{k}$ is closed.\\
	Next, we prove the second claim. To this end, we cover $\{ \textup{dist}(\cdot,\Omega)\leq 2 \}$ with closed balls whose radii are equal to $r_k$ and whose number does not exceed $c(\Omega)r_k^{-n}$.
	Let $\overline{B}_{r_k}(\hat x)$ be any of these balls.
	Define $R_t^k=\overline{B}_{r_k}(\hat x)\cap(S_t^k\setminus X_t^k)$.
	
	\underline{STEP 1 (covering of $R_t^k$):} For each $x\in R_t^k$ there exists some $r_x\in(0,r_k)$ such that $|T_t^{j_k}|_\tau(\overline{B}_{r_x}(x))< \lambda r_x$.
	Let $f_x(r)=|T_t^{j_k}|_\tau(\overline{B}_r(x))$ for $x\in R_t^k$ and $r\in\R_+$.
	By \cref{GrowthSBall} we have $\Leb([0,r_x]\cap \{ f_x'=0 \})>(1-\lambda/\tau(1))r_x$.
	Further, for each $x\in R_t^k$ there exists $\tilde r_{x}\in((1-\lambda/\tau(1))r_{x},r_{x})$ such that $f_{x}'(\tilde r_{x})=0$ (otherwise $\Leb([0,r_x]\cap \{ f_x'=0 \})\leq (1-\lambda/\tau(1))r_x$).
	Note that $\{ \overline{B}_{\tilde r_x}(x)\:|\: x\in R_t^k \}$ is a Besicovitch cover of $R_t^k$ because $\sup_{x\in R_t^k}\tilde r_x\leq r_k<\infty$.
	By \cref{Besicovitch} there exists a sequence $(x_\ell)\subset R_t^k$ such that the balls $\overline{B}_{\tilde r_{x_\ell}}(x_\ell)$ are mutually disjoint and 
	\begin{equation*}
		|T_t^{j_k}|_\tau(R_t^k)\leq\mathbb B_n\sum_{\ell=1}^{\infty}|T_t^{j_k}|_\tau(\overline{B}_{\tilde r_{x_\ell}}(x_\ell))\leq  2\mathbb B_n\sum_{\ell=1}^{\ell_0}|T_t^{j_k}|_\tau(\overline{B}_{\tilde r_{x_\ell}}(x_\ell)),
	\end{equation*}
	where $\ell_0\in\mathbb N$ is chosen sufficiently large. We may assume that $|T_t^{j_k}|_\tau(\overline{B}_{\tilde r_{x_\ell}}(x_\ell))>0$ for all $\ell=1,\ldots,\ell_0$.
	
	\underline{STEP 2 (admissible deformation $f_1^h\in\adm_{j_k}(T_t^{j_k})$):}
\begin{figure}
	\centering
	\begin{tikzpicture}[scale=0.86, transform shape=false]
		\draw[line width=1,->] (0,0) -- (6.9,0);
		\draw[line width=1,->] (0,0) -- (0,6.9);
		
		\draw[color=lightgray,line width=2] (3.5,0) -- (5.5,0);
		
		\draw[line width=2] (0,0) -- (3.5,0) -- (5.5,5.5) -- (6.9,6.9);
		
		\draw[color=lightgray,line width=2] (0,0) -- (0,5.5);
		
		\node[lightgray,fill,circle,minimum size=2mm,inner sep=0pt] at (0,0) {};
		
		\node at (0,-0.4) {$x_\ell$};
		\node at (3.5,-0.4) {$\partial B_{\tilde r_{x_\ell}}(x_\ell)$};
		\node at (5.6,-0.4) {$\partial B_{\tilde r_{x_\ell}+h}(x_\ell)$};
		
		\node at (4.5,0.4) {$\textcolor{lightgray}{A}$};
		\node at (-0.6,2.75) {$\textcolor{lightgray}{f(A)}$};
		
		\node at (7.8,0) {$\partial B_r(x_\ell)$};
		\node at (0,7.3) {$f(\partial B_r(x_\ell))$};
		
		\draw[lightgray,line width=1]  (3.5,-0.15) -- (3.5,0.15);
		\draw[line width=1]  (5.5,-0.15) -- (5.5,0.15);
		\draw[lightgray,line width=1]  (-0.15,3.5) -- (0.15,3.5);
		\draw [line width=1] (-0.15,5.5) -- (0.15,5.5);
		
		\node at (2.5,4) {$\textcolor{gray}{\textup{Lip}(f|_{B_{\tilde r_{x_\ell}+h}(x_\ell)})=\frac{\tilde r_{x_\ell}+h}{h}}$};
	\end{tikzpicture}
	\caption{Admissible deformation $f_1^h$ near $x_\ell$.}
	\label{fig3}
\end{figure}
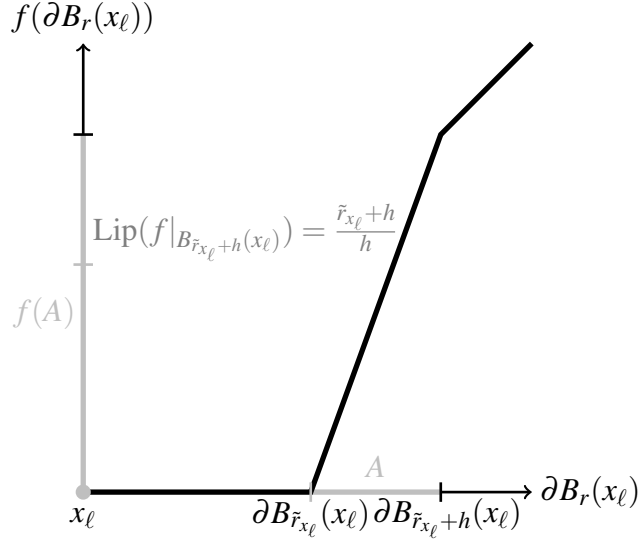
	Take $h\in\R_+^*$ such that the balls $\overline{B}_{\tilde r_{x_\ell}+h}(x_\ell)$ are mutually disjoint ($\ell=1,\ldots,\ell_0$).
	Define $f_1^h:\R^n\to\R^n$ by (see \cref{fig3})
	\begin{equation*}
		f_1^h(x)=\begin{cases*}
			x&if $x\in\R^n\setminus \bigcup_{\ell=1}^{\ell_0}B_{\tilde r_{x_\ell}+h}(x_\ell )$,\\
			x_\ell&if $x\in \overline{B}_{\tilde r_{x_\ell}}(x_\ell)$ for some $\ell=1,\ldots,\ell_0$,\\
			x_\ell+ \frac{\tilde{r}_{x_\ell}+h}{h}(|x-x_\ell|-\tilde{r}_{x_\ell}) \frac{x-x_\ell}{|x-x_\ell|}&if $x\in B_{\tilde r_{x_\ell}+h}(x_\ell)\setminus \overline{B}_{\tilde r_{x_\ell}}(x_\ell)$ for some $\ell=1,\ldots,\ell_0$.
		\end{cases*}
	\end{equation*}
	Note that $f_1^h$ is Lipschitz with $\textup{Lip}(f_1^h)=(\max_{\ell=1,\ldots,\ell_0}\tilde r_{x_\ell}+h)/h$. 
	We get
	\begin{equation*}
		|f_1^h-\textup{id}|_\infty\leq \tilde{r}_{x_\ell}<r_{x_\ell}<r_k\leq j_k^{-2} ,
	\end{equation*}
	which shows that the property in the first bullet point in \cref{AdmDeform} is satisfied.
	Next, we use the proposition on \cite[pp.\ 371--372]{F} with $U=\R^n,T=T_t^{j_k},f=\textup{id}$, and $g=f_1^h$. Let $\sigma\in C(\R^n)$ with 
	\begin{equation*}
		\textup{max}(|\Jm f(x)|_2,|\Jm g(x)|_2)\leq\sigma(x)
	\end{equation*}
	for $\Leb^n$-a.e.\ $x\in\R^n$. One can take any continuous $\sigma$ with $\sigma\geq 1$ and $\sigma =(\tilde r_{x_\ell}+h)/h$ on $B_{\tilde r_{x_\ell}+h}(x_\ell)\setminus \overline{B}_{\tilde r_{x_\ell}}(x_\ell)$ for all $\ell=1,\ldots,\ell_0$. Then the proposition implies (recall \cref{AdmDeform})
	\begin{equation}
		\label[ineq]{in1:SigmaEst}
		\mathbb \Mb\left((h_{f_1^h})_\#(\llbracket 0,1\rrbracket\times T_t^{j_k})\right)\leq |T_t^{j_k}|(|f_1^h-\textup{id}|\sigma).
	\end{equation}
	To simplify the notation, we abbreviate $\overline{B}_\ell=\overline{B}_{\tilde r_{x_\ell}}(x_\ell),A_\ell=B_{\tilde r_{x_\ell}+h}(x_\ell)\setminus \overline{B}_{\tilde r_{x_\ell}}(x_\ell), T=\theta\Ha\mres S=T_t^{j_k}$, and $(f_1^h)_\# T=\eta\Ha\mres f_1^h(S)$, cf.\ \cref{FedererPushforward}.
	We get (\cref{FedererPushforward} and area formula \cite[Cor.\ 3.2.20]{F})
	\begin{multline*}
		|(f_1^h)_\# T|_\tau(\overline{B}_\ell\cup A_\ell)=\int_{f_1^h(A_\ell\cap S)}\tau(|\eta|)\e\Ha=\int\sum_{x\in (f_1^h)^{-1}(y)}1_{A_\ell\cap S}(x)\tau(|\theta(x)|)\e\Ha(y)\\
		=\int_{A_\ell\cap S}\tau(|\theta|)\apJ(f|_{A_\ell\cap S})\e\Ha\leq\frac{\tilde{r}_{x_\ell}+h}{h}|T|_\tau(A_\ell)\to (\tilde{r}_{x_\ell}+h)\cdot 0=0
	\end{multline*}
	for $h\to 0$ for all $\ell=1,\ldots,\ell_0$.
	Thus, we can choose $h$ sufficiently small such that 
	\begin{equation}
		\label[ineq]{in1:auxh}
		-\frac{1}{2}|T|_\tau(\overline{B}_\ell)+\frac{1}{2h}|T|_\tau(A_\ell)\leq -|(f_1^h)_\# T|_\tau(\overline{B}_\ell\cup A_\ell)
	\end{equation}
	for all $\ell=1,\ldots,\ell_0$.
	Further, in \cref{in1:SigmaEst} we can just take $\sigma$ with $\sigma = 1$ on $\overline{B}_\ell$, $\sigma= (\tilde r_{x_\ell}+h)/h$ on $A_\ell$, and $\sigma=\max_{\ell=1,\ldots,\ell_0}(\tilde r_{x_\ell}+h)/h$ else. 
	We get (using $j_k^{-3/2}\leq C/2$, \cref{GrowthCondition}, and \cref{in1:auxh})
	\begin{multline*}
		\mathbb \Mb\left((h_{f_1^h})_\#(\llbracket 0,1\rrbracket\times T)\right)\leq |T|(|f_1^h-\textup{id}|\sigma)=\sum_{\ell=1}^{\ell_0}\left(|T|(|f_1^h-\textup{id}||_{\overline{B}_\ell})+|T|(|f_1^h-\textup{id}||_{A_\ell}\sigma)\right)\\
		\leq \sum_{\ell=1}^{\ell_0}\left(\tilde r_{x_\ell}|T|(\overline{B}_\ell)+\tilde{r}_{x_\ell}\frac{\tilde{r}_{x_\ell}+h}{h}|T|(A_\ell)\right)\leq r_k\sum_{\ell=1}^{\ell_0}\left(|T|(\overline{B}_\ell)+\frac{1}{h}|T|(A_\ell)\right)\leq j_k^{-1}\frac{C}{2}\sum_{\ell=1}^{\ell_0}\left(|T|(\overline{B}_\ell)+\frac{1}{h}|T|(A_\ell)\right)\\
		\leq j_k^{-1}\sum_{\ell=1}^{\ell_0}\left( |T|_\tau(\overline{B}_\ell) -|(f_1^h)_\# T|_\tau(\overline{B}_\ell\cup A_\ell)\right)
		\leq j_k^{-1}(|T|_\tau(\R^n)-|(f_1^h)_\# T|_\tau(\R^n)).
	\end{multline*}
	This shows that the first condition in the second bullet point in \cref{AdmDeform} is satisfied. The second condition in the second bullet point in \cref{AdmDeform} is also satisfied:
	By the proof of the lemma in \cite[Sec.\ 4.10]{B78} (cf.\ \cite[Lem.\ 4.12]{KT17}) it is enough to verify that $|(f_1^h)_\# T|_\tau(\overline{B}_\ell\cup A_\ell)\leq e^{-j_k(\tilde r_{x_\ell} +h)}|T|_\tau(\overline{B}_\ell\cup A_\ell)$ for all $\ell=1,\ldots,\ell_0$ and $h$ sufficiently small because then (using \cref{propertiesF}) 
	\begin{multline*}
		|(f_1^h)_\# T|(\phi)-|T|_\tau(\phi)=\sum_{\ell=1}^{\ell_0}\left( |(f_1^h)_\# T|(\phi1_{\overline{B}_\ell\cup A_\ell})-|T|_\tau(\phi 1_{\overline{B}_\ell\cup A_\ell}) \right)\\
		\leq\sum_{\ell=1}^{\ell_0}\left( \max_{\overline{B}_\ell\cup\overline{A}_\ell}\phi|(f_1^h)_\# T|(\overline{B}_\ell\cup A_\ell)-\min_{\overline{B}_\ell\cup\overline{A}_\ell}\phi|T|_\tau(\overline{B}_\ell\cup A_\ell) \right)\\
		\leq \sum_{\ell=1}^{\ell_0}\left( \min_{\overline{B}_\ell\cup\overline{A}_\ell}\phi e^{j_k(\tilde r_{x_\ell}+h)}|(f_1^h)_\# T|(\overline{B}_\ell\cup A_\ell)-\min_{\overline{B}_\ell\cup\overline{A}_\ell}\phi |T|_\tau(\overline{B}_\ell\cup A_\ell) \right)\leq 0
	\end{multline*}
	for all $\phi\in\mathcal{F}_{j_k}^+$. But we already argued that $|(f_1^h)_\# T|_\tau(\overline{B}_\ell\cup A_\ell)\to 0$ for $h\to 0$ and $|T|_\tau(\overline{B}_\ell)>0$ for all $\ell=1,\ldots,\ell_0$. Thus, the desired inequality is satisfied for all $\ell=1,\ldots,\ell_0$. This finally shows $f_1^h\in\mathcal A_{j_k}(T_j^{j_k})$.
	
	\underline{STEP 3 (final estimates):} Let $\delta\in(0,1)$. We can decrease $h$ such that (the left-hand side goes to zero, cf.\ STEPS 1 and 2, and the right-hand side is assumed to be positive)
	\begin{equation*}
		|(f_1^h)_\# T|_\tau(\overline{B}_\ell\cup A_\ell)-|T|_\tau(A_\ell)\leq\delta|T|_\tau(\overline{B}_\ell)
	\end{equation*}
	for all $\ell=1,\ldots,\ell_0$. Hence, we obtain (recall STEP 1)
	\begin{multline*}
		\Delta_{j_k}|T|_\tau(\R^n)\leq|(f_1^h)_\# T|_\tau(\R^n)-|T|_\tau(\R^n)=\sum_{\ell=1}^{\ell_0}\left(|(f_1^h)_\# T|_\tau(\overline{B}_\ell\cup A_\ell)-|T|_\tau(\overline{B}_\ell\cup A_\ell)\right)\\
		\leq (\delta-1)\sum_{\ell=1}^{\ell_0}|T|_\tau(\overline{B}_\ell)
		\leq \frac{\delta-1}{2\mathbb B_n}|T|_\tau(R_t^k).
	\end{multline*}
	By the beginning of the proof this gives
	\begin{equation*}
		|T|_\tau(S\setminus X_t^k)\leq\frac{2\mathbb B_nc(\Omega)}{(\delta-1)r_k^n}\Delta_{j_k}|T|_\tau(\R^n).
	\end{equation*}
	Note that this inequality does not depend on $h$ anymore. Hence, we can let $\delta\to 0$ and take $\hat c=2\mathbb B_nc(\Omega)$.
\end{proof}
Next, we apply Allard's monotonicity formula \cite[Sec.\ 5.1, Thm.\ 1]{A72}.
For the reader's convenience, we briefly recall the statement. Let $U\subset\R^n$ open and $V\in\mathcal V_k(U)$. Define $\|\delta V\|$ as the largest Borel regular (outer) measure on $U$ with
\begin{equation*}
	\|\delta V\|(\widetilde{U})=\sup\{ \delta V(g)\:|\: g\in C_c^\infty(U;\R^n),\spt(g)\subset\widetilde U,|g|_\infty\leq 1 \}
\end{equation*}
for all open $\widetilde U\subset U$, see \cite[Sec\ 4.2]{A72}. If $\|\delta V \|$ is a Radon measure, i.e.\ $\|\delta V \|(K)<\infty$ for all compact $K\subset U$, then recall that it can be interpreted as a linear functional $\|\delta V \|:C_c(U)\to\R$ with $\|\delta V \|(f)\geq 0$ for all $f\in C_c(U;\R_+)$ (Riesz representation theorem).
\begin{thm}[{Allard's monotonicity formula \cite[Sec.\ 5.1, Thm.\ 1]{A72}}]
	\label{Allard}
	Let $U\subset\R^n$ open and $V\in\mathcal V_k(U)$. Assume that $\|\delta V \|$ is a Radon measure. Then 
	\begin{multline*}
		s^{-k}\| V\|(\overline{B}_s(x))e^{ \int_{[r,s]} \delta V(f(V,x;u))\e\Leb(u)}-r^{-k}\| V\|(\overline{B}_r(x))\\
		=\int_{\{ (y,S)\:|\: r<|x-y|\leq s \}}|x-y|^{-k-2}|P^\perp(x-y)|^2e^{\int_{[r,|x-y|]}\delta V(f(V,x;u))\e\Leb(u)}\e V(y,P)
	\end{multline*}
	for all $\textup{dist}(x,\spt(\| V\|))<r<s<\textup{dist}(x,\R^n\setminus U)$, where
	\begin{equation*}
		f(V,x;u)(y)=\begin{cases*}
			(u\| V\|(\overline{B}_u(x)))^{-1}(y-x)&if $y\in\overline{B}_u(x)$,\\
			0&else.
		\end{cases*}
	\end{equation*}
\end{thm}
We also need the following simple fact.
\begin{lem}[$\delta (\varphi \ast V)$ bounded]
	\label{auxAllard}
	Let $U\subset\R^n$ open, $V\in\mathcal V_k(U)$, and $\varphi\in C_c^1(U)$. Assume that $\|V\|(U)<\infty$. Then 
	\begin{equation*}
		|\delta (\varphi \ast V)(g)|\leq c(n,\varphi,V)|g|_\infty
	\end{equation*}
	for all $g\in C_c^1(U;\R^n)$, in particular $\delta (\varphi \ast V)\in \mathcal M(U;\R^n)$ respectively $\|\delta (\varphi \ast V)\| \in\M_+(U)$.
\end{lem}
\begin{proof}
	We have
	\begin{multline*}
		|\delta (\varphi \ast V)(g)|=\left|\int_{\mathcal G_k(U)}\int_{U}P:\Jm g(x-y)\varphi(y)\e\Leb^n(y)\e V(x,P)\right|\\
		=\left|\int_{\mathcal G_k(U)}\int_{U}P:(g(x-y)\otimes\nabla\varphi(y))\e\Leb^n(y)\e V(x,P)\right|\leq \sqrt k|g|_\infty|\nabla\varphi|_\infty\Leb^n(\spt(\varphi))V(\mathcal G_k(U)).
	\end{multline*}
	Thus, we can take $c(n,\varphi,V)=\sqrt k|\nabla\varphi|_\infty\Leb^n(\spt(\varphi))V(\mathcal G_k(U))$. The last conclusion follows directly,
	\begin{equation*}
		\|\delta (\varphi \ast V)\|(U)=	\sup\{ \delta (\varphi \ast V)(g)\:|\: g\in C_c^\infty(U;\R^n),|g|_\infty\leq 1 \}\leq c(n,\varphi,V).\qedhere
	\end{equation*}
\end{proof}
We will also use the following auxiliary statement.
\begin{lem}[Properties of $\Psi_\varepsilon\ast 1_{\overline{B}_r(x)}$]
	\label{aux2}
	Let $\mu\in\M_+(\R^n), \gamma,\delta\in(0,1)$, and $\overline{B}_r(x)\subset\R^n$ some closed ball. Then
	\begin{equation*}
		(\Psi_\varepsilon\ast\mu)(\overline{B}_r(x))=\mu(\Psi_\varepsilon\ast 1_{\overline{B}_r(x)}).
	\end{equation*}
	Moreover, we have
	\begin{equation*}
		\Psi_\varepsilon\ast 1_{\overline{B}_r(x)}\geq (1-\gamma)1_{\overline{B}_{(1-\delta)r}(x)}
	\end{equation*}
	in $\R^n$ for $\varepsilon$ sufficiently small.
\end{lem}
\begin{proof}
	By the outer regularity of $\Psi_\varepsilon\ast\mu$ there exist bounded and open $U_i\subset\R^n$ with $\overline{B}_r(x)\subset U_i$ and a sequence $(\delta_i)\subset(0,1)$ with $\delta_i\to 0$ such that $(\Psi_\varepsilon\ast\mu)(U_i)\leq(\Psi_\varepsilon\ast\mu)(\overline{B}_r(x))+\delta_i$. Now pick a sequence $(\phi_i)\subset C_c(\R^n)$ with $\spt(\phi_i)\subset U_i,\phi_i\equiv 1$ on $\overline{B}_r(x),\phi_i\in[0,1]$, and $\phi_i\to 1_{\overline{B}_r(x)}$ pointwise. Then $(\Psi_\varepsilon\ast\mu)(\overline{B}_r(x))\leq(\Psi_\varepsilon\ast\mu)(\phi_i)\leq(\Psi_\varepsilon\ast\mu)(U_i)\leq(\Psi_\varepsilon\ast\mu)(\overline{B}_r(x))+\delta_i$ implying (Lebesgue's dominated convergence theorem)
	\begin{equation*}
		(\Psi_\varepsilon\ast\mu)(\overline{B}_r(x))=\lim_i(\Psi_\varepsilon\ast\mu)(\phi_i)=\lim_i\mu(\Psi_\varepsilon\ast\phi_i)=\lim_i\int\int \Psi_\varepsilon(z-y)\phi_i(y)\e\Leb^n(y)\e\mu(z).
	\end{equation*}
	This gives (again by Lebesgue's dominated convergence theorem)
	\begin{equation*}
		|(\Psi_\varepsilon\ast\mu)(\overline{B}_r(x))-\mu(\Psi_\varepsilon\ast 1_{\overline{B}_r(x)})|\leq\max\Psi_\varepsilon\lim_i\int\int|\phi_i(y)-1_{\overline{B}_r(x)}(y)|\e\Leb^n(y)\e\mu=0,
	\end{equation*}
	which proves the first claim. By definition of $\Psi_\varepsilon$ the function $\Psi_\varepsilon\ast 1_{\overline{B}_r(x)}$ is bounded from below by $1-\gamma$ on $\overline{B}_{(1-\delta)r}(x)$ for $\varepsilon$ sufficiently small, which directly yields the second statement.
\end{proof}
The following proposition corresponds to \cite[Lem.\ 5.4]{KT20}.
\begin{prop}
	\label{LowerBound}
	Let the setup be as in \cref{Prop6.1,Prop6.4}. Recall the definition of $X_t^{j_k}$ from \cref{Lemma1}. For a.e.\ $t\in\R_+$ there exists $c_2=c_2(t)\in(0,1)$ and a subsequence $\ell \mapsto j_{k_\ell}$ (depending on $t$) such that 
	\begin{equation*}
		|T_t^{j_{k_\ell}}|_\tau(\overline{B}_r(x))\geq\frac{\lambda r}{11}
	\end{equation*}
	for all $x\in X_t^{j_{k_\ell}},r\in (0,c_2]$.
\end{prop}
\begin{proof}
	We follow the proof of \cite[Lem.\ 5.4]{KT20}. By definition of $X_t^{j_k}$ it is enough to find (for a.e.\ $t\in\R_+$) some $c_2\in(0,1)$ and validate $|T_t^{j_{k_\ell}}|_\tau(\overline{B}_r(x))\geq\frac{\lambda r}{11}$ for $x\in X_t^{j_{k_\ell}},r\in [r_{k_\ell},c_2]$ for a suitable subsequence $\ell\mapsto j_{k_\ell}$. 
	\Cref{in1:6.19} implies
	\begin{equation*}
		\limsup_k\int_{[0,R]}   I_k(t) \e\Leb(t)<\infty
	\end{equation*}
	for all $R\in\R_+$, where we use the abbreviation $I_k(t)=\int|\Psi_{\varepsilon_{j_k}}\ast\delta_\tau T_t^{j_k}|^2/(\Psi_{\varepsilon_{j_k}}\ast|T_t^{j_k}|_\tau+\varepsilon_{j_k})\e\Lebn$.
	In particular, Fatou's lemma yields
	\begin{equation*}
		\int_{[0,R]}\liminf_kI_k(t)\e\Leb(t)<\infty
	\end{equation*}
	for all $R\in\R_+$. Thus, for a.e.\ $t\in\R_+$ there exists a subsequence $\ell\mapsto j_{k_\ell}$ with
	\begin{equation*}
		\liminf_kI_k(t)=\lim_{\ell} I_{k_{\ell}}(t)\in\R_+.
	\end{equation*}
	Now fix any such $t$ and a corresponding subsequence $\ell\mapsto j_{k_\ell}$.
	Define
	\begin{equation*}
		\check c=\check c(t)=\sup_{\ell} I_{k_{\ell}}(t).
	\end{equation*}
	In the remainder of the proof, we abbreviate $T=\theta\Ha\mres S=T_t^{j_{k_\ell}},X=X_t^{j_{k_\ell}}$, and $\varepsilon=\varepsilon_{j_{k_\ell}}$. Let $V=v(\tau(|\theta|),S)\in\mathcal V_1(\R^n)$ (recall \cref{inducedV}). By \cref{auxAllard} we can apply \cref{Allard} to $V_\varepsilon=\Psi_{\varepsilon}\ast V$. In particular, we obtain
	\begin{equation}
		\label[ineq]{in1:AllardIn}
		s^{-1}\| V_\varepsilon \|(\overline{B}_s(x))e^{ \int_{[r,s]} \delta V_{\varepsilon}(f(V_{\varepsilon},x;u))\e\Leb(u)}\geq r^{-1}\| V_{\varepsilon }\|(\overline{B}_r(x))
	\end{equation}
	for all $\textup{dist}(x,\spt(\| V_{\varepsilon}\|))<r<s<\infty$.
	Note that $X\subset \spt(\| V_{\varepsilon}\|)$. 
	Take $x\in X$, $a\in(0,1)$, and $u\in [ar_{k_\ell},1]$.
	Depending on the factor $a\in(0,1)$, we can choose $\ell$ sufficiently large such that the forthcoming estimates are valid. 
	Later, we take $a=1/2$.
	By the same conversion as in \cite[(5.13)]{KT20} we have (recall \cref{Smoothed})
	\begin{equation}
		\label[ineq]{in1:Y}
		|\delta V_{\varepsilon}(f(V_{\varepsilon},x;u))|\leq \| V_{\varepsilon}\|(\overline{B}_u(x))^{-1}\int_{\overline{B}_u(x)}|\Psi_{\varepsilon}\ast\delta V|\e\Leb^n.
	\end{equation}
	Further, using the Cauchy\textendash Schwarz inequality, we get
	\begin{equation*}
		\int_{\overline{B}_u(x)}|\Psi_{\varepsilon}\ast\delta V|\e\Leb^n\leq I_{k_\ell}(t)^{1/2}\left( \int_{\overline{B}_u(x)}(\Psi_{\varepsilon}\ast\| V\|+\varepsilon)\e\Leb^n \right)^{1/2},
	\end{equation*}
	where we used (cf.\ text passage below \cref{hEstimates})
	\begin{equation*}
		I_{k_\ell}(t)=\int\frac{|\Psi_{\varepsilon}\ast\delta_\tau T|^2}{\Psi_{\varepsilon}\ast|T|_\tau+\varepsilon } \e\Leb^n=\int\frac{ |\Psi_{\varepsilon}\ast\delta V|^2}{\Psi_{\varepsilon}\ast\| V\|+\varepsilon}\e\Leb^n.
	\end{equation*}
	This gives
	\begin{equation}
		\label[ineq]{in1:Z}
		\int_{\overline{B}_u(x)}|\Psi_{\varepsilon}\ast\delta V|\e\Leb^n\leq \check c^{1/2}\left( \int_{\overline{B}_u(x)}\Psi_{\varepsilon}\ast |T|_\tau\e\Leb^n+\varepsilon\Leb^n(\overline{B}_u(x)) \right)^{1/2}.
	\end{equation}
	Now it is not difficult to see that $\int_{\overline{B}_u(x)}\Psi_{\varepsilon}\ast |T|_\tau\e\Leb^n\leq (\Psi_{\varepsilon}\ast|T|_\tau)(\overline{B}_u(x))$\footnote{Formally, abbreviate $\int_{\overline{B}_u(x)}\Psi_{\varepsilon}\ast |T|_\tau\e\Leb^n\leq (\Psi_{\varepsilon}\ast|T|_\tau)(\overline{B}_u(x))$ as $I\leq \mu(\overline{B})$. By the outer regularity of $\mu$ we have $\mu(\overline{B})=\inf\{ \mu(U)\:|\: U\supset\overline{B}\textup{ open} \}$. In particular, for every $\delta>0$ there is some open and bounded set $U\supset \overline{B}$ with $\mu(U)\leq\mu(\overline{B})+\delta$. Now take $\phi\in C_c(\R^n)$ with $\spt(\phi)\subset U,\phi\equiv 1$ on $\overline{B}$, and $\phi\in[0,1]$. Then $I\leq \mu(\phi)\leq\mu(U)\leq \mu(\overline{B})+\delta$.}.
	Define $\xi(u)=(\Psi_{\varepsilon}\ast|T|_\tau)(\overline{B}_u(x))$.
	Applying \cref{aux2} we have
	\begin{equation*}
		\xi(u)=|T|_\tau(\Psi_{\varepsilon}\ast1_{\overline{B}_u(x)})\geq\frac{1}{2}|T|_\tau(1_{ \overline{B}_{u/2}(x) })
	\end{equation*}
	(In the remainder, we always assume that $\ell$ is taken such that $\varepsilon_{j_{k_{\ell}}}$ is sufficiently small.)
	By $ar_{k_\ell}\leq u$ and $x\in X$ we have
	\begin{equation*}
		\frac{1}{2}|T|_\tau(1_{ \overline{B}_{u/2}(x) })\geq \frac{1}{2}|T|_\tau(1_{ \overline{B}_{ar_{k_\ell}/2}(x) })\geq \lambda\frac{ar_{k_\ell}}{4}.
	\end{equation*}
	Further, recalling $\varepsilon<j_{k_\ell}^{-6}<aj_{k_\ell}^{-5/2}=ar_{k_\ell}\leq u\leq 1$ (see \cref{Prop6.1}), we get
	\begin{equation*}
		\varepsilon\Lebn(\overline{B}_u(x))\leq \varepsilon\Lebn(\overline{B}_1(x))\leq\lambda\frac{ar_{k_\ell}}{4}\leq \xi(u).
	\end{equation*}
	Using the above together with \cref{in1:Y,in1:Z} gives
	\begin{multline*}
		|\delta V_{\varepsilon}(f(V_{\varepsilon},x;u))|\leq \| V_{\varepsilon}\|(\overline{B}_u(x))^{-1}\check c^{1/2}\left( \int_{\overline{B}_u(x)}\Psi_{\varepsilon}\ast |T|_\tau\e\Leb^n+\varepsilon\Leb^n(\overline{B}_u(x)) \right)^{1/2}\\
		\leq \xi(u)^{-1}\check c^{1/2}\left( \xi(u)+\xi(u) \right)^{1/2}=(2\check c)^{1/2}\xi(u)^{-1/2}.
	\end{multline*}
	Using this in \cref{in1:AllardIn} yields
	\begin{equation}
		\label[ineq]{in1:inAllard2}
		s^{-1}\xi(s)e^{ \tilde c\int_{[r,s]} \xi(u)^{-1/2}\e\Leb(u)}\geq r^{-1}\xi(r)
	\end{equation}
	for all $ar_{k_\ell}\leq r<s\leq 1$, where $\tilde c=(2\check c)^{1/2}$. Note that $\xi$ is non-decreasing on $[ar_{k_\ell},1]$, thus it is differentiable a.e.\ in $[ar_{k_\ell},1]$. It is also continuous: for $\phi\in C_c(\R^n)$ one may identify $(\Psi_\varepsilon\ast |T|_\tau) (\phi)$ as the $L^2$-pairing of $\phi$ with the $C_c^\infty$-function $y\mapsto \int\Psi_\varepsilon(z-y)\e|T|_\tau(z)$. Hence, the fundamental theorem of calculus implies that $\xi$ admits an antiderivative on $[ar_{k_\ell},1]$, in particular this is true for $\xi^{-1/2}$ (this function is continuous and non-increasing on $[ar_{k_\ell},1]$). Let $\Xi$ be the antiderivative of $\xi^{-1/2}$ given by $\Xi(r)=\int_{[ar_{k_\ell},r]}\xi(u)^{-1/2}\e\Leb(u)$ for $r\in [ar_{k_\ell},1]$. \Cref{in1:inAllard2} directly implies that 
	\begin{equation*}
		r\mapsto r^{-1}\xi(r)e^{\tilde c\Xi(r)}
	\end{equation*}
	is non-decreasing on $[ar_{k_\ell},1]$. Thus, we have $\partial_r(r^{-1}\xi(r)e^{\tilde c\Xi(r)})\geq 0$ for a.e.\ $r\in [ar_{k_\ell},1]$. A straightforward calculation shows that this is equivalent to 
	\begin{equation*}
		\xi'(r)\geq-\tilde c\xi(r)^{1/2}+\frac{\xi(r)}{r}
	\end{equation*}
	for a.e.\ $r\in [ar_{k_\ell},1]$. Using this, it is easy to obtain
	\begin{equation*}
		r\frac{\e}{\e r}\left( \frac{\xi(r)^{1/2}}{r^{1/2}}+\tilde cr^{1/2} \right)\geq 0
	\end{equation*}
	for a.e.\ $r\in [ar_{k_\ell},1]$, which shows that $r\mapsto \xi(r)^{1/2}/r^{1/2}+\tilde cr^{1/2} $ is non-decreasing on $ [ar_{k_\ell},1]$. Thus (using also the previous estimate $\xi(u)\geq \lambda ar_{k_\ell}/4$ for $u\in[ar_{k_\ell},1]$)
	\begin{equation*}
		\frac{\xi(r)^{1/2}}{r^{1/2}}+\tilde cr^{1/2}\geq \frac{\xi(ar_{k_\ell})^{1/2}}{(ar_{k_\ell})^{1/2}}+\tilde cr_{k_\ell}^{1/2}\geq  \frac{\lambda^{1/2}}{2}+\tilde cr_{k_\ell}^{1/2}.
	\end{equation*}
	Hence, we can choose $\bar c\in(0,1/2)$ such that 
	\begin{equation*}
		\frac{\xi(r)}{r}\geq \left(  \frac{\lambda^{1/2}}{2}-\tilde cr^{1/2}  \right)^2\geq \frac{\lambda}{5}
	\end{equation*}
	for all $r\in [ar_{k_\ell},1]\cap [0,\bar c]$.
	Next, we estimate $|T|_\tau(\overline{B}_{2r}(x))/(2r)$. If $y\in\R^n\setminus \overline{B}_{2r}(x)$, then
	\begin{multline*}
		\Psi_\varepsilon\ast 1_{\overline{B}_r(x)}(y)=\int_{\overline{B}_r(x)}\Psi_\varepsilon(y-z)\e\Lebn(z)=\frac{c_\varepsilon}{(2\pi\varepsilon^2)^{n/2}}\int_{\overline{B}_r(x)}\psi(y-z)e^{-|y-z|^2/(2\varepsilon^2)}\e\Lebn(z)\\
		\leq \frac{c_\varepsilon}{(2\pi\varepsilon^2)^{n/2}}e^{-(ar_{k_\ell})^2/(2\varepsilon^2)}\Lebn(\overline{B}_r(x))\leq c_\varepsilon\varepsilon^{-n}e^{-1/(2\varepsilon)}\Lebn(\overline{B}_r(x)),
	\end{multline*}
	where we also used $\varepsilon<(ar_{k_\ell})^2$ in the last inequality. Note that $1_{\overline{B}_{2r}(x)}+(\Psi_\varepsilon\ast 1_{\overline{B}_r(x)})|_{\R^n\setminus \overline{B}_{2r}(x)}\geq \Psi_\varepsilon\ast 1_{\overline{B}_r(x)}$. This gives (\cref{aux2})
	\begin{multline*}
		|T|_\tau(\overline{B}_{2r}(x))\geq |T|_\tau( \Psi_\varepsilon\ast 1_{\overline{B}_r(x)})- |T|_\tau((\Psi_\varepsilon\ast 1_{\overline{B}_r(x)})|_{\R^n\setminus \overline{B}_{2r}(x)})\\
		\geq  ( \Psi_\varepsilon\ast |T|_\tau)(\overline{B}_r(x))-  c_\varepsilon\varepsilon^{-n}e^{-1/(2\varepsilon)}\Lebn(\overline{B}_r(x))|T|_\tau(\R^n)
	\end{multline*}
	implying (using also \cref{in1:6.3})
	\begin{equation*}
		\frac{|T|_\tau(\overline{B}_{2r}(x))}{2r}\geq\frac{\xi(r)}{2r}-c_\varepsilon\varepsilon^{-n}e^{-1/(2\varepsilon)}\Lebn(\overline{B}_r(x))(|T_0|_\tau(\R^n)+1)/(ar_{k_\ell})
	\end{equation*}
	for all $r\in [ar_{k_\ell},1]\cap [0,\bar c]$. Hence, taking $a=1/2$ and $c_2=c_2(t)=2\bar c\in(0,1)$ gives 
	\begin{equation*}
		\frac{|T|_\tau(\overline{B}_{r}(x))}{r}\geq \frac{\lambda}{11}
	\end{equation*}
	for all $r\in [r_{k_\ell},1]\cap [0, c_2]=2([ar_{k_\ell},1]\cap [0,\bar c])$ as desired.
\end{proof}
We recall Allard's rectifiability theorem (for the case $k=1$ and positive one-dimensional density a.e.).
\begin{thm}[{Allard's rectifiability theorem \cite[Sec.\ 5.5, Thm.\ 1]{A72}}]
	\label{AllardRect}
	Let $U\subset\R^n$ open and $V\in\mathcal V_1(U)$. Assume that $\|\delta V \|$ is a Radon measure and $\Theta^{*1}(\|  V\|,x)>0$ for $\|  V\|$-a.e.\ $x\in U$. Then $\Theta^{1}(\|  V\|,x)\in\R_+^*$ for $\|  V\|$-a.e.\ $x\in U$ and $V$ is $1$-rectifiable in the following sense \cite[Sec.\ 3.5]{A72}: there exist $\Ha$-measurable $A_i\subset U$ and numbers $w_i\in\R$ such that 
	\begin{itemize}
		\item for all $i$ and compact $K\subset U$ there exists a $(\Ha,1)$-rectifiable\footnote{By \cite[Sec.\ 2.8, Def.\ 4]{A72} we have $\Ha(S)<\infty$ and there exists a sequence $(P_i)\subset G(n,1)$ and $C^1$ maps $f_i:P_i\to\R^n$ with $P_if_i=\textup{id}$ and $\Ha\left( S\setminus\bigcup_if_i(P_i) \right)=0$.} set $S$ with $A_i\cap K\subset S$
		\item and $V=\sum_{i=1}^{\infty}w_iV(E_i,1)$.
	\end{itemize}
\end{thm}
We will also need the following statement,which corresponds to \cite[Prop.\ 5.6]{KT17}, see \cite[Sec.\ 4.18]{B78}.
\begin{prop}
	\label{Prop5.6}
	For each $i\in\mathbb N$ let $T_i=\theta_i\Ha\mres S_i\in\rectn$ be compactly supported. Assume that $\sup_i|T_i|_\tau(\R^n)<\infty$. Define $V_i=v(\tau(|\theta_i|),S_i)\in\M_+(\G_1(\R^n))\subset\V_1(\R^n)$ for all $i\in\mathbb N$. Further, assume that
	\begin{equation*}
		\liminf_i\int\frac{|\Psi_{\varepsilon_i}\ast\delta_\tau T_i|^2}{\Psi_{\varepsilon_i}\ast|T_i|_\tau+\varepsilon_i}\e\Lebn<\infty
	\end{equation*} 
	for some sequence $(\varepsilon_i)\subset(0,1)$ with $\varepsilon_i\to 0$. Then there exists a subsequence $(V_{i_m})$ with $V_{i_m}\ws V\in\V_1(\R^n)$ for $m\to\infty$ and $V$ admits a generalized curvature vector field $h_V$\footnote{$h_V$ is $\| V\|$-measurable with $\delta V(g)=-\int g\cdot h_V\e\| V\|$ for all $g\in C_c^1(\R^n;\R^n)$.} satisfying
	\begin{equation}
		\label[ineq]{in1:GenMean}
		\int|h_V|^2\phi\e\| V\|\leq\liminf_{m}\int\frac{|\Psi_{\varepsilon_{i_m}}\ast\delta_\tau T_{i_m}|^2\phi}{\Psi_{\varepsilon_{i_m}}\ast|  T_{i_m}|_\tau+\varepsilon_{i_m}}\e\Lebn
	\end{equation}
	for all $\phi\in\bigcup_j\mathcal F_j^+$. In particular, vector field $h_V$ is square-integrable with respect to $\| V\|$.
\end{prop}
\begin{proof}
	We can restrict to a subsequence $(i_m)$ with $V_{i_m}\ws V\in\M_+(\G_1(\R^n))$ for $m\to\infty$ (Banach\textendash Alaoglu theorem).
	Take $\phi\in\bigcup_j\mathcal F_j^+$ with $\phi>0$ in $\R^n$.
	As in the proof of \cite[Prop.\:5.6]{KT17}, we define a Hilbert space by
	\begin{equation*}
		H_\phi=\left\{ g\in L_{loc}^2(\| V\|;\R^n)\:\middle|\:|g|_{H_\phi}<\infty \right\},\qquad\textup{where $|\cdot|_{H_\phi}$ induced by}\qquad (g_1,g_2)_{H_\phi}=\int g_1\cdot g_2\phi^{-1}\e\|V\|.
	\end{equation*}
	Note that $C_c^\infty(\R^n;\R^n)$ is dense in $H_\phi$. Take $g\in C_c^\infty(\R^n;\R^n)$. Then $g\in\mathcal F_{j_0}^n$ for some $j_0\in\mathbb N$. Using $V_{i_m}\ws V$ and applying \cref{Prop5.5} (eventually $j_0\leq 2^{-1}\varepsilon_{i_m}^{-1/6}$ and $\varepsilon_{i_m}$ sufficiently small) yields the analogue of \cite[(5.51)]{KT17},
	\begin{equation*}
		\delta V(g)=\lim_m\delta_\tau T_{i_m}(g)=-\lim_m\int h_{\tau,\varepsilon_{i_m}}(T_{i_m})\cdot g\e|T_{i_m}|_\tau.
	\end{equation*}
	Using the Cauchy\textendash Schwarz inequality and \cref{in1:5.24} we get (as in \cite[(5.52)]{KT17})
	\begin{equation*}
		\delta V(g)=-\lim_m\int h_{\tau,\varepsilon_{i_m}}(T_{i_m})\cdot g\e|T_{i_m}|_\tau\leq\left( \liminf_m\int\frac{|\Psi_{\varepsilon_{i_m}}\ast\delta_\tau T_{i_m}|^2\phi }{\Psi_{\varepsilon_{i_m}}\ast|T_{i_m}|_\tau+\varepsilon_{i_m}}\e\Lebn \right)^{1/2}|g|_{H_\phi},
	\end{equation*}
	which is bounded by assumption and $\phi\leq 1$.
	As in the end of the proof of \cite[Prop.\ 5.6]{KT17}, we can apply the Hahn\textendash Banach theorem yielding a continuous and linear functional $\delta V:H_\phi\to\R$ (extended from $C_c^\infty(\R^n;\R^n)$). The Riesz representation theorem then implies the existence of a unique vector field $f\in H_\phi$ with $\delta V=(f,\cdot)_{H_\phi}$. Then 
	\begin{equation*}
		h_V=f\phi^{-1}\qquad\textup{and}\qquad |f|_{H_\phi}^2\leq \liminf_m\int\frac{|\Psi_{\varepsilon_{i_m}}\ast\delta_\tau T_{i_m}|^2\phi }{\Psi_{\varepsilon_{i_m}}\ast|T_{i_m}|_\tau+\varepsilon_{i_m}}\e\Lebn.
	\end{equation*}
	The last inequality is equivalent to \cref{in1:GenMean}. For $\phi\equiv 0$ \cref{in1:GenMean} is trivially satisfied. For the last statement we can take $\phi\equiv 1$.
\end{proof}
\begin{proof}[Proof of \cref{rectifiability}]
	For a.e.\ $t\in\R_+$ we have $c_2=c_2(t)$ and a subsequence $\ell\mapsto j_{k_\ell}$ (depending on $t$) as in \cref{LowerBound}.	
	Now fix any such $t\in\R_+$.
	\Cref{Prop5.6} yields 
	\begin{equation*}
		V_t^{j_{k_\ell}}\ws V_t\in \M_+(\mathcal G_1(\R^n))
	\end{equation*}
	for $\ell\to\infty$ up to a subsequence (for which we use the same label) and the existence of $h_{V_t}\in L^2(\| V_t\|;\R^n)$ (the $L^2$ regularity follows from \cref{in1:GenMean} with $\phi\equiv 1$) such that $\delta V_t(g)=-\int g\cdot h_{V_t}\e\| V_t \|$ for all $g\in C_c^1(\R^n;\R^n)$. Since $\| V_t \|$ is compactly supported, we actually have $h_{V_t}\in L^1(\| V_t\|;\R^n)$, hence $\delta V_t$ is a vector-valued Radon measure (in particular, its total variation is a Radon measure). Therefore, it remains to show that
	\begin{equation*}
		\Theta^{*1}(\| V_t \|=\mu_t,x)>0
	\end{equation*}
	for $\| V_t \|$-a.e.\ $x\in\R^n$. Let $x\in\spt(\| V_t\|)$. First, we claim that for each $\ell\in\mathbb{N}$ there is $x_\ell\in X_t^{j_{k_{\ell}}}$ such that $x_\ell\to x$ (maybe after restricting to another subsequence). Assume that the converse is true, there are $\ell_0\in\mathbb N$ and $\delta\in(0,1)$ with 
	\begin{equation*}
		X_t^{j_{k_{\ell}}}\cap B_\delta(x)=\emptyset
	\end{equation*}
	for all $\ell\geq \ell_0$. Since $x\in \spt(\| V_t\|)$, $B_\delta(x)$ open, and $V_t^{j_{k_{\ell}}}\ws V_t$,
	\begin{multline*}
		0<\| V_t\|(B_\delta(x))\leq\liminf_\ell\|  V_t^{j_{k_{\ell}}}\|(B_\delta(x))\\
		=\liminf_\ell\|  V_t^{j_{k_{\ell}}}\|((S_t^{j_{k_{\ell}}}\setminus X_t^{j_{k_{\ell}}})\cap B_\delta(x))+\|  V_t^{j_{k_{\ell}}}\|(X_t^{j_{k_{\ell}}}\cap B_\delta(x))=0,
	\end{multline*}
	where the last equality follows from \cref{Lemma1}\footnote{We can assume that the subsequence is such that $-\Delta_{j_{k_{\ell}}}|T_t^{j_{k_{\ell}}}|_\tau(\R^n)\leq c(t)\Delta t_{j_{k_{\ell}}}$ (\cref{Prop6.4}). Further, the beginning of the proof of \cref{Prop6.1} gives $\Delta t_{j_{k_{\ell}}}\leq \varepsilon_{j_{k_{\ell}}}^{c_1}<j_{k_{\ell}}^{-6c_1}=j_{k_{\ell}}^{-6(3n+20)}$.} and $X_t^{j_{k_{\ell}}}\cap B_\delta(x)=\emptyset$ for $\ell\geq \ell_0$. This yields a contradiction. Now let $r\leq c_2(t)$ be arbitrary (recall \cref{LowerBound}). Since eventually $\overline B_{r/2}(x_\ell)\subset\overline B_r(x)$, we obtain
	\begin{equation*}
		\frac{\| V_t\|(\overline B_r(x)) }{r}\geq\limsup_{\ell}\frac{\| V_t^{j_{k_{\ell}}}\|(\overline B_r(x)) }{r}\geq \limsup_{\ell}\frac{\| V_t^{j_{k_{\ell}}}\|(\overline B_{r/2}(x_\ell)) }{r}\geq\frac{2\lambda}{11}
	\end{equation*}
	by $\| V_t^{j_{k_{\ell}}}\|\ws \| V_t \|$ and \cref{LowerBound}. Hence, we get 
	\begin{equation*}
		\Theta^{*1}(\| V_t \|,x)\geq \frac{\lambda}{11}
	\end{equation*}
	for all $x\in\spt(\| V_t \|)$, which shows the applicability of \cref{AllardRect}.
	By \cite[Sec.\ 3.5, Thm.\ 1]{A72} and \cite[Sec.\ 2.8, Thm.\ 5]{A72} we have
	\begin{equation}
		\label{9.22}
		\int\psi(x,P)\e V_t(x,P)=\int\psi(x,\textup{Tan}^1(\| V_t \|,x))\Theta^1(\| V_t \|,x)\e \Ha(x)=\int\psi(x,\textup{Tan}^1(\| V_t \|,x))\e \| V_t\|(x)
	\end{equation}
	for all $\psi\in C_c(\mathcal G_1(\R^n))$. 
	This indicates that the varifold $V_t$ is uniquely characterized by its weight measure $\| V_t \|=\mu_t$ (recall that $\mu_t$ was obtained without passing to a subsequence of $(j_k)$).
	\Cref{LowerBound} can actually be applied to any subsequence of $(j_k)$ yielding another subsequence $(j_{k_\ell})$ which can be further refined as in this proof.
	In particular, every subsequence of $(j_k)$ admits a subsequence  $(j_{k_\ell})$ such that $V_t$ equals the limit (in the weak-$*$ sense) of the corresponding induced varifolds.
	Therefore, the whole sequence $(V_t^{j_{k}})$ satisfies $V_t^{j_{k}}\ws V_t$.
	Finally, \cref{in1:timeIntmean} follows from the beginning of the proof of \cref{LowerBound}. Indeed, for a.e.\ $t\in\R_+$ there is a subsequence $N_t:\mathbb N\to \{ j_k\:|\: k\in\mathbb N \}$ with 
	\begin{equation*}
		\lim_\ell\int\frac{ |\Psi_{\varepsilon_{N_t(\ell)}}\ast\delta_\tau T_t^{N_t(\ell)}|^2 }{\Psi_{\varepsilon_{N_t(\ell)}}\ast|T_t^{N_t(\ell)}|_\tau +\varepsilon_{N_t(\ell)}}\e\mathcal L^n=	\liminf_k\int\frac{ |\Psi_{\varepsilon_{j_{k}}}\ast\delta_\tau T_t^{j_{k}}|^2 }{\Psi_{\varepsilon_{j_{k}}}\ast|T_t^{j_{k}}|_\tau +\varepsilon_{j_{k}}}\e\mathcal L^n
	\end{equation*}
	and such that \cref{in1:GenMean} is satisfied with $V=V_t$ and $m\mapsto i_m$ replaced by $\ell\mapsto N_t(\ell)$.
	This gives (using $\phi\equiv 1$ and recalling \cref{in1:6.19}).
	\begin{equation*}
		\int_{[0,R]}\int |h_{V_t}|^2\e\| V_t\|\e\Leb(t)\leq \int_{[0,R]}\liminf_\ell \int\frac{ |\Psi_{\varepsilon_{N_t(\ell)}}\ast\delta_\tau T_t^{N_t(\ell)}|^2 }{\Psi_{\varepsilon_{N_t(\ell)}}\ast|T_t^{N_t(\ell)}|_\tau +\varepsilon_{N_t(\ell)}}\e\mathcal L^n\e\Leb(t)<\infty
	\end{equation*}
	for all $R\in\R_+$.
\end{proof}
\section{Motion law of geometric flow $(V_t)$}
\label{MotionLawSection}
We are now ready to show a Brakke-type inequality for $(V_t)$.
We will follow the proof of \cite[Thm.\ 9.3]{KT17}, see also \cite[Sec.\ 4.29]{B78}.
Note that the $V_t$ cannot be expected to be integral due to the definition of $V_t^{j_k}$ in \cref{rectifiability} for a.e.\ $t\in\R_+$, which explains the presence of the term $P^\perp$. In particular, Brakke's perpendicularity theorem \cite[Sec.\ 5]{B78} is not applicable.
\begin{proof}[Proof of \cref{MotionLaw}]
	First, take $\phi\in C_c^\infty(\R^n;\R_+)$ independent of $t$. 
	We can assume $\phi<1$ in $\R^n$ by rescaling the inequality in the statement.
	Let $i\in\mathbb N$ such that $\hat\phi=\phi+1/i<1$.
	We have $\min\hat\phi\in\R_+^*$. Hence, there is $p\in\mathbb N$ with $\hat\phi\in\mathcal F_p^+$.
	Fix $t_1,t_2\in\R_+$ with $t_1<t_2$. Let $k$ sufficiently large such that $j_k>\max\{ p,t_2 \}$ and $\Dt_{j_k}<t_2-t_1$.
	There exist $N_1=N_1(k),N_2=N_2(k)\in\mathbb N$ with $t_1\in ( (N_1-2)\Dt_{j_k},(N_1-1)\Dt_{j_k} ]$ and $t_2\in ( (N_2-1)\Dt_{j_k},N_2\Dt_{j_k} ]$.
	By \cref{in1:6.5} we have (cf.\ \cite[(9.6)]{KT17})
	\begin{equation}
		\label[ineq]{in:9.6}
		|T^{j_k,\ell}|_\tau(\hat\phi)-|T^{j_k,\ell-1}|_\tau(\hat\phi)\leq\Dt_{j_k}\left( \delta_\tau(T^{j_k,\ell},\hat\phi)(h_{\tau,\varepsilon_{j_k}}(T^{j_k,\ell}))+\varepsilon_{j_k}^{1/8} \right)
	\end{equation}
	for all $\ell=1,\ldots,j_k2^{p_{j_k}}$. 
	Summation over $\ell=N_1,\ldots,N_2$ yields (see \cite[(9.7)]{KT17})
	\begin{equation}
		\label[ineq]{in1:9.7}
		|T^{j_k,N_2}|_\tau(\hat\phi)-|T^{j_k,N_1-1}|_\tau(\hat\phi)\leq\sum_{\ell=N_1}^{N_2}\Dt_{j_k}\left( \delta_\tau(T^{j_k,\ell},\hat\phi)(h_{\tau,\varepsilon_{j_k}}(T^{j_k,\ell}))+\varepsilon_{j_k}^{1/8} \right).
	\end{equation}
	Using the definition of $\hat\phi $ and \cref{in1:6.3}, we can estimate (\cite[(9.8)]{KT17})
	\begin{multline*}
		|T^{j_k,N_2}|_\tau(\hat\phi)-|T^{j_k,N_1-1}|_\tau(\hat\phi)\geq|T^{j_k,N_2}|_\tau(\phi)-|T^{j_k,N_1-1}|_\tau(\phi)-\frac1i|T^{j_k,N_1-1}|_\tau(\R^n)\\
		\geq |T^{j_k,N_2}|_\tau(\phi)-|T^{j_k,N_1-1}|_\tau(\phi)-\frac1i( |T_0|_\tau(\R^n)+\varepsilon_{j_k}^{1/8}(N_1-1)\Dt_{j_k} ),
	\end{multline*}
	hence
	\begin{multline}
		\limsup_k|T^{j_k,N_2}|_\tau(\hat\phi)-|T^{j_k,N_1-1}|_\tau(\hat\phi)\geq\limsup_k |T^{j_k,N_2}|_\tau(\phi)-|T^{j_k,N_1-1}|_\tau(\phi)-\frac1i |T_0|_\tau(\R^n)\\
		=\| V_{t_2} \|(\phi)-\| V_{t_1} \|(\phi)-\frac1i |T_0|_\tau(\R^n).\label[ineq]{in1:9.9}
	\end{multline}
	Now abbreviate $T=\theta\Ha\mres S=T^{j_k,\ell},h=h_{\tau,\varepsilon_{j_k}}(T^{j_k,\ell})$, and $\varepsilon=\varepsilon_{j_k}$.
	Recall \cref{TauVariation,phiWeightedTauVariation}\footnote{Note that $\delta_\tau T(\hat\phi h)=\int\hat\phi\textup{div}_S(h)\e|T|_\tau+\int h^\intercal P\nabla\hat\phi\e|T|_\tau=\delta_\tau(T,\hat\phi)(h)-\int\nabla\hat\phi^\intercal  h\e|T|_\tau+\int h^\intercal P \nabla\hat\phi\e |T|_\tau=\delta_\tau(T,\hat\phi)(h)-\int h^\intercal P^\perp \nabla\hat\phi\e|T|_\tau$, cf.\ calculation in \cite[(2.5)]{KT17}.},
	\begin{equation}
		\label{9.10}
		\delta_\tau(T,\hat\phi)(h)=\delta_\tau T(\hat\phi h)+\int h^\intercal P^\perp \nabla\hat\phi\e|T|_\tau,
	\end{equation}
	where $P=P(x)$ denotes the projection onto $T_xS$ for $\Ha$-a.e.\ $x\in S$.
	Applying the Cauchy\textendash Schwarz inequality to $\int h^\intercal P^\perp \nabla\hat\phi\e|T|_\tau$ yields (cf.\ \cite[(9.13)]{KT17})
	\begin{equation*}
		\int h^\intercal P^\perp \nabla\hat\phi\e|T|_\tau\leq \left( \int \frac{|\nabla\hat\phi|^2}{\hat\phi}\e|T|_\tau \right)^{1/2}\left( \int \hat\phi|h|^2\e|T|_\tau \right)^{1/2}.
	\end{equation*}
	By definition of $\hat\phi$, \cref{in1:6.27}, and $\hat\phi\in\mathcal F_p^+$ we get
	\begin{equation*}
		\frac{|\nabla\hat\phi|^2}{\hat\phi}\leq \frac{|\nabla\phi|^2}{\phi}\leq 2||\nabla^2\phi|_2|_\infty.
	\end{equation*}
	This gives
	\begin{equation*}
		\left( \int \frac{|\nabla\hat\phi|^2}{\hat\phi}\e|T|_\tau \right)^{1/2}\leq \sqrt{2}||\nabla^2\phi|_2|_\infty^{1/2}|T|_\tau(\R^n)^{1/2}.
	\end{equation*}
	Abbreviating $\hat I=\int\frac{\hat\phi|\Psi_{\varepsilon}\ast\delta_\tau T|^2}{\Psi_{\varepsilon}\ast|T|_\tau+\varepsilon}\e\Lebn$, we obtain (using \cref{in1:5.24})
	\begin{equation*}
		\int h^\intercal P^\perp \nabla\hat\phi\e|T|_\tau\leq c(\nabla^2\phi)|T|_\tau(\R^n)^{1/2}\left( (1+\varepsilon^{1/4})\hat I+\varepsilon^{1/4} \right)^{1/2}.
	\end{equation*}
	\Cref{in1:5.23} yields $|\delta_\tau T(\hat\phi h)+\hat I|\leq \varepsilon^{1/4}(\hat I+1)$, hence (revoking the abbreviations $T=T^{j_k,\ell},h=h_{\tau,\varepsilon_{j_k}}(T^{j_k,\ell})$, and $\varepsilon=\varepsilon_{j_k}$)
	\begin{equation*}
		\delta_\tau(T^{j_k,\ell},\hat\phi)(h_{\tau,\varepsilon_{j_k}}(T^{j_k,\ell}))\leq \varepsilon_{j_k}^{1/4}(\hat I+1)-\hat I+c(\nabla^2\phi)|T^{j_k,\ell}|_\tau(\R^n)\left( (1+\varepsilon_{j_k}^{1/4})\hat I+\varepsilon_{j_k}^{1/4} \right)^{1/2}.
	\end{equation*}
	Now it is not difficult to see that this implies that $\sup_{\ell=N_1-1,\ldots,N_2+1}\delta_\tau(T^{j_k,\ell},\hat\phi)(h_{\tau,\varepsilon_{j_k}}(T^{j_k,\ell}))$ is bounded by some constant $A=A(|T_0|_\tau(\R^n),t_2,\nabla^2\phi)\in\R_+$ (using also \cref{in1:6.3,in1:6.7}), see \cite[(9.14)]{KT17}\footnote{Indeed, we have $ \delta_\tau(T^{j_k,\ell},\hat\phi)(h_{\tau,\varepsilon_{j_k}}(T^{j_k,\ell}))\leq \varepsilon_{j_k}^{1/4}(\hat I+1)-\hat I+c(|T_0|_\tau(\R^n),\nabla^2\phi)( \hat I^{1/2}+\varepsilon_{j_k}^{1/8}\hat I^{1/2}+\varepsilon_{j_k}^{1/8} )$ and the functions (depending on $x$) $\varepsilon_{j_k}^{1/4}x-x/3,c(|T_0|_\tau(\R^n),\nabla^2\phi)x^{1/2}-x/3,$ and $c(|T_0|_\tau(\R^n),\nabla^2\phi)\varepsilon_{j_k}^{1/8}x^{1/2}-x/3$ are bounded on any compact interval in $\R_+$ for $k$ sufficiently large.}. This gives (see \cite[(9.15)]{KT17})
	\begin{multline}
		\limsup_k \sum_{\ell=N_1}^{N_2}\Dt_{j_k} \delta_\tau(T^{j_k,\ell},\hat\phi)(h_{\tau,\varepsilon_{j_k}}(T^{j_k,\ell}))\leq\limsup_k\int_{[t_1,t_2+\Dt_{j_k}]}  \delta_\tau(T_t^{j_k},\hat\phi)(h_{\tau,\varepsilon_{j_k}}(T_t^{j_k}))\e\Leb(t)\\
		=\limsup_k\int_{[t_1,t_2]}  \delta_\tau(T_t^{j_k},\hat\phi)(h_{\tau,\varepsilon_{j_k}}(T_t^{j_k}))\e\Leb(t)\leq \int_{[t_1,t_2]}  \limsup_k\delta_\tau(T_t^{j_k},\hat\phi)(h_{\tau,\varepsilon_{j_k}}(T_t^{j_k}))\e\Leb(t),\label[ineq]{in1:9.15}
	\end{multline}
	where (reversed) Fatou's lemma and the bound $A$ were used in the last inequality. Next, we use the same calculations as in \cite[(9.16) \& (9.17)]{KT17} applied to the integrand in the latter integral. Fix any $t\in[t_1,t_2]\setminus \mathcal N$, where $\Leb$-zero set $\mathcal N\subset [t_1,t_2]$ is chosen such that the subsequent conclusions hold. Let $(j_{k_m})\subset (j_k)$ be a subsequence such that $\lim_m\delta_\tau(T_t^{j_{k_m}},\hat\phi)(h_{\tau,\varepsilon_{j_{k_m}}}(T_t^{j_{k_m}}))=\limsup_k\delta_\tau(T_t^{j_k},\hat\phi)(h_{\tau,\varepsilon_{j_k}}(T_t^{j_k}))$. By \cref{9.10} we have
	\begin{multline*}
		\lim_m -\delta_\tau T_t^{j_{k_m}}(\hat\phi h_{\tau,\varepsilon_{j_{k_m}}}(T_t^{j_{k_m}}))-\int h_{\tau,\varepsilon_{j_{k_m}}}(T_t^{j_{k_m}})^\intercal (P_t^{j_{k_m}})^\perp\nabla\hat\phi \e|T_t^{j_{k_m}}|_\tau\\
		=\liminf_k-\delta_\tau T_t^{j_k}(\hat\phi h_{\tau,\varepsilon_{j_k}}(T_t^{j_k}))-\int h_{\tau,\varepsilon_{j_k}}(T_t^{j_k})^\intercal (P_t^{j_{k_m}})^\perp\nabla\hat\phi \e|T_t^{j_k}|_\tau.
	\end{multline*}
	Again, by the two estimates which were previously used to bound \cref{9.10}, the right-hand side is bounded from above by $\liminf_k 2\hat I_t^{j_k}+A$, and the left-hand side can be estimated from below by $\limsup_m \hat I_t^{j_{k_m}}-A$\footnote{Formally, $-\delta_\tau-\int=-(\delta_\tau+\hat I+\int)+\hat I\geq -A+\hat I$.}. This gives
	\begin{equation*}
		\limsup_m \hat I_t^{j_{k_m}}\leq \liminf_k 2\hat I_t^{j_k}+2A.
	\end{equation*}
	Define $f(t)=\liminf_k 2\hat I_t^{j_k}+2A$. Note that $f(t)<\infty$ (\cref{Prop6.4} and Fatou's lemma). Recall that $\hat\phi\geq 1/i$, thus (cf.\ \cite[(9.19)]{KT17})
	\begin{equation}
		\label[ineq]{in1:9.19}
		\limsup_m\int\frac{|\Psi_{\varepsilon_{j_{k_m}}}\ast\delta_\tau T_t^{j_{k_m}}|^2}{\Psi_{\varepsilon_{j_{k_m}}}\ast|T_t^{j_{k_m}}|_\tau+\varepsilon_{j_{k_m}}}\e\Lebn \leq \limsup_m i\hat I_t^{j_{k_m}}\leq if(t)<\infty.
	\end{equation}
	By $|\delta_\tau T(\hat\phi h)+\hat I|\leq \varepsilon^{1/4}(\hat I+1)$, the above, $f(t)<\infty$, and \cref{Prop5.6} we have
	\begin{equation*}
		\limsup_m\delta_\tau T_t^{j_{k_m}}(\hat\phi h_{\tau,\varepsilon_{j_{k_m}}}(T_t^{j_{k_m}}))\leq \limsup_m -\hat I_t^{j_{k_m}}= -\liminf_m \hat I_t^{j_{k_m}}\leq-\int |h_{ V_t}|^2\hat\phi\e\|  V_t \|.
	\end{equation*}
	Using \cref{9.10} this gives (see \cite[(9.21)]{KT17})
	\begin{multline}
		\limsup_k\delta_\tau(T_t^{j_k},\hat\phi)(h_{\tau,\varepsilon_{j_k}}(T_t^{j_k}))=\lim_m\delta_\tau(T_t^{j_{k_m}},\hat\phi)(h_{\tau,\varepsilon_{j_{k_m}}}(T_t^{j_{k_m}}))\\
		\leq-\int |h_{ V_t}|^2\hat\phi\e\|  V_t \|+\limsup_m\int h_{\tau,\varepsilon_{j_{k_m}}}(T_t^{j_{k_m}})^\intercal (P_t^{j_{k_m}})^\perp\nabla\hat\phi \e|T_t^{j_{k_m}}|_\tau.\label[ineq]{in1:9.21}
	\end{multline}
	Let $\varepsilon>0$ be arbitrary.
	Recall that (\cref{9.22})
	\begin{equation}
		\label{9.22recall}
		\int\psi(x,P)\e V_t(x,P)=\int\psi(x,\textup{Tan}^1(\| V_t \|,x))\Theta^1(\| V_t \|,x)\e \Ha(x)=\int\psi(x,\textup{Tan}^1(\| V_t \|,x))\e \| V_t\|(x)
	\end{equation}
	for all $\psi\in C_c(\mathcal G_1(\R^n))$. 
	By density there is $g\in C_c^\infty(\R^n;\R^n)$ with
	\begin{equation}
		\label[ineq]{in1:9.23}
		\int| \textup{Tan}^1(\| V_t\|,x)^\perp\nabla\hat\phi(x)-g(x) |^2\e\| V_t\|(x)<\varepsilon^2.
	\end{equation}
	Write \cite[(9.24)]{KT17}
	\begin{multline}
		\label{9.24}
		\int h_{\tau,\varepsilon_{j_{k_m}}}(T_t^{j_{k_m}})^\intercal (P_t^{j_{k_m}})^\perp\nabla\hat\phi \e|T_t^{j_{k_m}}|_\tau=\underbrace{\int ((P_t^{j_{k_m}})^\perp\nabla\hat\phi -g)\cdot h_{\tau,\varepsilon_{j_{k_m}}}(T_t^{j_{k_m}})\e|T_t^{j_{k_m}}|_\tau}_{E_1}\\
		+\underbrace{\int g\cdot h_{\tau,\varepsilon_{j_{k_m}}}(T_t^{j_{k_m}})\e|T_t^{j_{k_m}}|_\tau+\delta_\tau T_t^{j_{k_m}}(g)}_{E_2}+\underbrace{-\delta_\tau T_t^{j_{k_m}}(g)+\delta V_t(g)}_{E_3}\\
		+\underbrace{\int h_{V_t}\cdot(g-\textup{Tan}^1(\| V_t\|,\cdot)^\perp\nabla\hat\phi)\e\| V_t\|}_{E_4}+\int h_{V_t}(x)\cdot (P^\perp\nabla\hat\phi(x))\e V_t(x,P).
	\end{multline}
	Using the Cauch\textendash Schwarz inequality together with \cref{9.22recall,in1:9.23,in1:9.19,in1:GenMean}, we get \cite[(9.27)]{KT17}
	\begin{equation}
		\label[ineq]{in1:9.27}
		|E_1|\leq \left(  \int |(P_t^{j_{k_m}})^\perp\nabla\hat\phi -g|^2\e|T_t^{j_{k_m}}|_\tau\right)^{1/2}\left( \int |h_{\tau,\varepsilon_{j_{k_m}}}(T_t^{j_{k_m}})|^2\e|T_t^{j_{k_m}}|_\tau \right)^{1/2}\leq\varepsilon (if(t))^{1/2}.
	\end{equation}
	Application of \cref{Prop5.5,in1:9.19} yields
	\begin{equation}
		\label{9.28}
		\lim_m|E_2|=0.
	\end{equation}
	Property $v(\tau(|\theta_t^{j_{k_m}}|),S_t^{j_{k_m}})\ws V_t$ gives
	\begin{equation}
		\label{9.29}
		\lim_m|E_3|=0.
	\end{equation}
	Again, invoking the Cauchy\textendash Schwarz inequality together with \cref{in1:GenMean,in1:9.19,in1:9.23}, we obtain \cite[(9.30)]{KT17}
	\begin{multline}
		\label[ineq]{in1:9.30}
		|E_4|\leq \int |h_{V_t}||g-\textup{Tan}^1(\| V_t\|,\cdot)^\perp\nabla\hat\phi|\e\| V_t\|\leq \left( \int |h_{V_t}|^2\e\| V_t\| \right)^{1/2}\left( \int |g-\textup{Tan}^1(\| V_t\|,\cdot)^\perp\nabla\hat\phi|^2\e\| V_t\| \right)^{1/2}\\
		\leq\varepsilon (if(t))^{1/2}.
	\end{multline}
	In summary, \cref{in1:9.27}, \cref{9.28,9.29}, and \cref{in1:9.30} give (letting $\varepsilon\to 0$) \cite[(9.32)]{KT17}
	\begin{equation}
		\label[ineq]{in1:9.32}
		\limsup_m \int h_{\tau,\varepsilon_{j_{k_m}}}(T_t^{j_{k_m}})^\intercal (P_t^{j_{k_m}})^\perp\nabla\hat\phi \e|T_t^{j_{k_m}}|_\tau\leq \int h_{V_t}(x)\cdot (P^\perp\nabla\phi(x))\e V_t(x,P),
	\end{equation}
	where we also used $\nabla\hat\phi=\nabla\phi$.
	Using \cref{in1:9.32} in \cref{in1:9.21} and $\phi\leq\hat \phi$ gives \cite[(9.33)]{KT17}
	\begin{equation}
		\limsup_k\delta_\tau(T_t^{j_k},\hat\phi)(h_{\tau,\varepsilon_{j_k}}(T_t^{j_k}))\leq-\int |h_{ V_t}|^2\phi\e\|  V_t \|+\int h_{V_t}(x)\cdot (P^\perp\nabla\phi(x))\e V_t(x,P).\label[ineq]{in1:9.33}
	\end{equation}
	\Cref{in1:9.9,in1:9.7,in1:9.15,in1:9.33} imply
	\begin{multline*}
		\| V_{t_2} \|(\phi)-\| V_{t_1} \|(\phi)-\frac1i |T_0|_\tau(\R^n)\leq\limsup_k |T^{j_k,N_2}|_\tau(\phi)-|T^{j_k,N_1-1}|_\tau(\phi)-\frac{1}{i}|T_0|_\tau(\R^n)\\
		\leq \limsup_k\sum_{\ell=N_1}^{N_2}\Delta t_{j_k}(\delta_\tau(T^{j_k,\ell},\hat\phi)(h_{\tau,\varepsilon_{j_k}}(T^{j_k,\ell}))+\varepsilon_{j_k}^{1/8})\leq\int_{[t_1,t_2]}\limsup_k \delta_\tau(T_t^{j_k},\hat\phi)(h_{\tau,\varepsilon_{j_k}}(T_t^{j_k}))\e\Leb(t)\\
		\leq\int_{[t_1,t_2]} -\int |h_{V_t}|^2\phi\e\| V_t\|+\int h_{V_t}(x)\cdot (P^\perp\nabla\phi(x))\e V_t(x,P)\e\Leb(t). 
	\end{multline*}
	Letting $i\to\infty$, we obtain
	\begin{equation*}
		\| V_{t_2} \|(\phi)-\| V_{t_1} \|(\phi) \leq-\int_{[t_1,t_2]}\int |h_{V_t}|^2\phi\e\| V_t\|\e\Leb(t)+\int_{[t_1,t_2]}\int h_{V_t}(x)\cdot (P^\perp\nabla\phi(x))\e V_t(x,P)\e\Leb(t),
	\end{equation*}
	which shows the statement for time-independent $\phi\in C_c^\infty(\R^n;\R_+)$.
	Now assume that $\phi=\phi(x,t)\in C_c^\infty(\R^n\times\R_+;\R_+)$.
	Again, we can assume that $\phi\in(0,1)$ and define $\tilde\phi=\phi+1/i<1$ for sufficiently large $i\in\mathbb N$.
	Then, as before, there is $q\in\mathbb N $ such that $\tilde\phi\in\mathcal F_q^+$ for all $t\in\R_+$. 
	We may choose $k$ sufficiently large such that $j_k>\max\{ q,t_2 \}$.
	Application of \cref{in1:6.5} gives (cf.\ \cref{in:9.6})
	\begin{multline*}
		|T^{j_k,\ell}|_\tau(\tilde\phi(.,\ell\Delta t_{j_k}))-|T^{j_k,\ell-1}|_\tau(\tilde\phi(.,(\ell-1)\Delta t_{j_k}))\\
		=|T^{j_k,\ell}|_\tau(\tilde\phi(.,\ell\Delta t_{j_k}))-|T^{j_k,\ell-1}|_\tau(\tilde\phi(.,\ell\Delta t_{j_k}))+|T^{j_k,\ell-1}|_\tau(\tilde\phi(.,\ell\Delta t_{j_k}))-|T^{j_k,\ell-1}|_\tau(\tilde\phi(.,(\ell-1)\Delta t_{j_k}))\\
		\leq \Delta t_{j_k}(\delta_\tau (T^{j_k,\ell},\tilde\phi(.,\ell\Delta t_{j_k}))(h_{\tau,\varepsilon_{j_k}}(T^{j_k,\ell}))+\varepsilon_{j_k}^{1/8})+|T^{j_k,\ell-1}|_\tau(\phi(.,\ell\Delta t_{j_k})-\phi(.,(\ell-1)\Delta t_{j_k}))
	\end{multline*}
	for all $\ell=1,\ldots,j_k2^{p_{j_k}}$.
	As in \cref{in1:9.7}, we take the telescope sum,
	\begin{multline*}
		|T^{j_k,N_2}|_\tau(\tilde\phi(.,N_2\Delta t_{j_k}))-|T^{j_k,N_1-1}|_\tau(\tilde\phi(.,(N_1-1)\Delta t_{j_k}))\\
		\leq \sum_{\ell=N_1}^{N_2}\Delta t_{j_k}(\delta_\tau (T^{j_k,\ell},\tilde\phi(.,\ell\Delta t_{j_k}))(h_{\tau,\varepsilon_{j_k}}(T^{j_k,\ell}))+\varepsilon_{j_k}^{1/8})+\sum_{r=N_1}^{N_2}|T^{j_k,r-1}|_\tau(\phi(.,r\Delta t_{j_k})-\phi(.,(r-1)\Delta t_{j_k})).
	\end{multline*}
	Applying $\limsup_k$ to the last term yields (applying the calculation in \cite[(9.36)]{KT17})
	\begin{equation*}
		\limsup_k \sum_{r=N_1}^{N_2}|T^{j_k,r-1}|_\tau(\phi(.,r\Delta t_{j_k})-\phi(.,(r-1)\Delta t_{j_k}))\leq\int_{[t_1,t_2]}\| V_t\|(\partial_t\phi(.,t))\e\Leb(t).
	\end{equation*}
	Next, for $\ell=0,1,\ldots,j_k2^{p_{j_k}}$ and $t\in[0,j_k]$, we set
	\begin{equation*}
		\tilde\phi_t^{j_k}=\tilde\phi(.,\ell\Delta t_{j_k})\qquad\textup{if}\qquad t\in( (\ell-1)\Delta t_{j_k},\ell\Delta t_{j_k} ].
	\end{equation*}
	Then \cref{in1:9.15} becomes
	\begin{equation*}
		\limsup_k\sum_{\ell=N_1}^{N_2}\Delta t_{j_k}(\delta_\tau (T^{j_k,\ell},\tilde\phi(.\ell\Delta t_{j_k}))(h_{\tau,\varepsilon_{j_k}}(T^{j_k,\ell}))+\varepsilon_{j_k}^{1/8})\leq\int_{[t_1,t_2]}\limsup_k \delta_\tau (T_t^{j_k},\tilde\phi_t^{j_k})(h_{\tau,\varepsilon_{j_k}}(T_t^{j_k}))\e\Leb(t).
	\end{equation*}
	Exactly the same argument as on \cite[p.\ 129]{KT17} finishes the proof.
\end{proof}
 \section*{Acknowledgements}
J.L.’s research was supported by the JSPS Postdoctoral Fellowship for Research in Japan and the KAKENHI Grant-in-Aid for Scientific Research (Grant Number JP24KF0215) awarded by the Japan Society for the Promotion of Science. Y.T.’s was partially supported by the KAKENHI Grant-in-Aid for Scientific Research (Grant Number 23H00085) awarded by the Japan Society for the Promotion of Science.

\printbibliography
\end{document}